\documentclass{article} 
\usepackage{amssymb}
\usepackage{amsmath}
\usepackage{amsthm} 
\usepackage[a4paper]{geometry}
\usepackage{pgfplots}
\pgfplotsset{compat=1.16}
\usepackage[round]{natbib}
\usepackage[colorlinks=true, linkcolor=blue, breaklinks=true,linktoc=page]{hyperref}

\usepackage{amssymb}
\usepackage{amsmath,amscd}
\usepackage{amsthm}
\usepackage{cancel}     
\usepackage{centernot}  
\usepackage{graphicx}   
\usepackage{soul}       
\usepackage{color}      
\usepackage{enumitem}   
\usepackage{commath}    
\usepackage{algorithm}  
\usepackage{subcaption} 
\usepackage{booktabs}   
\usepackage{mathrsfs}

\newtheorem{thm}{Theorem}[section]
\newtheorem{lemma}[thm]{Lemma}
\newtheorem{cor}[thm]{Corollary}
\newtheorem{prop}[thm]{Proposition}

\newtheorem*{thm*}{Theorem}
\newtheorem*{lemma*}{Lemma}
\newtheorem*{cor*}{Corollary}
\newtheorem*{prop*}{Proposition}
\newtheorem*{conjecture*}{Conjecture}

\theoremstyle{definition}

\newtheorem*{defn*}{Definition}
\theoremstyle{definition}

\theoremstyle{definition}
\newtheorem{ex}{Example}
\theoremstyle{remark}
\newtheorem*{ex*}{Example}
\theoremstyle{definition}
\newtheorem{assm}{Assumption}
\theoremstyle{definition}
\newtheorem*{assm*}{Assumption}
\theoremstyle{definition}

\theoremstyle{definition}
\newtheorem*{condn*}{Condition}
\theoremstyle{remark}
\newtheorem{remark}[thm]{Remark}
\theoremstyle{remark}
\newtheorem*{remark*}{Remark}

\DeclareFontFamily{U}{mathx}{\hyphenchar\font45}
\DeclareFontShape{U}{mathx}{m}{n}{
      <5> <6> <7> <8> <9> <10> gen * mathx
      <10.95> mathx10 <12> <14.4> <17.28> <20.74> <24.88> mathx12
      }{}
\DeclareSymbolFont{mathx}{U}{mathx}{m}{n}
\DeclareMathSymbol{\intop}  {1}{mathx}{"B3}

\newcommand{\wt}{\widetilde}
\newcommand{\wh}{\widehat}

\let\temp\phi
\let\phi\varphi
\let\varphi\temp

\newcommand{\pr}{\mathbb{P}}

\newcommand{\R}{\mathbb{R}}

\newcommand{\normalN}{\mathcal{N}}
\newcommand{\given}{\,|\,}

\renewcommand{\norm}[1]{\Vert#1\Vert}

\newcommand{\ip}[1]{\langle#1\rangle}

\newcommand{\mix}{\mathscr{M}}
\newcommand{\mor}{\mathcal{R}}

\title{Uniform Consistency in Nonparametric Mixture Models}
\author{Bryon Aragam and Ruiyi Yang}
\date{\emph{University of Chicago and Princeton University}}

\begin{document}
\maketitle

{\let\thefootnote\relax\footnote{Contact: \texttt{bryon@chicagobooth.edu}, \texttt{ry8311@princeton.edu}}}

\begin{abstract}
We study uniform consistency in nonparametric mixture models as well as closely related mixture of regression (also known as mixed regression) models, where the regression functions are allowed to be nonparametric 
and the error distributions are assumed to be convolutions of a Gaussian density.
We construct uniformly consistent estimators under general conditions while simultaneously highlighting several pain points in extending existing pointwise consistency results to uniform results. 
The resulting analysis turns out to be nontrivial, and several novel technical tools are developed along the way. 
In the case of mixed regression, we prove $L^{1}$ convergence of the regression functions while allowing for the component regression functions to intersect arbitrarily often, which presents additional technical challenges.
We also consider generalizations to general (i.e., non-convolutional) nonparametric mixtures.
\end{abstract}

\setcounter{tocdepth}{2}
\tableofcontents

\section{Introduction}
\label{sec:intro}

Mixture models are a classical approach to modeling heterogeneous populations composed of many subpopulations, and have found a variety of applications in prediction and classification \citep{castelli1995,castelli1996,cozman2003,dan2018ssl}, clustering \citep{fraley2002,melnykov2010finite}, and latent variable models \citep{allman2009identifiability,gassiat2020identifiability,kivva2021learning,kivva2022identifiability}. Mixture models can also be used as a flexible tool for density estimation \citep{genovese2000rates,ghosal2007,kruijer2010adaptive} and arise in the study of empirical Bayes \citep{saha2020nonparametric,feng2015nonparametric} and deconvolution \citep{fan1991optimal,zhang1990fourier,moulines1997maximum}. When covariates are involved, mixtures can be used to model heterogeneous dependencies between an observation $Y$ and some covariate(s) $X$, in which the conditional distribution $\pr[Y\given X=x]$ arises as a mixture of multiple (noisy) regression curves. Despite their relevance and usefulness in applications,  mixture models   can be   notoriously difficult to analyze:   Except in special cases, mixture models are a classical example of a nonidentifiable, irregular statistical model. 
  As one might imagine, this situation is exacerbated for \emph{nonparametric} mixtures, to the extent that even fundamental properties such as identifiability and consistency remain only partially addressed.

For parametric mixture models, many of these issues have been carefully addressed: We now have optimal estimators for Gaussian mixtures \citep{heinrich2018strong,wu2020optimal,doss2020optimal}, a detailed understanding of the EM algorithm for mixtures \citep{balakrishnan2017statistical,cai2019chime}, and efficient algorithms for mixed linear regression models \citep{yi2014,kwon2021minimax}. The situation for nonparametric mixtures, however, is quite different. Here and in the sequel, by a ``nonparametric mixture'' we mean a finite mixture whose mixture components belong to a nonparametric family of distributions. For both vanilla nonparametric mixtures and mixtures of nonparametric regressions, much less is known despite many decades of work. 
  For example, although there is a substantial body of work focused on core identifiability and estimation problems, \emph{uniform} consistency has been comparatively understudied; see Section~\ref{sec:review} for a more detailed review of previous work.

Motivated by this disparity, in this paper we study uniform consistency in nonparametric mixture models and highlight several subtleties that arise when constructing uniformly consistent estimators and that appear to be peculiar to the setting of nonparametric mixtures. Although uniform consistency is often an afterthought---typically amounting to compactness and uniformity assumptions on the model---we hope to illustrate that for nonparametric mixtures, uniform consistency is a subtle matter with some surprising properties.
By ``uniform consistency'' we mean consistency that is uniform over a statistical model. 
We will focus on the nonparametric generalization of mixed regression, in which $\pr[Y\given X=x]$ is a mixture over $K$ nonparametric regression models $m_{k}(x)+z_{k}$, where the error distribution is assumed to be unknown and comes from a nonparametric family of densities. As a special case, this subsumes vanilla nonparametric mixtures (i.e., without covariates), which will be considered as well.

Let us begin by introducing the statistical model that will be our primary interest (for technical definitions, see Section~\ref{sec:model}): 
The response $Y$ is modeled by $m_{k}(X)+z_{k}$ with probability $\lambda_{k}$ ($k=1,\ldots,K$) and $z_{k}\sim f$, where $m_{k}$ are regression functions, $f$ is a density function, and $\lambda_{k}$ are weights satisfying $0<\lambda_{k}<1$ and $\sum_{k}\lambda_{k}=1$. Assuming as usual that the noise $z_{k}$ is independent of the covariates $X$, this implies that the conditional density $p(\,\cdot\given x)$ satisfies %
\begin{align}
\label{eq:defn:npmor}
p(y\given x)
= \sum_{k}\lambda_{k}f(y-m_{k}(x)).
\end{align}
In other words, for each fixed $x$, we have a mixture model whose weights and components are given by $\lambda_{k}$ and $f(y-m_{k}(x))$, respectively. Geometrically, this can be visualized as a location mixture of $K$ components whose shape is given by $f$ and whose location (mean) is given by $m_{k}(x)$; as $x$ varies, the density $f$ is translated by the value $m_{k}(x)$ (see Figure \ref{figure:modal regression}). Although our main focus will be on the case where each mixture component has the same error density $f$, extensions to unequal error densities are discussed in Sections~\ref{sec:npmix} and~\ref{sec:gen}.

\begin{figure}[t]
\minipage{0.333\textwidth}
\resizebox{\columnwidth}{!}{
\begin{tikzpicture}
\begin{axis}[
    axis lines = left,
    xlabel = \(x\),
    ylabel = {},
    color=black,
    xmin=-4*pi, ymin=-6, xmax=4*pi, ymax=6
]
    
\addplot[domain=-15:15, samples=100, color=red]{sin(deg(x))};
\addplot[domain=-15:15, samples=100, color=blue]{-sin(deg(x))};
\draw[line width=0.55mm,->] (axis cs:0, 0) -- (axis cs:1, 1);
    \draw[line width=0.5mm,->] (axis cs:0, 0) -- (axis cs:1, -1);
\end{axis}
\end{tikzpicture}
}
\subcaption{} \label{figure:transversality}
\endminipage
\minipage{0.333\textwidth}
\resizebox{\columnwidth}{!}{
\begin{tikzpicture}
\begin{axis}[
    axis lines = left,
    xlabel = \(x\),
    ylabel = {},
    color=black,
    xmin=-3.5*pi, ymin=-6, xmax=3.5*pi, ymax=6
]
     
\addplot[domain=-15:15, samples=100, color=red]{3+sin(deg(x))};
\addplot[domain=-15:15, samples=100, color=red,dashed]{3.4+sin(deg(x))};
\addplot[domain=-15:15, samples=100, color=red,dashed]{2.6+sin(deg(x))};

\addplot[domain=-15:15, samples=100, color=blue]{-3-sin(deg(x))};
\addplot[domain=-15:15, samples=100, color=blue,dashed]{-3.4-sin(deg(x))};
\addplot[domain=-15:15, samples=100, color=blue,dashed]{-2.6-sin(deg(x))};

\addplot[green, samples=150,line width=0.5mm] (1*(0.1*exp(-(x-2.6)^2/2/0.05)+0.4*exp(-(x-3.4)^2/2/0.05)+0.4*exp(-(x+2.6)^2/2/0.05)+0.1*exp(-(x+3.4)^2/2/0.05))/1.414/1.77/0.05,x);

\end{axis}
\end{tikzpicture}
}
\subcaption{} \label{figure:modal regression}
\endminipage
\minipage{0.333\textwidth}
\resizebox{\columnwidth}{!}{
\begin{tikzpicture}
\begin{axis}[
    axis lines = left,
    xlabel = \(x\),
    ylabel = {},
    color=black,
    xmin=-3.5*pi, ymin=-6, xmax=3.5*pi, ymax=6
]

\addplot[domain=-15:15, samples=100, color=red]{1+sin(deg(x))};
\addplot[domain=-15:15, samples=100, color=blue]{-1-sin(deg(x))};
\addplot[green, samples=150,line width=0.5mm] (0.25*(exp(-(x-2)^2/2/0.05)+exp(-(x+2)^2/2/0.05))/1.414/1.77/0.05+pi/2,x);
\pgfplotsinvokeforeach{pi/2}{
  \draw[dashed] ({rel axis cs: 0,0} -| {axis cs: #1, 0}) -- ({rel axis cs: 0,1} -| {axis cs: #1, 0});}
\end{axis}
\end{tikzpicture}
}
\subcaption{} \label{figure:point of separation}
\endminipage
    \caption{Examples of mixed regression models and the underlying assumptions. (a). Transversality: The two curves have different derivatives at points of intersection as shown by the black arrows. (b). Modal regression: The green curve represents a multi-modal error distribution, for which modal regression yields a four-component regression model (dashed lines) instead of two components (solid lines). 
    (c). An example of two non-transversal (equal derivatives at intersections) regression functions that intersect infinitely often. The dashed line represents a point of separation. 
    }
    \label{fig:my_label}
\end{figure}
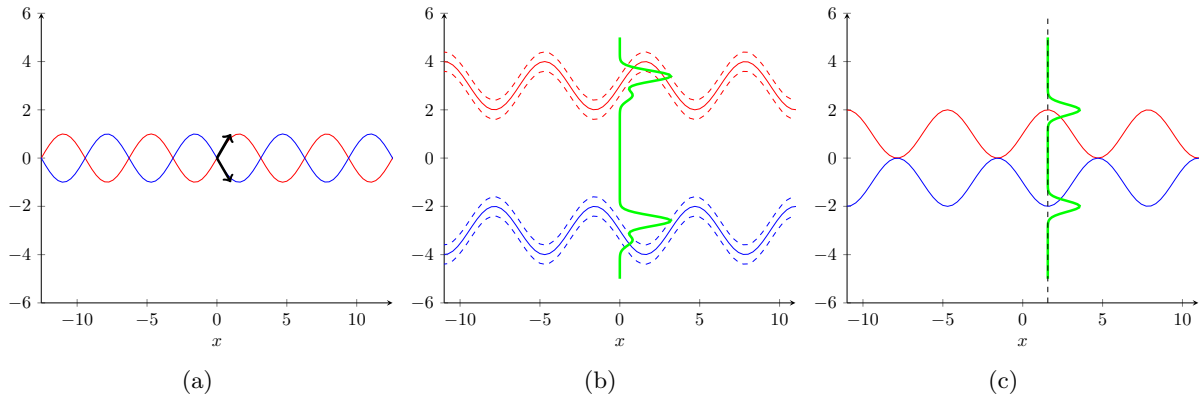

Our main results establish identifiability and \emph{uniform} consistency in the model \eqref{eq:defn:npmor} when both $f$ and the $m_{k}$ are unknown and nonparametric (and in particular, nonlinear and non-Gaussian, respectively).
In order to rescue identifiability, we rely on local separation between regression functions.   This is a natural assumption that arises in applications (often implicitly) involving data clustering \citep{fraley2002,melnykov2010finite}, such as computer vision \citep{kampffmeyer2019deep} and differential expression in genetics \citep[e.g.][]{pan2002model,si2014model,erola2020model}. Weakly separated mixtures present estimation challenges even in parametric models \citep{arora2005learning,regev2017learning}, and this has practical implications for example in causal inference \citep{feller2016weak}. Of course, unless stronger assumptions are made, one would not expect to be able to distinguish mixture components that overlap significantly.
To make the concept of separation concrete, we focus on regression models \eqref{eq:defn:npmor} whose error distributions can be written as the convolution of a Gaussian density,   i.e., when 
\begin{align}
    f=\phi_{\sigma}\ast G_0:=\int_{\mathbb{R}} \phi_{\sigma}(x-\theta)dG_0(\theta),\label{eq:convolution assumption}
\end{align}
where $\phi_{\sigma}$ is the density of $\mathcal{N}(0,\sigma^2)$ and $G_0$ is a compactly supported probability measure over $\mathbb{R}$. The model \eqref{eq:convolution assumption} gives a natural and accessible way to quantify the separation between individual components of \eqref{eq:defn:npmor} at a point $x$, for instance as the distance between the supports of the measures $G_0(\cdot-m_k(x))$ while retaining identifiability.
Densities of the form \eqref{eq:convolution assumption} are quite flexible and have appeared previously in the literature on nonparametric estimation \citep[e.g.][]{genovese2000rates,ghosal2001entropies} and hypothesis testing \citep[e.g.][]{efron2004large,cai2010optimal}.
Indeed, any Borel probability measure on $\mathbb{R}$ can be approximated by such a density \cite[see e.g.][Corollary 6]{nguyen2019approximations}, which satisfies the need in applications for flexible error models. 
Thus, this model serves as a natural first step in understanding more general nonparametric mixtures.

Unlike previous work, we focus on consistency of estimating $f$ and the $m_k$'s in the $L^{1}$ norm, which presents particular challenges for the $m_{k}$'s.  In particular, estimation of $m_{k}(x)$ for any fixed $x$ is much simpler and does not require careful handling near points where two different regression curves may intersect. Our results in fact allow for up to countably many such intersections 
as long as there exists a \emph{single} point where the regression curves are well-separated. 
Crucially, however, we do not assume that the regression functions are uniformly separated and in fact allow for different regression functions to intersect. 
As a matter of independent interest, our results also require a careful analysis of a distance-based estimator for vanilla nonparametric mixtures. This analysis involves several new ideas and is crucial to obtaining uniform bounds on the error of our proposed mixed regression estimator.
 
\begin{remark}
The term ``pointwise'' can have two distinct meanings in our setting: The usual pointwise consistency of an estimator and pointwise convergence of the functions $\widehat{m}_k$. Recall that the latter means $\widehat{m}_k(x)\to m_k(x)$ (e.g. in probability) for each $x$, as opposed to $L^1$ consistency which requires $\norm{\widehat{m}_k - m_k}_1\to 0$. To avoid confusion, we refer to pointwise convergence of the function values as ``convergence of $\widehat{m}_k(x)$ for fixed $x$'', and reserve ``pointwise'' for pointwise consistency, which in our setting \emph{always} assumes $L^1$ consistency of the regression estimates $\widehat{m}_k$. In particular, uniform consistency means $L^1$ convergence, uniformly over a family of regression functions to be defined shortly.
\qed
\end{remark}

\paragraph*{Contributions}
More precisely, the main results of this paper can be summarized as follows:

\begin{enumerate}
    \item   \emph{Impossibility} (Sections \ref{sec:unif:npmor:fail},~\ref{sec:unif:npmix:fail}): We show with explicit examples that without additional assumptions on the model, there cannot exist a uniformly consistent estimator, even when there exists a pointwise consistent estimator. See also Section~\ref{sec:consistent estimation special case}.   
    \item \emph{Mixed regression} (Section \ref{sec:npmor}):   By exploiting a point of separation,   we introduce a uniformly consistent estimator for classes of mixed regression models \eqref{eq:defn:npmor} with convolutional Gaussian error densities \eqref{eq:convolution assumption}.
    The resulting analysis reveals several subtleties regarding uniform consistency issues for these models. 
    \item \emph{Vanilla nonparametric mixtures} (Section \ref{sec:npmix}): We introduce a uniformly consistent estimator for finite nonparametric mixture models   under the same convolution assumption \eqref{eq:convolution assumption} while allowing different mixture components with different $G_0$'s for each component.    
    The resulting analysis introduces a novel project-smooth-denoise construction, which may be of independent interest.
    \item \emph{Extension to general densities and other generalizations} (Section \ref{sec:gen}): We consider generalizations where the error density is not necessarily a convolutional Gaussian mixture and introduce a pointwise consistent estimator in this setting. We also consider the case where the error densities $f$ are allowed to depend on $k$.
      Finally, we discuss other generalizations such as allowing for different points of separation and higher-dimensional analogues of our results.

\end{enumerate}
  Each of these sections contains a proof outline for the main result in that section, while deferring all technical proofs to the appendices.  

We also discuss identifiability in these models (Section~\ref{sec:model}) and show that \emph{pointwise} consistency in both models is straightforward.   A unifying theme throughout is that while pointwise consistency results may be straightforward to derive based on existing literature, uniform results are much more subtle and require novel estimators, above and beyond simply adding uniform assumptions to existing pointwise estimators.
As our intention is to expose and highlight these subtleties, our emphasis is on generality and minimal assumptions. As an aid to the reader, we have included numerous examples to illustrate our assumptions, as well as a concrete set of assumptions in Section~\ref{sec:npmor:concrete} for
the reader interested in a more digestable version of our general assumptions.

\paragraph*{Overview}

\begin{figure}[t]
\begin{subfigure}[t]{0.49\textwidth}
  \centering 
  \includegraphics[width=\textwidth]{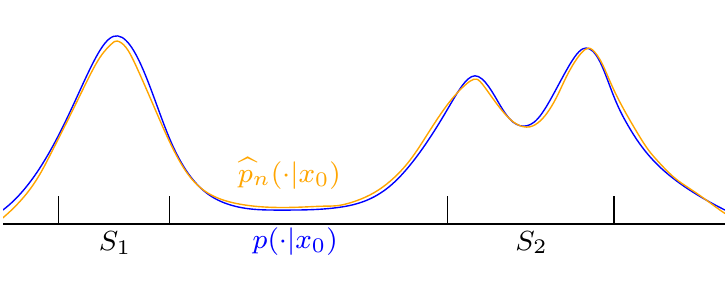}
  \caption{True density $p(\,\cdot\given x_0)$ (blue) and the conditional density estimator $\widehat{p}_n(\,\cdot\given x_0)$ (yellow). High-density regions of $p(\,\cdot\,\given x_0)$ are marked by the intervals $S_1$ and $S_2$.
  }\label{fig:true}
\end{subfigure}
~\begin{subfigure}[t]{0.49\textwidth}
  \centering 
  \includegraphics[width=\textwidth]{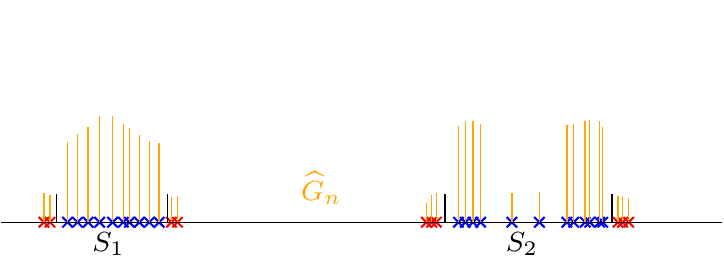}
  \caption{Estimated (discrete) mixing measure $\widehat{G}_n$ after projecting $\widehat{p}_n(\,\cdot\given x_0)$. The $\times$s represent the atoms of $\widehat{G}_n$, the sticks represent the associated weights. Red $\times$s denote outliers.}\label{fig:project}
\end{subfigure}
\\
\vspace{1em}
\\
\begin{subfigure}[t]{0.49\textwidth}
  \centering 
  \includegraphics[width=\textwidth]{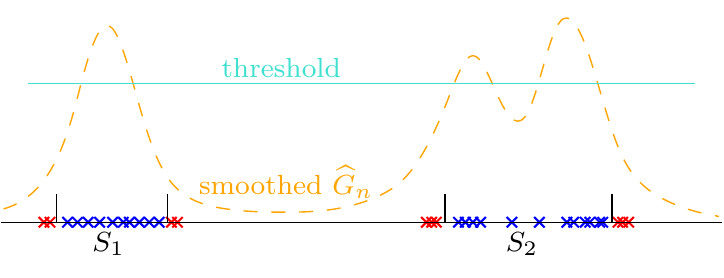}
  \caption{Smoothed estimate of $\widehat{G}_n$ (note that this is different from $\widehat{p}_n(\,\cdot\given x_0)$) together with the estimated threshold used to locate high-density regions.}\label{fig:smooth}
\end{subfigure}
~
\begin{subfigure}[t]{0.49\textwidth}
  \centering 
  \includegraphics[width=\textwidth]{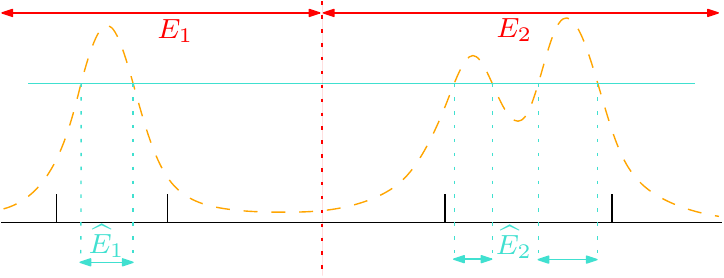}
  \caption{Estimated high density regions $\widehat{E}_1$ and $\widehat{E}_2$, which are used to construct a partition $E_1\cup E_2$ of the real line.}\label{fig:denoise}
\end{subfigure}
\caption{Overview of the project-smooth-denoise procedure. The intervals $S_1$ and $S_2$ represent high-density regions of $p(\,\cdot\,\given x_0)$; see \eqref{eq:def:supp} for a formal definition. 
An (a) initial conditional density estimator $\widehat{p}_n(\,\cdot\given x_0)$ is (b) projected then (c) smoothed and denoised to construct a partition $(E_1,E_2)$ of the input space.
The outputs are the estimated component mixing measures $\widehat{G}_n(\cdot\,|\,E_1)$ and $\widehat{G}_n(\cdot\,|\,E_2)$.
}\label{fig:overview}
\end{figure}

In order to present the main ideas at a high-level, here we give an overview of our proposed estimator. For fixed $x$, recall that the conditional density $p(\,\cdot\,\given x)$ in \eqref{eq:defn:npmor} is itself a mixture model with $K$ components.
\begin{enumerate}
    \item Construct a conditional density estimator $\widehat{p}_n(\,\cdot\,\given x)$.  
    \item Estimate the mixing proportions and the error density using the following project-smooth-denoise procedure (Figure~\ref{fig:overview}):
    \begin{enumerate}
        \item Project $\widehat{p}_n(\,\cdot\given x_0)$ onto the space of mixture of Gaussians, where $x_0$ is a point of separation (Figure \ref{fig:project}).  
        \item Smooth the obtained mixing measure $\widehat{G}_n$ to a density (Figure \ref{fig:smooth}).
        \item Denoise the original mixture by thresholding the smoothed density
        to recover the mixture components (Figure \ref{fig:denoise}).   
    \end{enumerate}
    \item Estimate the regression functions using a minimum distance estimator. 
\end{enumerate}
As mentioned, a crucial ingredient in the above procedure is the estimation of vanilla nonparametric mixtures (Step 2), whose discussion is separated into Section~\ref{sec:npmix}.

The rest of the paper is organized as follows. We review previous work in Section~\ref{sec:review}. In Section \ref{sec:model} we formalize our problem setup and assumptions, and state an identifiability result. Section \ref{sec:npmor} introduces an estimation procedure and shows its uniform consistency for certain families of mixed regression models. Section~\ref{sec:npmix} discusses estimation of vanilla nonparametric mixtures, which is a key ingredient in estimating mixed regression models. Section \ref{sec:gen} discusses several generalizations stemming from our model assumptions, and Section~\ref{sec:discussion} concludes with some discussion. All proofs are deferred to the appendices; Appendix~\ref{sec:proofs} contains detailed technical proofs, with additional supporting lemmas deferred to Appendices~\ref{sec:technical lemma} and \ref{sec:density estimator}.

\vspace{0.2cm}
\noindent{\textbf{Notation.}}
For $1\leq p <\infty$ we shall denote $\|f\|_{L^p(S)}=(\int_S |f(x)|^p dx)^{1/p}$ and $\|f\|_{L^{\infty}(S)}=\operatorname{ess\,sup}_{x\in S} |f(x)|$ for a subset $S\subset \mathbb{R}$, and simply $\|f\|_p$ when $S=\mathbb{R}$. We denote $a\vee b=\operatorname{max}\{a,b\}$ and $a\wedge b=\operatorname{min}\{a,b\}$. The notation $\mathcal{F}$ stands for the Fourier transform and $\mathcal{F}^{-1}$ its inverse.   For a density $\varphi$ and a probability measure $G$, we denote $(\varphi\ast G)(x)=\int \varphi(x-\theta) dG(\theta)$ as their convolution. For two nonempty sets $A,B\subset \mathbb{R}$, we denote $\operatorname{dist}(A,B)=\operatorname{inf}_{x\in A,y\in B}|x-y|$.      Lastly for two probability measures $P,Q$ over $\mathbb{R}$, the 1-Wasserstein distance is defined as 
\begin{align}
    W_1(P,Q)=\underset{\gamma\in\Gamma(P,Q)}{\operatorname{inf}}\int_{\mathbb{R}\times \mathbb{R}} |x-y| d\gamma (x,y), \label{eq:def wass}
\end{align} 
where $\gamma\in\Gamma(P,Q)$ denotes the set of all couplings between $P$ and $Q$.

\section{Review of previous work}
\label{sec:review}

  The prototypical mixture of regression model as in \eqref{eq:defn:npmor} would be one where both the $m_k$'s and $f$ are parametric, e.g., linear regression with Gaussian errors. Based on this, there have been many extensions that move beyond parametric regression functions, parametric errors, constant mixing proportions, or any combination of the three. In this section, we shall review some of the literature based on whether parametric assumptions are imposed on the regression functions. For our later developments, analysis of vanilla nonparametric mixtures turns out to be crucial, and we also give an updated account for them.  

\paragraph*{Parametric regression functions}
  
Arguably the most popular version of the mixed regression (also known as mixture of regressions) model is the mixed linear regression model, which is the special case $m_{k}(X) = \ip{\theta_{k},X}$ in \eqref{eq:defn:npmor}, where $\theta_{k}\in\R^{p}$. In this setting, \cite{young2010mixtures,huang2012mixture} extend the usual parametric model to allow covariate-dependent mixing proportions, and \cite{hunter2012semiparametric,vandekerkhove2013estimation} study the case of general nonparametric errors. Focusing more on computational guarantees,   \cite{yi2014,kwon2020converges,kwon2021minimax} investigate convergence of EM algorithms and \cite{chen2014convex,hand2018convex,li2018learning,yen2018mixlasso} study estimation with low or optimal sample complexity, with \cite{chen2014convex} also considering nonparametric errors.   An excellent overview of mixed linear regression, and mixture models more broadly, is \cite{fruhwirth2006}.

Closely related to mixed linear regression models are so-called mixture of experts models \citep{jacobs1991,jordan1994}. In the mixture of experts model, it is assumed that \citep[following the notation of][]{jiang1999}
\begin{align}
p(y\given x)
= \sum_{j}g_{j}(x)\pi(h_{j}(x), y),
\end{align}
where $h_{j}(x)=\alpha_{j}+\ip{\beta_{j},x}$ is the linear mean response of $Y$ conditional on $X$.
Related work on these models includes results on approximation \citep{jiang1999,nguyen2016,zeevi1998}, identifiability \citep{jiang1999identifiability}, and estimation \citep{makkuva2019breaking,ho2019convergence}.
There are two technical distinctions between mixtures of experts and mixtures of regressions: 1) The weights $g_{j}(x)$ (also known as the ``gating functions'') are allowed to depend on $x$, and 2) The conditional mean functions $\pi(h_{j}(x), y)$ typically follow a very specific parametric (e.g. generalized linear) structure. At a more basic level, mixtures of experts are typically used to approximate a single, nonparametric response and are not in general identifiable. These distinctions stand in contrast to our setting in which the motivation is to identify and estimate the $K$ heterogeneous response curves $m_{k}(x)$. 

\paragraph*{Nonparametric regression functions }
More closely related to our work are \cite{huang2013,xiang2018semiparametric}, where the former considers a special case of \eqref{eq:defn:npmor} in which the error densities $f_{k}$ are Gaussian (i.e., $z_{k}\sim\normalN(0,\sigma_{k}^{2}(x))$) while allowing for general $\lambda_{k}(x)$, $m_{k}(x)$, and $\sigma_{k}^{2}(x)$.
Here, the authors prove consistency and asymptotic normality for fixed $x$ by assuming the $m_{k}$ are differentiable and transversal, i.e., that the derivatives of each $m_{k}$ differ at points of intersection (see Figure \ref{figure:transversality}). It is not hard to see why this condition is useful: If two regression functions are allowed to match derivatives at a point of intersection, then one can construct two smooth sets of regression functions that will yield the same joint $\pr(X,Y)$: The original $m_{k}$ and $m_{j}$, as well as 
\begin{align}
\label{eq:transversal:fail}
\wt{m}_{k}(u) = \begin{cases}
m_{k}(u) & u \le x \\
m_{j}(u) & u \ge x
\end{cases},
\qquad
\wt{m}_{j}(u) = \begin{cases}
m_{j}(u) & u \le x \\
m_{k}(u) & u \ge x
\end{cases}.
\end{align}
Transversality implies that $\wt{m}_{k}$ and $\wt{m}_{j}$ as constructed above would be nondifferentiable at $x$, violating the differentiability requirement.
A similar ``non-parallel'' condition appears in \cite{kitamura2018nonparametric}, where identifiability in nonparametric mixture models is studied in depth. 
Their approach to identifiability is based on moment generating functions, similar to \cite{teicher1963}.   Our approach is quite different to these works that stop short of proving uniform consistency,   which is our main focus and more subtle than pointwise consistency. See Remark~\ref{rem:compare} for a more detailed comparison.

Another common approach to modeling heterogeneous regression functions is modal regression \citep{yao2012local,yao2014new,chen2016nonparametric}, which aims to find the modes of the joint distribution $\pr(X,Y)$. This is a flexible nonparametric model that provides an attractive alternative to parametric models for mixed regression, however, one notable drawback of modal regression is the inability to handle error distributions that are themselves multimodal, even if they are well-separated (see Figure~\ref{figure:modal regression}). Modal regression also suffers from issues near points of intersection, as illustrated by Figure~2 in \cite{chen2016nonparametric}.

\paragraph*{Vanilla nonparametric mixtures }
As indicated above, a key technical hurdle in the study of nonparametric mixed regression models is the identifiability of the \emph{conditional} mixture for fixed $x$. As such, it is worth pausing to review what is known about identifiability in vanilla (i.e., covariate-free) nonparametric mixtures. For a detailed overview, see \cite{fruhwirth2006} and \cite{ritter2014}. By a nonparametric mixture, we mean a probability measure of the form $\mu=\sum_{k=1}^{K}\lambda_{k}\mu_{k}$, where $\mu_{k}$ are probability measures and $\lambda_{k}$ nonnegative weights summing to unity. It is clear that \emph{any} probability measure $\mu$ can be written as a mixture in an infinite number of ways, simply by noting that $\mu(A)=\mu(A\given E)\mu(E) + \mu(A\given E^{c})\mu(E^{c})$ for any measurable sets $A$ and $E$. 
Clearly, additional assumptions are needed to ensure identifiability.
  For example, it is well-known that translation families and product mixtures are identifiable \citep{teicher1960,teicher1961,teicher1963,teicher1967}, and a necessary and sufficient condition is linear independence of densities \citep{chandra1977mixtures,yakowitz1968}. The use of minimum distance estimators to estimate the mixing measure is classical \citep{deely1968}, and there is a now refined analysis of optimality in estimating parametric mixtures \citep[e.g.][]{ishwaran1996identifiability,chen1995,heinrich2018strong,ho2016convergence,ho2019singularity}. To the best of our knowledge, similar optimality results have not been obtained for nonparametric mixtures.
  Recent work focuses on nonparametric extensions when observations are grouped according to the latent class assignments \citep{vandermeulen2019operator,ritchie2020consistent}, when the covariates carry certain latent structures \citep{allman2009identifiability,gassiat2016nonparametric}, and when the component distributions are symmetric \citep{bordes2006,hunter2007}, products of univariate distributions \citep{hall2003nonparametric,hall2005mixture,elmore2005}, covariate-dependent \citep{compiani2016using}, or well-separated \citep{aragam2020identifiability}. 
  The basic thrust of this line of work on nonparametric identifiability is to restrict the component measures $\mu_{k}$ to satisfy various regularity assumptions such as independence, symmetry, or separation. In the present work, we build on the idea of separation studied in \cite{aragam2020identifiability}. Finally, we note that the models we introduce in the next section, based on convolutional mixtures, have appeared in a variety of contexts previously including Bayesian nonparametrics \citep{nguyen2013} and the empirical geometry of multivariate data \citep{koltchinskii2000empirical}.

\section{Model assumptions and identifiability}\label{sec:model}

We begin by presenting the details of our model assumptions and some preliminaries on identifiability in this section. We also briefly discuss pointwise vs. uniform consistency in these models.

\subsection{Model assumptions}
\label{sec:model:assm}
Recall the basic model \eqref{eq:defn:npmor}. It follows that for any marginal density $p_{X}(x)$ over $x$, the joint density is specified as follows:
\begin{align}
    p(x,y)=p_X(x)\sum_{k=1}^K \lambda_kf(y-m_k(x)), \quad \quad (x,y)\in [a,b]\times \mathbb{R}, \label{eq:mor joint density}
\end{align}
where $\sum_{k=1}^K\lambda_k=1$ and $f$ is a common error density satisfying $\int_{\mathbb{R}} xf(x)dx=0$.   Throughout this paper, we shall assume $K$ is known and $\operatorname{min}_k\lambda_k>0$ so that \eqref{eq:mor joint density} is indeed a $K$-component mixture model for each $x$; indeed when $K$ is unknown nonparametric mixtures are known to be fundamentally nonidentifiable, even under strong additional assumptions \cite[see e.g.][Section 2.2]{aragam2020identifiability}.  We also assume that $p_X>0$ on $[a,b]$ for simplicity; see Section~\ref{sec:step 1} for discussion and Section~\ref{sec:gen:other} on how to generalize this. Consider the following parameter space  
\begin{align*}
    \mor=\bigg\{\left(p_X,f,\{\lambda_k\}_{k=1}^K, \{m_k\}_{k=1}^K\right)\in \mathcal{P}([a,b])\times\mathcal{P}(\mathbb{R})\times \mathbb{R}^K\times [\mathcal{C}^0([a,b])]^K:\\
    \underset{x\in[a,b]}{\operatorname{inf}}\, p_X(x)>0,\,\,\int_{\mathbb{R}}xf(x)dx=0,\,\,\underset{k=1,\ldots,K}{\operatorname{min}}\,\lambda_k>0,\,\,\sum_{k=1}^K \lambda_k=1\bigg\},
\end{align*} 
where $\mathcal{P}(\Omega)$ and $\mathcal{C}^0(\Omega)$ denote the set of all probability densities and the set of real-valued continuous functions on $\Omega\subset\R$, respectively.   Let $\Phi:\mor \rightarrow \mathcal{P}([a,b]\times \mathbb{R})$ be the map that associates a parameter tuple in $\mor$ to the corresponding density \eqref{eq:mor joint density}. We say that a subfamily $\Phi(\mathcal{L})$ of  $\Phi(\mor)$ with $\mathcal{L}\subset \mor$ is identifiable if $\Phi$ is injective over $\mathcal{L}$. Without additional assumptions, the model $\Phi(\mor)$ is not identifiable,   and therefore the purpose of this section is to introduce subfamilies of $\mor$ over which identifiability is ensured.

Our main results are for the case where $f$ can be expressed as a Gaussian convolution as in \eqref{eq:convolution assumption}.
We do not assume that $G_{0}$ has a density. 
As we discuss in Section~\ref{sec:unif:npmix:fail}, even when they are identifiable, nonparametric mixtures may be poorly behaved, so this assumption is made in order to make uniform estimation feasible, although we do not believe it is fundamentally necessary and can likely be relaxed. 
The key structural assumption that we will exploit is the following idea of a point of separation. 
\begin{assm}[Point of separation]\label{assumption:point of sep}
There exists $x_0\in [a,b]$ so that 
\begin{align*}
    \underset{j\neq k}{\operatorname{min}}\, |m_j(x_0)-m_k(x_0)|>2\operatorname{diam}(\operatorname{supp}(G_0)). 
\end{align*}
\end{assm}
\noindent
In particular, we only require there be one such point of separation: Away from this point of separation, the $m_k$'s can be arbitrarily close and even intersect multiple times on the rest of the domain (see Figure \ref{figure:point of separation}). The rationale of this assumption is that if the regression functions stay close over the whole interval, then there is less hope to estimate each of them. 
Since the conditional density of \eqref{eq:mor joint density} at $x_0$ is a convolutional Gaussian mixture with mixing measure $\sum_{k=1}^K \lambda_k G_0(\cdot-m_k(x_0))$, Assumption \ref{assumption:point of sep} is essentially requiring that the separation between the supports of $G_0(\cdot-m_k(x_0))$ is larger than their diameter so that single linkage clustering for instance can identify them. Under stronger assumptions on $G_0$, this assumption can be relaxed; see Section~\ref{sec:npmix:cond}.

\subsection{Identifiability}
\label{sec:model:ident}
Before presenting our main results, we pause to discuss identifiability in the model $\Phi(\mor)$. See \cite{kitamura2018nonparametric,aragam2020identifiability} 
for a more detailed investigation of identifiability in nonparametric mixtures, including generalizations of some of the results below. Let $\mor_{x_0}\subset \mor$ be a subfamily satisfying the following conditions:
\begin{enumerate}[label=(A\arabic*)]
    \item\label{assm:npmor:conv} $f=\phi_{\sigma}\ast G_0$ with $G_0$ having compact support (cf. \eqref{eq:convolution assumption}). 
    \item\label{assm:npmor:distinct} $\underset{j\neq k}{\operatorname{min}}\,|\lambda_j-\lambda_k|>0$. 
    \item\label{assm:npmor:sep} $\underset{j\neq k}{\operatorname{min}}\, |m_j(x_0)-m_k(x_0)|>2\operatorname{diam}(\operatorname{supp}(G_0))$. 
    \item\label{assm:npmor:intersect} The set $Z=\{x\in[a,b]: \exists j\neq k, \,m_j(x)=m_k(x) \}$ is countable. 
\end{enumerate}
\noindent
For obvious reasons, the value $x_{0}$ will be referred to as a \emph{point of separation} and points in the set $Z$ will be referred to as \emph{points of intersection} in the sequel.

\begin{thm}\label{thm:identifiability convolutional}
The mixed regression model $\Phi(\mor_{x_0})$ is identifiable.
\end{thm}
\begin{remark}\label{remark:identifiability}
Apart from assuming a point of separation, we remark that another crucial assumption in the definition of $\mor_{x_0}$ is that $\operatorname{min}_{j\neq k}|\lambda_j-\lambda_k|>0$, i.e., the mixing proportions are distinct. This will play an important role in the estimation procedure as we demonstrate in Section \ref{sec:npmor}. Roughly speaking, the distinct mixing proportions can be used to solve the label switching issue near points where the regression functions intersect. 
Although we do not require that the $m_k$'s are differentiable for identifiability, Condition \ref{assm:npmor:distinct} can be replaced by assuming the $m_k$'s are transversal (and hence differentiable) as in \cite[][Theorem 1]{huang2013}. 
\qed
\end{remark}

\begin{remark}
The assumption \ref{assm:npmor:distinct} rules out a set of weights that has measure zero and thus the model can be considered as \emph{generically} identifiable \citep{allman2009identifiability} if we drop \ref{assm:npmor:distinct}. Similar conclusions have also been observed for parametric \citep{vandermeulen2019operator,ho2019singularity} and semiparametric \citep{hunter2007,hunter2012semiparametric,bordes2006} models.   \qed
\end{remark}

The proof of Theorem \ref{thm:identifiability convolutional} can be found in Appendix~\ref{appen:proof of identifiability}, but we illustrate the main idea here. 
First, we identify the error density and the mixing proportions by exploiting the point of separation $x_0$. By restricting the mixed regression model at $x_0$, the problem reduces to that of a finite mixture and we can make use of the following identifiability result of vanilla mixture models. Define 
\begin{align}
    \mix=\left\{\left(\{f_k\}_{k=1}^K,\{\lambda_k\}_{k=1}^K, \{\mu_k\}_{k=1}^K\right)\in \mathcal{P}(\mathbb{R})^K\times \mathbb{R}^K\times \mathbb{R}^K\right\} \label{eq:finite mixture model}
\end{align}
to be   the parameter space   satisfying 
\begin{enumerate}[label=(B\arabic*)]
    \item\label{assm:npmix:conv} $f_k=\phi_{\sigma}\ast G_k$ with $G_k$ having compact support and $\int_{\mathbb{R}}x f_k(x) dx=0$. 
    \item\label{assm:npmix:wgt} $\underset{k}{\operatorname{min}}\, \lambda_k>0$ and $\sum_{k=1}^K \lambda_k=1$. 
    \item\label{assm:npmix:sep} $\underset{j\neq k}{\operatorname{min}}\,  \operatorname{dist} \left(\operatorname{supp} (G_j(\cdot-\mu_j)),\,\operatorname{supp}( G_k(\cdot-\mu_k))\right)> \underset{k}{\operatorname{max}}\,\operatorname{diam}(\operatorname{supp}(G_k))$. 
\end{enumerate}
Here $G_k(\cdot-\mu_k)$ is the translation of $G_k$ by $\mu_k$.   Let $\Psi:\mix \rightarrow \mathcal{P}(\mathbb{R})$ be the map that associates a parameter tuple $\big(\{f_k\}_{k=1}^K,\{\lambda_k\}_{k=1}^K, \{\mu_k\}_{k=1}^K\big)$ to the corresponding density $\sum_{k=1}^K \lambda_k f_k(\cdot-\mu_k)$.   
\begin{prop}\label{prop:identifiability convolutional}
  The mixture model $\Psi(\mix)$ is identifiable, i.e., $\Psi$ is injective over $\mix$.  
\end{prop}
\begin{remark}
Proposition~\ref{prop:identifiability convolutional} allows for different error densities for each component in the definition of $\mix$ and can be seen as a special case of the general results from \cite{aragam2020identifiability}, although our proof is more straightforward owing to the additional structure provided by \ref{assm:npmix:conv}.
For mixed regression models,
assumption \ref{assm:npmor:sep} implies that the conditional density $p(\,\cdot\given x_0)$ satisfies \ref{assm:npmix:sep} and so Proposition \ref{prop:identifiability convolutional} gives identifiability of $f$ and $\{\lambda_k\}_{k=1}^K$.  \qed
\end{remark}
Once the error density and the mixing proportions have been identified, we see that the model \eqref{eq:mor joint density} at each $x$ is then a finite mixture of the location family $\{f(\cdot-\mu)\}_{\mu\in\mathbb{R}}$, from which we can identify the values of the regression functions for each $x$. The difficulty would then to assemble the $m_k(x)$'s correctly across all $x\in[a,b]$. Since $Z$ is countable, we can decompose the interval $[a,b]$ as a union of subinterval $Z_k$'s where the $m_k$'s do not intersect. Continuity would allow us to correctly identify the $m_k$'s over each $Z_k$, however, a new difficulty arises when attempting to connect these $m_{k}$ across points of intersection (cf. \eqref{eq:transversal:fail}), as illustrated by the following example:

\begin{ex}
\label{ex:uneq}
Consider a two-component mixed regression model
 \begin{align*}
    y|x\sim 
    \begin{cases}
    m_1(x)+\varepsilon  \quad \operatorname{w.p.} \,\, \lambda\\
    m_2(x)+\varepsilon  \quad \operatorname{w.p.} \,\, 1-\lambda
    \end{cases}
\end{align*}  
where $m_1(x)=|x|,\,m_2(x)=-|x|$, $0<\lambda<1$ and $\varepsilon$ is Gaussian. Theorem \ref{thm:identifiability convolutional} implies that the model is identifiable if $\lambda\neq\frac12$. If $\lambda=\frac12$, it is then indistinguishable from the alternate model with means $\tilde{m}_1(x)=-x$ and $\tilde{m}_2(x)=x$. The problem occurs precisely at the point $x=0$ when we try to connect the segments on either side. However if $\lambda\neq \frac12$, then one can join the different pieces by matching the associated mixing proportions and there is a unique way to do so. In \cite{huang2013}, transversality was used to deal with this issue based on derivative information of the $m_k$'s. In particular, higher order smoothness assumptions can alleviate the label switching issue across points of intersection and guarantee identifiability. Nonetheless, we will show in Section~\ref{sec:unif:npmor:fail} that even for $\mathcal{C}^{\infty}$ regression functions there are issues with \emph{uniformity} if the mixing proportions are not distinct. \qed
\end{ex}

\noindent
In order to avoid this difficulty, 
Condition \ref{assm:npmor:distinct} on distinct mixing proportions gives a unique way to assign the labels.

\subsection{Pointwise vs. uniform consistency}\label{sec:consistent estimation special case}
Under the model assumptions introduced above, one can construct (pointwise) consistent estimators of the parameters $(f,\{\lambda_k\}_{k=1}^K,$ $ \{m_k\}_{k=1}^K)$ by following the same steps as in the identifiability arguments. In other words, we can use the local information provided by the point of separation $x_0$ to construct estimators of  the error density and the mixing proportions, which are then used to infer the regression functions globally. Let $\{(X_i,Y_i)\}_{i=1}^n$ be i.i.d. samples from the joint density \eqref{eq:mor joint density}.
\begin{prop}\label{prop:pointwise consistency convolutional}
Suppose $\mor_{x_0}$ satisfies additionally 
\begin{enumerate}[label=(A5)]
    \item\label{assumption:joint:holder} The joint density $p_X(x)\sum_{k=1}^K \lambda_kf(y-m_k(x))$ is $\beta$-H\"older continuous for some $\beta>0$. 
\end{enumerate}
Then there exists estimators $\widehat{f}_n$, $\{\widehat{\lambda}_{n,k}\}_{k=1}^K$ and $\{\widehat{m}_{n,k}\}_{k=1}^K$ so that with probability one
\begin{align*}
    \|\widehat{f}_n-f\|_1 \vee \underset{k}{\operatorname{max}}\, |\widehat{\lambda}_{n,k}-\lambda_k| \vee \underset{k}{\operatorname{max}}\, \|\widehat{m}_{n,k}-m_k\|_{L^1[a,b]} \xrightarrow{n\rightarrow \infty} 0.
\end{align*}
\end{prop}
\noindent
This result is a special case of more general results discussed in Section~\ref{sec:gen}.
Roughly speaking, consistent estimators $\widehat{f}_n$ and $\{\widehat{\lambda}_{n,k}\}_{k=1}^K$ can be obtained from a conditional density estimate $\widehat{p}_n(\,\cdot\given x_0)$ thanks to Assumption \ref{assumption:point of sep}. 
A minimum distance estimator can then be employed to obtain the estimates $\{\widehat{m}_{n,k}(x)\}_{k=1}^K$ for each $x$, which are joined using \ref{assm:npmor:distinct} to yield the $\widehat{m}_{n,k}$'s. 
This result extends to mixed regression models whose error densities are not necessarily convolutional Gaussian; see Section \ref{sec:gen} for more discussion. 

The crucial point here is that pointwise consistent estimators are relatively easy to obtain under our model assumptions, while uniformly consistent estimation requires substantially more efforts.
  As illustrated in Example~\ref{ex:uneq}, without additional regularity conditions on the $m_{k}$, it is impossible to decide how to ``split'' the curves past this intersection when the associated mixing proportions are equal (e.g., see the discussion around \eqref{eq:transversal:fail}). 
As we will show, without additional assumptions this issue about equal mixing proportions proves fatal when it comes to the existence of uniformly consistent estimators in this model (Section~\ref{sec:unif:npmor:fail}). 
Moreover, one might hope to extend the analysis of \cite{huang2013} to nonparametric errors using recent work on nonparametric mixture models \citep{aragam2020identifiability}, however, as we shall see these conditions also preclude uniformity in estimation (Section~\ref{sec:unif:npmix:fail}).  
Evidently, uniform consistency is a particularly subtle issue when it comes to nonparametric mixtures.

\section{Uniformly consistent estimation of mixed regression} \label{sec:npmor}
In this section we will study uniformly consistent estimation of the mixed regression model introduced in Section \ref{sec:model}. As in Section \ref{sec:consistent estimation special case}, the high level idea is again similar to the identifiability argument but requires a more refined analysis to push through. 
As in the previous section, we assume that $f=\phi_{\sigma}\ast G_0$ for some unknown $G_0$.
To avoid technical digressions, we fix throughout the marginal density $p_X$, assuming $\operatorname{inf}_{[a,b]} \, p_X >0$ and $\|p_X\|_{L^{\infty}[a,b]} \vee \|p^{\prime}_X\|_{L^{\infty}[a,b]}<\infty$. 
Under these assumptions, let $\mathcal{U}_{x_0}\subset \mor_{x_0}$ be a subfamily of tuples $\vartheta=(p_X,f,\{\lambda_k\}_{k=1}^K, \{m_k\}_{k=1}^K)$ satisfying the following regularity assumptions:
\begin{enumerate}[label=(C\arabic*)]
     \item $\underset{\vartheta\in\mathcal{U}_{x_0}}{\operatorname{sup}}\, \operatorname{diam}(\operatorname{supp}(G_0))<\infty$ and $\underset{\vartheta\in\mathcal{U}_{x_0}}{\operatorname{sup}}\,\Big( \underset{k}{\operatorname{max}}\, \|m_k\|_{L^{\infty}[a,b]} \vee \underset{k}{\operatorname{max}}\, \|m_k^{\prime}\|_{L^{\infty}[a,b]}\Big) <\infty$.\label{assumption:unpmor:regularity}
 \end{enumerate}
together with the following structural assumptions
 \begin{enumerate}[label=(C\arabic*)]
 \setcounter{enumi}{1}
     \item $\underset{\vartheta\in\mathcal{U}_{x_0}}{\operatorname{inf}}\, \underset{j}{\operatorname{min}}\, \lambda_j >0,
     \quad \underset{\vartheta\in\mathcal{U}_{x_0}}{\operatorname{inf}}\, \underset{j\neq k}{\operatorname{min}}\, |\lambda_j-\lambda_k|>0.$ \label{assumption:unpmor:lambda}
     \item $\underset{\vartheta\in\mathcal{U}_{x_0}}{\operatorname{inf}}\,\Big[\underset{j\neq k}{\operatorname{min}}\, |m_j(x_0)-m_k(x_0)|- 2\operatorname{diam}(\operatorname{supp}(G_0))\Big]>0.$ \label{assumption:unpmor:separation}
     \item $\underset{\vartheta\in\mathcal{U}_{x_0}}{\operatorname{inf}}\, |(\mathcal{F}G_0)(t)|\geq J(t)$ for some function $J>0$ and every $t$.  \label{assumption:unpmor:FT}
     \item $\underset{\vartheta\in\mathcal{U}_{x_0}}{\operatorname{sup}}\,\operatorname{Leb}\!\big(\{x: \exists \,j\neq k ,\,  |m_j(x)-m_k(x)|\leq a_n \}\big) \xrightarrow{n\rightarrow \infty} 0$ for all $a_n\rightarrow 0,$ where $\operatorname{Leb}(\cdot)$ denotes the Lebesgue measure of a set. \label{assumption:unpmor:sep of reg}
 \end{enumerate}
   Additional discussion on these assumptions can be found in Remark~\ref{rem:diff:mk}.  
The main result in this section is the following: 
\begin{thm}
\label{thm:uniform consistency}
 
There exist uniformly consistent estimators $\widehat{f}_n$, $\{\widehat{\lambda}_{n,k}\}_{k=1}^K$, and $\{\widehat{m}_{n,k}\}_{k=1}^K$ for $\mathcal{U}_{x_0}$, where is $\mathcal{U}_{x_0}$ is defined by \ref{assumption:unpmor:regularity}-\ref{assumption:unpmor:sep of reg}.
More specifically, for any $\varepsilon>0$, 
\begin{align*}
    \underset{\mathcal{U}_{x_0}}{\operatorname{sup}}\,\,\mathbb{P}\left( \|\widehat{f}_n-f\|_{1} \vee \underset{k}{\operatorname{max}}\, |\widehat{\lambda}_{n,k}-\lambda_k| \vee \underset{k}{\operatorname{max}}\, \|\widehat{m}_{n,k}-m_k\|_{L^1[a,b]} >\varepsilon\right) \xrightarrow{n\rightarrow\infty} 0.
\end{align*}
\end{thm}
\noindent
The remainder of this section is devoted to discussing this result and its assumptions in detail, along with a proof outline that constructs the estimators explicitly.
More specifically, in Section~\ref{sec:est}, we outline the main ideas behind the constructive proof of Theorem \ref{thm:uniform consistency}, while deferring technical details to Appendix~\ref{app:proof:sec4}.  

We emphasize that these estimators are not abstract, and will be explicitly constructed in the sequel. There are three main steps:
\begin{enumerate}
\item Estimation of the conditional density $p(\cdot\given x)$ via kernel density estimators (Section~\ref{sec:step 1});
\item Estimation of $\lambda_{k}$ and $f$ via estimation of the vanilla mixture model $p(\cdot\given x_{0})$ (Section~\ref{sec:step 2}; details in Section~\ref{sec:npmix});
\item Estimation of the regression functions $m_k$ via a minimum distance estimator (Section~\ref{sec:step 3}).
\end{enumerate}
\noindent
  Throughout, we assume that $x_0$ and $\sigma$ are known, which is crucial to illustrating our main point that uniform consistency is challenging even under such knowledge---i.e. the difficulties are not somehow due to orthogonal problems in estimating the variance or points of separation. 
The second step above is the most delicate, and involves a careful project-smooth-denoise construction that is detailed in Section~\ref{sec:npmix}.

\begin{remark}
\label{rem:diff:mk}
We now discuss briefly the assumptions made on $\mathcal{U}_{x_0}$. 
\begin{itemize}
    \item \ref{assumption:unpmor:regularity}-\ref{assumption:unpmor:separation} are simply uniform versions of those in \ref{assm:npmor:conv}-\ref{assm:npmor:sep}.
    \item \ref{assumption:unpmor:regularity} ensures a uniformly consistent conditional density estimator as a crucial first step.
To focus on mixture models, we assume the $m_k$ are differentiable in \ref{assumption:unpmor:regularity} for simplicity, however, we expect that a similar result assuming only weaker H\"older-type continuity is possible.
We remark that it is only in this step that we need the differentiability of $p_X$ and the $m_k$'s. 
    \item \ref{assumption:unpmor:lambda} is of utmost importance: We show below by example that uniform consistency is impossible without assuming distinct mixing proportions (see also Remark \ref{remark:identifiability}). 
    \item \ref{assumption:unpmor:separation} is a crucial separation condition that allows different mixture components to be identified as in Section~\ref{sec:model:ident} and will be the key structural assumption that we exploit for estimating vanilla nonparametric mixtures in Section~\ref{sec:npmix}. This can be relaxed under strong regularity assumptions on $G_0$ that we discuss in Section~\ref{sec:npmix:cond}. 
    \item \ref{assumption:unpmor:FT} is a technical assumption to ensure a modulus-of-continuity type result \eqref{eq:moc remark} for finite mixture models and allows a wide range of nonparametric $G_0$'s as discussed in Section~\ref{sec:npmor:cond:FT}.  
    This type of assumption is standard in the nonparametric deconvolution literature \citep{fan1991optimal,nguyen2013}, where $J$ is usually taken to be $c|t|^{-\beta}$ or $c\exp(-|t|^{\beta}/\gamma)$ for constants $c,\beta,\gamma$.
    \item \ref{assumption:unpmor:sep of reg} controls the separation  between different regression functions around points of intersection and will be discussed in more detail in Section~\ref{sec:npmor:cond}. In particular, this assumption is satisfied under a uniform version of the transversality assumption from \cite{huang2013}. We also provide an example where transversality fails (Example~\ref{ex:C6}), thereby proving that this is not necessary. \qed
\end{itemize}

\end{remark}

\subsection{Nonexistence of uniformly consistent estimators }\label{sec:unif:npmor:fail}
Before describing the estimator, we illustrate why the second part of Condition~\ref{assumption:unpmor:lambda} is crucial,   which might be somewhat surprising at first.   
Consider the mixed regression model
 \begin{align*}
    (\mathscr{P}):\quad y|x\sim 
    \begin{cases}
    m_1(x)+\varepsilon \quad \operatorname{w.p.} \,\, \frac12 \\
    m_2(x)+\varepsilon \quad \operatorname{w.p.} \,\, \frac12 
    \end{cases}
\end{align*} 
where $m_1,m_2\in \mathcal{C}^{\infty}(\mathbb{R})$ and $\varepsilon$ is Gaussian. Let $\{(X_i,Y_i)\}_{i=1}^n$ be i.i.d. samples from it. We will construct an alternate model that generates the same   data   distribution as $\mathscr{P}$. Without loss of generality assume there are no $X_i$'s that lie in between $u:=X_1$ and $v:=X_2$ and that $u<v$. Let $\phi\in\mathcal{C}^{\infty}(\mathbb{R})$ be a smooth function satisfying $\phi=0$ for $x\leq u$ and $\phi=1$ for $x\geq v$, and define
\begin{align*}
    \widetilde{m}_1&=(1-\phi) m_1+\phi m_2 \\
    \widetilde{m}_2&=(1-\phi) m_2 +\phi m_1.
\end{align*}
We then have $\widetilde{m}_1,\widetilde{m}_2\in\mathcal{C}^{\infty}(\mathbb{R})$ and 
\begin{align*}
    \widetilde{m}_1= \begin{cases}
    m_1 \quad &x\leq u\\
    m_2 \quad &x\geq v
    \end{cases}
    \quad \quad 
    \widetilde{m}_2=
    \begin{cases}
    m_2 \quad &x\leq u\\
    m_1 \quad &x\geq v
    \end{cases}.
\end{align*}
In particular the two models 
 \begin{align*}
    (\mathscr{P}):\quad y|x\sim 
    \begin{cases}
    m_1(x)+\varepsilon \quad \operatorname{w.p.} \,\, \frac12 \\
    m_2(x)+\varepsilon \quad \operatorname{w.p.} \,\, \frac12 
    \end{cases}\quad (\mathscr{P}_n):\quad 
    y|x\sim 
    \begin{cases}
    \widetilde{m}_1(x)+\varepsilon \quad \operatorname{w.p.} \,\, \frac12 \\
    \widetilde{m}_2(x)+\varepsilon \quad \operatorname{w.p.} \,\, \frac12 
    \end{cases}
\end{align*} 
have the same conditional distribution on $(-\infty,u]\cup [v,\infty)$ and any estimator based on $\{(X_i,Y_i)\}_{i=1}^n$ cannot distinguish between them.   Notice that the construction of such a $\mathscr{P}_n$ can be carried out for every $n$, giving a sequence of models $\{\mathscr{P}\}_{n=2}^{\infty}$, each of which is indistinguishable from $\mathscr{P}$ based on data $\{(X_i,Y_i)\}_{i=1}^n$.   
Hence there cannot be a uniformly consistent estimator for the regression functions over subfamilies of $\mor_{x_0}$ that allow equal mixing proportions even if the $m_k$'s are restricted to be $\mathcal{C}^{\infty}(\mathbb{R}).$
The problem lies precisely in the fact that in between any two adjacent $X_i$'s, the model could undergo a label switching and no degree of smoothness can prevent this if the mixing proportions are equal. In particular, the best one can do in this case is to estimate the function values $m_k(x)$ for fixed $x$ but not how to connect them. Similar problems arise in mixtures with more than two components, where one is not expected to estimate those components with equal mixing proportions.

\subsection{Construction of estimator}\label{sec:est}
Throughout this section we assume to be given $n$ i.i.d. samples $\{(X_i,Y_i)\}_{i=1}^n$ from the model \eqref{eq:mor joint density}. The estimation procedure starts by constructing a conditional density estimator $\widehat{p}_n(y|x)$ of the mixed regression model (Section \ref{sec:step 1}). By exploiting the point of separation $x_0$ in Assumption \ref{assumption:point of sep}, we will then construct estimators of the error density and the mixing proportions (Section \ref{sec:step 2})
  based on  $\widehat{p}_n(\,\cdot\given x_0)$.  %
The remaining estimation of the regression functions for a fixed $x$ reduces to a parametric estimation problem, where a minimum distance estimator is employed to achieve $L^{1}$ (as opposed to pointwise) approximation (Section \ref{sec:step 3}). While introducing in detail each of these steps, we shall see how the assumptions \ref{assumption:unpmor:regularity}-\ref{assumption:unpmor:sep of reg} progressively build up.

\subsubsection{Conditional density estimator}\label{sec:step 1}
Our estimation procedure starts by estimating the conditional density.   
Conditional density estimators have been studied extensively in the literature \citep[e.g.][]{de2003conditional,efromovich2007conditional,efromovich2005estimation,li2021minimax}, but a precise $L^1$ rate result (with explicit dependence on all the model parameters) seems to be missing. Therefore, to make our presentation self-contained, we include such a construction and bound its $L^1$ error in Appendix~\ref{sec:consistency cde}. As noted in Remark~\ref{rem:diff:mk}, this step is the only step in which differentiability of the regression functions is needed, and we further assume $p_X$ is bounded away from zero so that a simple ratio of kernel density estimators (KDE) will suffice in our setting, although more sophisticated estimators exist under weaker assumptions.   More precisely, let $\widehat{p}_n(x,y)$ be a KDE for the joint density $p(x,y)$ and $\widehat{p}_{X,n}(x)$ be a KDE for the marginal density $p_{X}(x)$. We have the following result: 
\begin{prop}\label{prop:step 1 uniform consistency}
Let $\widehat{p}_{n}(\,\cdot\given x)=\widehat{p}_n(x,\cdot)/\widehat{p}_{X,n}(x)$ be a ratio of kernel density estimators   with the box kernel $H=\frac12 \mathbf{1}_{[-1,1]}$,   with bandwidth $h_n$ satisfying 
\begin{align*}
    h_n\rightarrow 0, \quad \frac{nh_n^2}{|\log h_n|}\rightarrow\infty, \quad \frac{|\log h_n|}{\log \log n}\rightarrow \infty, \quad h_n^2 \leq ch_{2n}^2 
\end{align*}
for some $c>0$. Let $\mathcal{V}_{x_0}\subset \mor_{x_0}$ be a subfamily satisfying \ref{assumption:unpmor:regularity}.
Then we have 
\begin{align}
     \underset{\mathcal{V}_{x_0}}{\operatorname{sup}}\,\underset{ x\in[a+h_n,b-h_n]}{\operatorname{sup}}\,\mathbb{E}\|\widehat{p}_{n}(\,\cdot\given x)-p(\,\cdot\given x)\|_1 &\xrightarrow{n\rightarrow\infty}0. \label{eq:uniform L1 cde}
\end{align}
\end{prop}
\noindent
\begin{remark}
In the sequel, any conditional density estimator satisfying \eqref{eq:uniform L1 cde} suffices. 
  Possible choices of the bandwidth include $h_n\asymp n^{-1/(2+\beta)}$ for any $\beta>0$.    \qed
\end{remark}

\subsubsection{Estimating the error density and mixing proportions}\label{sec:step 2}
The next step is to construct estimators of the error density $f$ and the mixing proportions $\lambda_{k}$. 
Similarly as the discussion before Proposition \ref{prop:identifiability convolutional}, the problem in this subsection reduces to that of a finite mixture and can be decoupled from the other parts.
Since these results may be of independent interest, we defer the details to  Section \ref{sec:npmix}, where a complete description of the estimator and its analysis can be found.

  The estimator is based on a novel ``project-smooth-denoise'' construction.  
The rough idea is to first project $\widehat{p}_{n}(\,\cdot\given x_0)$ onto finite mixtures of Gaussians to get an approximation of the mixing measure $\sum_{k=1}^K \lambda_k G_0(\cdot-m_k(x_0))$ and then employ a careful smooth-denoise step to recover the individual components $(\lambda_k, G_0(\cdot-m_k(x_0)))$, from which we obtain the estimates $\widehat{f}_n$ and $\{\widehat{\lambda}_{n,k}\}_{k=1}^K$ (cf. Figure~\ref{fig:overview}). 
This is where the point of separation $x_0$ is used: By \eqref{eq:defn:npmor}, $p(\,\cdot\given x_0)$ can be interpreted as a finite mixture, and the separation enables estimation of this mixture (see discussion surrounding Assumption~\ref{assumption:point of sep}). 
The following result, as a corollary of Theorem \ref{prop:npmix uniform consistency} applied to $p(\,\cdot\given x_0)$, establishes that this procedure provides uniformly consistent estimation of the mixture model at $x_{0}$.

\begin{prop}\label{prop:step 2 uniform consistency}
Let $\mathcal{V}_{x_0}\subset \mor_{x_0}$ be a subfamily satisfying \ref{assumption:unpmor:regularity}-\ref{assumption:unpmor:separation}. There exist estimators $\widehat{f}_n$, $\{\widehat{\lambda}_{n,k}\}_{k=1}^K$, and a sequence $\varepsilon_{n}\rightarrow 0$ so that 
\begin{align*}
    \underset{\mathcal{V}_{x_0}}{\operatorname{sup}}\,\, \mathbb{P}\left(\|\widehat{f}_n-f\|_1\vee\underset{k}{\operatorname{max}}\,|\widehat{\lambda}_{n,k}-\lambda_k| >\varepsilon_{n}\right)\xrightarrow{n\rightarrow\infty} 0.
\end{align*}
\end{prop}
\noindent

\begin{remark}\label{remark:relabel}
A direct application of Theorem \ref{prop:npmix uniform consistency} implies consistency up to a permutation. Since now the mixing proportions are uniformly separated, so are the estimates $\widehat{\lambda}_{n,k}$'s asymptotically. Hence after a relabelling which sorts both $\{\widehat{\lambda}_{n,k}\}_{k=1}^K$ and $\{\lambda_k\}_{k=1}^K$ in increasing order, the permutation will be the identity uniformly over $\mathcal{V}_{x_0}$ for all large $n$. This additional step is not necessary but is adopted for notational convenience. \qed
\end{remark}

\subsubsection{Estimating the regression functions}\label{sec:step 3}

Now it remains to introduce estimators for the regression functions $m_k$
via the following minimum distance estimator: Let $B$ be a constant so that $\operatorname{max}_k \|m_k\|_{L^{\infty}[a,b]}\leq B$ (this could be chosen based on assumption \ref{assumption:unpmor:regularity} or to be sufficiently large based on data). Define for $x\in[a,b]$
\begin{align}
    (\widehat{m}_{n,1}(x),\ldots,\widehat{m}_{n,K}(x))^T=\underset{\theta\in [-B,B]^K}{\operatorname{arg\,min}}\, \left\|\sum_{k=1}^K \widehat{\lambda}_{n,k}\widehat{f}_n(\cdot-\theta_k)-\widehat{p}_{n}(\,\cdot\given x)\right\|_1, \label{eq:MDE}
\end{align}
where we take any minimizer if there are multiple ones. 
\begin{lemma}\label{lemma:Hausdorff error regression function}
Suppose $\operatorname{min}_k \lambda_k>0$ and $|\mathcal{F}G_0|\geq J$ for some function $J>0$. For each fixed $x\in[a,b]$ we have
\begin{align*}
    \underset{k}{\operatorname{max}}\,\, \underset{j}{\operatorname{min}}\, |\widehat{m}_{n,j}(x)-m_k(x)|\leq \frac{e_n(x)}{\operatorname{min}_k \lambda_k}, 
\end{align*}
where 
\begin{align*}
    e_n(x) = C_B F_{J,\sigma}\Big(\|\widehat{f}_n-f\|_1 + \sum_{k=1}^K |\widehat{\lambda}_{n,k}-\lambda_k|+\|\widehat{p}_{n}(\,\cdot\given x)-p(\,\cdot\given x)\|_1\Big) 
\end{align*}
with $C_B$ a constant depending only on $B$ and $F_{J,\sigma}$ a strictly increasing function depending only on $J,\sigma$ that satisfies $F_{J,\sigma}(d)\xrightarrow{d\rightarrow 0}0$. Furthermore if $\operatorname{min}_{j\neq k}|\lambda_j-\lambda_k|>0$ and 
\begin{align}
    \underset{j\neq k}{\operatorname{min}}\, |m_j(x)-m_k(x)| >\frac{6e_n(x)}{\operatorname{min}_k\lambda_k\wedge \operatorname{min}_{j\neq k}|\lambda_j-\lambda_k|} \label{eq:assump reduction}
\end{align}
then
\begin{align*}
    \underset{k}{\operatorname{max}}\, |\widehat{m}_{n,k}(x)-m_k(x)|\leq \frac{e_n(x)}{\operatorname{min}_k \lambda_k}.
\end{align*}
\end{lemma}
\noindent
The key in proving Lemma \ref{lemma:Hausdorff error regression function} is the following modulus-of-continuity type result (Lemma \ref{lemma:moc}) that slightly generalizes \cite[][Theorem 2]{nguyen2013}: 
\begin{align}
    W_1(\widehat{V},V)\leq C_BF_{J,\sigma}\big(\|\widehat{V}\ast f-V\ast f\|_1\big),\quad \widehat{V}=\sum_{k=1}^K \lambda_k \delta_{\widehat{m}_{n,k}(x)},\quad V=\sum_{k=1}^K \lambda_k \delta_{m_{k}(x)}.  \label{eq:moc remark}
\end{align}
Here, $W_1$ is the 1-Wasserstein distance defined as in \eqref{eq:def wass} and $\delta_a$ is the Dirac delta at a point $a\in\mathbb{R}$.

For each mixed regression model the lower bound function $J$ can be taken as $|\mathcal{F}G_0|$ provided that it never vanishes, but a uniform lower bound as in \ref{assumption:unpmor:FT} is needed for uniform consistency.   In Section~\ref{sec:npmor:cond:FT} we explicitly construct families of compactly supported $G_0$'s satisfying \ref{assumption:unpmor:FT}.

Lemma \ref{lemma:Hausdorff error regression function} establishes error estimates over regions where the $m_k$'s satisfy \eqref{eq:assump reduction}. Since the error $e_n(x)\rightarrow 0$ in probability over $[a+h_n,b-h_n]$ (by Propositions~\ref{prop:step 1 uniform consistency} and~\ref{prop:step 2 uniform consistency}), such regions will eventually be the whole interval $(a,b)$ so that we can achieve consistent estimation in $L^1[a,b]$ norm. Before stating the result, however, we make a further remark on how the assumption of distinct mixing proportions enters our estimation procedure \eqref{eq:MDE}. In particular, we have defined the function estimate $\widehat{m}_{n,k}$ to be the collection of all $\widehat{m}_{n,k}(x)$'s that are associated with $\widehat{\lambda}_{n,k}$. Since the $\lambda_k$'s are distinct, so are the estimates $\widehat{\lambda}_{n,k}$'s asymptotically and hence this defines a unique consistent labelling procedure for the estimates $\{\widehat{m}_{n,k}(x)\}_{k=1}^K$.

By combining these observations, we can prove $L^{1}$ consistency of the resulting regression estimates:

\begin{prop}\label{prop:step 3 uniform consistency}
Let~$\mathcal{V}_{x_0}\subset \mor_{x_0}$ be a subfamily satisfying \ref{assumption:unpmor:regularity}-\ref{assumption:unpmor:sep of reg}. Then the estimators $\{\widehat{m}_{n,k}\}_{k=1}^K$ are uniformly consistent, i.e., for any $\varepsilon>0$
\begin{align*}
    \underset{\mathcal{V}_{x_0}}{\operatorname{sup}}\,\,\mathbb{P}\Big(\underset{k}{\operatorname{max}}\,\|\widehat{m}_{n,k}-m_k\|_{L^1[a,b]}>\varepsilon\Big) \xrightarrow{n\rightarrow\infty} 0.
\end{align*}
\end{prop}
\noindent
It is in this step (i.e., extending estimates for fixed $x$ to $L^1$ consistency) that \ref{assumption:unpmor:sep of reg} is invoked. It can be understood as a uniform control on the regions where \eqref{eq:assump reduction} is not satisfied, i.e., regions where the pointwise estimates are not provably accurate.

\subsection{Discussion of conditions}
\label{sec:npmor:cond:gen}

Most of the conditions in \ref{assumption:unpmor:regularity}-\ref{assumption:unpmor:sep of reg} are easily interpreted, however, \ref{assumption:unpmor:FT}-\ref{assumption:unpmor:sep of reg} are more technical and perhaps a bit opaque. We pause here to discuss these conditions in more detail.
 
\subsubsection{Discussion of Condition~\ref{assumption:unpmor:FT}}
\label{sec:npmor:cond:FT}
The following example provides nonparametric families of functions that satisfy \ref{assumption:unpmor:FT}. Recall $G_0$ as defined in \eqref{eq:convolution assumption}.

\begin{ex}\label{ex:G0}

Let $\ell>0$ be a function. Consider the collection of mixing measures $dG_0=gdx$ with $g$ belonging to 
\begin{align*}
    \mathcal{G}_{\ell}=\left\{g=\frac{(\varphi\ast\varphi)\cdot q}{\int_{\mathbb{R}}(\varphi\ast\varphi)\cdot q\,dx}: \varphi=\frac12 \mathbf{1}_{[-1,1]}, q \text{ is a density such that } \mathcal{F}q\geq \ell \right\}.
\end{align*}
Possible examples of the function $\ell$ are $c\exp(-|t|^{\beta}/\gamma)$ and $c|t|^{-\beta}$ for some constants $c,\beta,\gamma$, which characterize distributions that are supersmooth and smooth of order $\beta$ respectively \citep{fan1991optimal,nguyen2013}. 
Notice that without the constraint of having a compact support, the set of such $q$'s already satisfies \ref{assumption:unpmor:FT}. The additional term $\varphi\ast \varphi$ can be interpreted as a truncation to make $q$ compactly supported. We see that the family $\mathcal{G}_{\ell}$ satisfies $\operatorname{inf}_{g\in\mathcal{G}_{\ell}} |\mathcal{F}g|\geq J$ for some function $J>0$. Indeed, notice that
\begin{align*}
    \int_{\mathbb{R}} \varphi\ast \varphi \cdot q  \,  dx   \leq \|\varphi\ast \varphi\|_{\infty}
\end{align*}
since $q$ is a density, so that 
\begin{align*}
    |\mathcal{F}g(w)| \geq \frac{1}{\|\varphi\ast \varphi\|_{\infty}}\left| \big((\mathcal{F}q)\ast (\mathcal{F}\varphi)^2\big)(w) \right|\geq \frac{1}{\|\varphi\ast \varphi\|_{\infty}} \int_{\mathbb{R}}\ell(w-t)\frac{\sin^2(t)}{t^2}dt=:J>0. 
\end{align*} 
  We remark that $\varphi\ast \varphi$ can be replaced with any bounded nonnegative compactly supported function whose Fourier transform is also nonnegative. 
\qed 
\end{ex}

\begin{ex}\label{ex:family G}
We can modify Example \ref{ex:G0} as follows to drop the assumption that $q$ is a density. 
Let $0<\ell\leq u\in L^2(\mathbb{R})$ be two functions. Consider the collection of mixing measures $dG_0=g\,dx$ with $g$ belonging to
\begin{align*}
    \mathcal{G}_{\ell,u}=\left\{g=\frac{(\varphi\ast \varphi)\cdot q}{\int_{\mathbb{R}}(\varphi\ast \varphi)\cdot q \,dx}: \varphi=\frac12\mathbf{1}_{[-1,1]},\,\,  q=|\mathcal{F}^{-1}v|^2 \,\,\operatorname{with}\,\, \ell\leq v\leq u  \right\}, 
\end{align*}
where $|\cdot|$ denotes the modulus of a possibly complex number.
Similarly as in Example \ref{ex:G0}, it suffices to show that $\mathcal{F}q$ is uniformly lower bounded by a positive function and that the denominator uniformly upper bounded. This follows by noticing that
$q(x)=(\mathcal{F}^{-1}v)(x)(\mathcal{F}^{-1}v)(-x) $ since $v$ is real,
and 
\begin{align*}
    (\mathcal{F}q)(w)=\int_{\mathbb{R}} v(w-t)v(-t)dt \geq \int_{\mathbb{R}} \ell(w-t)\ell(-t)dt >0.
\end{align*}
Moreover, we have
\begin{align*}
    \int_{\mathbb{R}}(\varphi\ast \varphi)\cdot q \,dx\leq \|\varphi\ast\varphi\|_{\infty}\int_{\mathbb{R}}q\,dx 
    &\leq \|\varphi\ast\varphi\|_{\infty} \int_{\mathbb{R}} |\mathcal{F}^{-1}v|^2dx \\
    &= \|\varphi\ast\varphi\|_{\infty}\int_{\mathbb{R}} v^2dx \leq \|\varphi\ast\varphi\|_{\infty}\|u\|_2^2. 
\end{align*}
\qed
\end{ex}

\subsubsection{Discussion of Condition~\ref{assumption:unpmor:sep of reg}}
\label{sec:npmor:cond}

Roughly speaking, Condition~\ref{assumption:unpmor:sep of reg} prevents different regression functions from becoming arbitrarily close over an entire interval. 
For a single model where the $m_k$'s only intersect at countably many points, we have 
\begin{align*}
    \left\{x: \exists \,j\neq k ,\,  |m_j(x)-m_k(x)|\leq a_n \right\}\searrow \left\{x:\exists \,j\neq k ,\,  m_j(x)=m_k(x)\right\},
\end{align*}
which has measure zero and therefore \ref{assumption:unpmor:sep of reg} can be seen as a uniform analog of this fact. The reason why we need such assumption has been foreshadowed in Lemma \ref{lemma:Hausdorff error regression function} above in that we need uniform control on the size of the set where \eqref{eq:assump reduction} fails. 

To better understand Condition~\ref{assumption:unpmor:sep of reg}, let us further introduce sufficient conditions on the $m_k$'s that guarantee it. The first observation is that \ref{assumption:unpmor:sep of reg} is satisfied if the $m_k$'s are uniformly separated, 
however, this excludes intersections of the $m_k$'s and is less interesting. The following example generalizes this to allow intersections. 

\begin{ex}\label{ex:C5}
For $j\neq k$, let $Z_{jk}=\{x:m_j(x)=m_k(x)\}$ and $Z_{jk}(\delta)=\{x:\operatorname{dist}(x,Z_{jk})<\delta\}$. Suppose there exists $N\in\mathbb{N},\delta,\eta>0$ so that for all $\vartheta\in\mathcal{U}_{x_0}$
\begin{equation}\label{eq:uniform transversality assumption}
    \begin{aligned}
    &\underset{j\neq k}{\operatorname{max}}\,\operatorname{Card}(Z_{jk})\leq N,
    \qquad   \underset{j\neq k}{\operatorname{min}}\, \underset{x\in Z_{jk}(\delta)^c}{\operatorname{inf}}\, |m_j(x)-m_k(x)|\geq \eta,
     \\&\qquad \qquad\qquad \underset{j\neq k}{\operatorname{min}}\, \underset{x\in Z_{jk}(\delta)}{\operatorname{inf}}\, |m^{\prime}_j(x)-m^{\prime}_k(x)|\geq \eta. 
\end{aligned}
\end{equation}
Lemma \ref{lemma:uniform transversality} shows that \eqref{eq:uniform transversality assumption} implies \ref{assumption:unpmor:sep of reg}. These conditions essentially require uniform separation of the $m_k$'s away from the points of intersection and uniform separation of the derivatives $m_k^{\prime}$'s near points of intersection. The latter will guarantee that two regression functions do not stay close over a large interval when they intersect. \qed 
\end{ex}

The conditions in \eqref{eq:uniform transversality assumption} can be interpreted as a uniform version of transversality introduced in \cite{huang2013}.   We remark that although sufficient, uniform transversality is \emph{not} necessary for \ref{assumption:unpmor:sep of reg} to be satisfied as illustrated by the following example.

\begin{ex}
\label{ex:C6}
Consider the family of pairs of functions $\mathscr{R}_{P,C}=\{(0,cx^{p}): x\in[0,1], 0\leq p\leq P, \,c\geq C\}$ for some fixed  $P,C>0$. 
First, observe that for any $P\geq 1,C>0$, $\mathscr{R}_{P,C}$ violates the transversality condition and \eqref{eq:uniform transversality assumption}.
Next, if $P\to\infty$ or $C\to 0$, then there are nonzero functions in $\mathscr{R}_{P,C}$ that approximate the zero function \emph{uniformly} over small intervals around the origin, ostensibly in violation of \ref{assumption:unpmor:sep of reg}, but for fixed $P,C>0$ this issue is avoided.
More precisely, for any $a_n\rightarrow 0$ we have when $a_n\leq C$
\begin{align*}
    \underset{\mathscr{R}_{P,C}}{\operatorname{sup}}\, \operatorname{Leb}\left(\{|m_1-m_2|\leq a_n\}\right) \leq \left(\frac{a_n}{C}\right)^{1/P} \xrightarrow{n\rightarrow\infty} 0.
\end{align*}
Therefore the family of two-component mixed regression models with regression functions in $\mathscr{R}_{P,C}$ satisfies \ref{assumption:unpmor:sep of reg}.
This example illustrates a typical situation where the separation around a point of intersection is controlled by the condition $c\geq C$.
\qed
\end{ex}

\subsection{A concrete family} 
\label{sec:npmor:concrete}

We close this section by demonstrating Theorem~\ref{thm:uniform consistency} through a concrete example under which all of the assumptions hold.
As our original goal was to present a (nearly) minimal set of conditions under which uniform consistency is assured, the resulting Conditions~\ref{assumption:unpmor:regularity}-\ref{assumption:unpmor:sep of reg} are somewhat abstract (see also Remark~\ref{rem:diff:mk}).
Using the examples in Section~\ref{sec:npmor:cond:gen}, however, we can now present a more concrete set of assumptions, albeit at the expense of some generality. 

Consider the following family of two-component mixed regression models over the interval $[a,b]$. 
The assumption that $K=2$ is only for simplicity, and this example can easily be generalized to $K>2$.
Define
\begin{align*}
    \mathcal{V}=\big\{(p_X,f,\{\lambda_k\}_{k=1}^2, \{m_k\}_{k=1}^2)\big\}
\end{align*}
as the set of tuples such that for some $x_0\in[a,b]$, some $N\in\mathbb{N}$, and some small $\epsilon_1,\epsilon_2,\xi,\delta,\eta>0$, the following conditions hold:
\begin{enumerate}
    \item $p_X=\operatorname{Unif}[a,b]$.
    \item  $\lambda_1\wedge\lambda_2\geq\epsilon_1$ and $|\lambda_1-\lambda_2|\geq\epsilon_2$.
   \item $|m_1(x_0)-m_2(x_0)|- 2\operatorname{diam}(\operatorname{supp}(G_0))\geq \xi.$ %
    \item $f=\phi_{\sigma}\ast G_0$ with  $G_0$ as in Example \ref{ex:G0}. 
\item The set $Z$ of points of intersection between $m_1$ and $m_2$ has at most $N$ elements, with 
    \begin{align*}
        \underset{x\in Z(\delta)^c}{\operatorname{inf}} |m_1(x)-m_2(x)| \geq \eta, \quad \underset{x\in Z(\delta)}{\operatorname{inf}} |m_1^{\prime}(x)-m_2^{\prime}(x)| \geq \eta,
    \end{align*}
    where $Z(\delta)=\{x:\operatorname{dist}(x,Z)\leq \delta\}$. 
\end{enumerate}
Under these conditions, Theorem~\ref{thm:uniform consistency} applies to the family $\mathcal{V}$:
\begin{cor}\label{cor:combined example}
There exist uniformly consistent estimators for $(f,\{\lambda_k\}_{k=1}^2, \{m_k\}_{k=1}^2)$ over $\mathcal{V}$. 
\end{cor} 
\noindent
We remark that this only serves as one such concrete example: Additional ones can be constructed by combining Examples~\ref{ex:G0}-\ref{ex:C6} in different ways.

\section{Uniformly consistent estimation of mixture models}\label{sec:npmix}
In this section we consider estimation of vanilla nonparametric mixtures as in \eqref{eq:finite mixture model}, thereby completing the missing piece from Section \ref{sec:step 2}. As mentioned, the estimation procedure to be introduced applies to finite mixture models and we shall restrict ourselves to the family $\mix$ defined by \eqref{eq:finite mixture model}. In particular, unlike the previous section, we now allow the mixture components $f_{k}$ to be distinct, and treat the case $f_{k}\equiv f$ (i.e., as in the previous section) for all $k$ as a special case. Recall that $f_k=\phi_{\sigma}\ast G_k$ with $G_k$ having compact support.

The main result in this section is the following. Let $\mathscr{U}\subset \mix$ be a subfamily of tuples $\vartheta=(\{f_k\}_{k=1}^K,\{\lambda_k\}_{k=1}^K,\{\mu_k\}_{k=1}^K)$ satisfying
\begin{enumerate}[label=(D\arabic*)]
    \item $\underset{\vartheta\in\mathscr{U}}{\operatorname{sup}}\, \underset{k}{\operatorname{max}}\,\big[\operatorname{diam}(  \operatorname{supp} (G_k))\vee \mu_k\big]<\infty$. \label{assumption:D1}
    \item $\underset{\vartheta\in\mathscr{U}}{\operatorname{inf}}\, \underset{k}{\operatorname{min}}\, \lambda_k >0$. \label{assumption:D2}
    \item $\underset{\vartheta\in\mathscr{U}}{\operatorname{inf}}\,\Big[\underset{j\neq k}{\operatorname{min}} \, \operatorname{dist} \left(\operatorname{supp} (G_j(\cdot-\mu_j)),\,\operatorname{supp}( G_k(\cdot-\mu_k))\right)- \underset{k}{\operatorname{max}}\,\operatorname{diam}(\operatorname{supp}(G_k))\Big]>0$. \label{assumption:D3}
\end{enumerate}
We have the following result:

\begin{thm}\label{prop:npmix uniform consistency}
 
There exist uniformly consistent estimators $\widehat{\lambda}_{n,k}$ and $\widehat{f}_{n,k}$ for $\mathscr{U}$, where $\mathscr{U}$ is defined by \ref{assumption:D1}-\ref{assumption:D3}. More specifically, there exists a permutation $\Pi= \Pi(\vartheta,n) $ and a sequence $\varepsilon_{n}\rightarrow 0$ such that  
\begin{align*}
    \underset{\mathscr{U}}{\operatorname{sup}}\,\, \mathbb{P} \Big( \underset{k}{\operatorname{max}}\,\left(|\widehat{\lambda}_{n,k}-\lambda_{\Pi(k)}|\vee\|\widehat{f}_{n,k}-f_{\Pi(k)}\|_{1}\right)>\varepsilon_{n}\Big)\xrightarrow{n\rightarrow \infty} 0  .
\end{align*}
\end{thm}
\noindent 
Again Conditions \ref{assumption:D1}-\ref{assumption:D3} can be seen as uniform versions of \ref{assm:npmix:conv}-\ref{assm:npmix:sep}. 
  Note also the similarity between \ref{assumption:D1}-\ref{assumption:D3} and the corresponding conditions \ref{assumption:unpmor:regularity}-\ref{assumption:unpmor:separation} for mixed regression.

The proof of this result follows from the aforementioned project-smooth-denoise construction, which is described in this section.
As with the previous section, before describing the details of the estimation procedure and the project-smooth-denoise construction, we start by illustrating the failure of uniformly consistent estimation for general mixture models.

\subsection{Nonexistence of uniformly consistent estimators }\label{sec:unif:npmix:fail}

Previously, \cite{aragam2020identifiability} proved identifiabilty of nonparametric mixtures under \emph{regularity} and \emph{clusterability} assumptions. While these conditions are quite technical, roughly speaking they amount to assuming that the components $f_{k}$ are well-separated in some probability metric. Under the same conditions, a consistent minimum distance estimator was constructed. Here we argue that without additional assumptions, this estimator \emph{cannot be} uniformly consistent, and indeed, there cannot exist a uniformly consistent estimator. It suffices to construct mixtures that are simultaneously regular and clusterable, but arbitrarily close to being nonregular.

\begin{ex}
Let $f_{j}=(1-\alpha)g_{j}+\alpha h_{j}$ for $j=1,2$, where $g_{j}\sim\normalN(\mu_{j},1)$ and $h_{j}\sim \normalN(\xi_{j},1)$. For simplicity, let $\mu_{2}=-\mu_{1}$ and $\xi_{2}=-\xi_{1}$. 
By choosing $\alpha$ sufficiently small and $\mu_{1}$ sufficiently large, the Hellinger distance between $f_{1}$ and $f_{2}$ can be made arbitrarily large---so that separation is not an issue---and by taking $\xi_{1}\to 0$, this model collapses into the nonregular model described in Example~9 of \cite{aragam2020identifiability}. In other words, $F=\lambda_{1}f_{1}+\lambda_{2}f_{2}$ can be made arbitrarily close to a nonregular distribution while still being clusterable.\qed
\end{ex}
\noindent
The problem here boils down to the fact that the set of regular mixtures, as defined by \cite{aragam2020identifiability} is dense in the space of probability measures, but not closed. In the following subsections, we show that under fairly general assumptions, convolutional mixtures with disjoint component supports (i.e., \ref{assumption:D3}) avoid this degeneracy and uniform consistency can be rescued.

\subsection{A toy example}\label{sec:npmix:toy}
To start with, let's first illustrate the main idea behind the project-smooth-denoise procedure through a simple two-component mixture model. We shall focus on the high-level intuition and defer the precise details of the construction to Section \ref{sec:npmix:est}. 

Recall Figure~\ref{fig:overview} and consider the convolutional Gaussian mixture
\begin{align*}
    p=\phi_{\sigma}\ast G, \quad \quad G=\lambda \mathbf{1}_{[-3,-2]} + (1-\lambda)\mathbf{1}_{[2,3]},\quad \quad \lambda>0. 
\end{align*}
The underlying mixing measure $G$ is a mixture of two uniform distributions that satisfies the separation condition \ref{assm:npmix:sep}. The sets $S_1$ and $S_2$ in Figure~\ref{fig:true} correspond to the supports $[-3,-2]$ and $[2,3]$ of $G$'s components.
Suppose we have constructed an estimator $\widehat{p}_n$ of $p$. 
 
\subsubsection{The ``project'' step}
\label{sec:step:project}
The first step of the procedure is to estimate the overall mixing measure $G$ by projecting $\widehat{p}_n$ onto the space of finite mixtures of Gaussians. In other words, we are searching for a discrete measure $$\widehat{G}_n=\sum_{\ell=1}^{L_n} w_{\ell} \delta_{a_{\ell}}$$ 
so that $\phi_{\sigma}\ast \widehat{G}_n$ is as close as possible to $\widehat{p}_n$, and hence also $p=\phi_{\sigma}\ast G$.
Here $\delta_a$ is the Dirac delta at $x=a$, $w_{\ell}$ are nonnegative weights that sum to one, and $L_n$ is a number growing to infinity with $n$. It turns out that $\widehat{G}_n$ constructed above indeed approximates $G$, in the sense of Lemma \ref{lemma:W distance between G hat and G} below. The remaining task is then to cluster the atoms of $\widehat{G}_n$ so that each cluster approximates a corresponding component of $G$. 

The separation condition \ref{assm:npmix:sep} provides a hint as how to accomplish this goal. Ideally, the atoms $a_{\ell}$ would all lie in the support of $G$, $[-3,-2]\cup [2,3]$, which is a union of two well-separated intervals and most clustering algorithms would find the correct assignment. However, since $\widehat{p}_n$ does not necessarily share the same convolutional structure as $p$, this is not guaranteed. In fact, we can only say that most of the atoms lie near the support of $G$, in the sense of Lemma \ref{lemma:total weights of outlier}. An intuitive picture is shown in Figure \ref{fig:project}, where  some potential outliers (marked in red) can be outside the support of $G$. This voids the use of na\"ive clustering schemes and motivates a crucial smooth-denoise step.

\subsubsection{The ``smooth-denoise'' step}
\label{sec:step:smoothdenoise}
The good news is that these outlier atoms carry a small total weight (Lemma \ref{lemma:total weights of outlier})---otherwise $\widehat{G}_n$ cannot converge to $G$ asymptotically. However, there may also be ``good'' atoms inside the support of $G$ with small weights, so that simply eliminating atoms with small weights does not work.
It turns out that we can remove outliers by locating the high density regions 
of a smoothed version of $\widehat{G}_n$. The intuition is that the smoothing step combines information on both the size of the $w_{\ell}$'s and the locations of the $a_{\ell}$ that leads to a viable procedure. A visualization is provided in Figures~\ref{fig:smooth}-\ref{fig:denoise}, where the yellow curve represents the smoothed $\widehat{G}_n$ (Figure~\ref{fig:smooth}) and a thresholding step suffices to locate the high density regions $\widehat{E}_1\cup\widehat{E}_2$ (Figure~\ref{fig:denoise}). 

The key ingredient is to set the threshold (Lemma \ref{lemma:thresholding and Ek hat}) so that  $\widehat{E}_1$ and $\widehat{E}_2$ are well-separated but are both nonempty, so that we can recover them from their union, say by performing single-linkage clustering on the subintervals in $\widehat{E}_1\cup\widehat{E}_2$. We are almost done by considering the estimators $\widehat{G}_n(\,\cdot\given \widehat{E}_1)$ and $\widehat{G}_n(\,\cdot\given \widehat{E}_2)$ as approximating the two components of $G$. However, it could be the case that $\widehat{E}_1$ and $\widehat{E}_2$ are missing some nonnegligible parts of the support of $G$ due to the thresholding step. We can resolve this issue by extending $\widehat{E}_1$ and $\widehat{E}_2$ to a partition of $\mathbb{R}$, which gives the sets $E_1,E_2$ (red) as in Figure \ref{fig:denoise}. It turns out that the $E_k$'s can be seen as a Voronoi tessellation of $\mathbb{R}$ based on the $\widehat{E}_k$'s, i.e., $E_k=\{x\in\mathbb{R}: \operatorname{dist}(x,\widehat{E}_k)\leq \operatorname{dist}(x,\widehat{E}_j) \,\,\operatorname{for\,\,all}\,\, j\neq k\}$. Finally the components of $G$ are estimated by $\widehat{G}_n(\,\cdot\given E_1)$ and $\widehat{G}_n(\,\cdot\given E_2)$.

\subsection{Construction of estimator}
\label{sec:npmix:est}
With the intuition above in mind, we shall now introduce the full details of our estimation procedure.  Readers interested in skipping to the main result of this section are referred to Proposition~\ref{lemma:conditioning}.
Recall the family $\mix$ defined in \eqref{eq:finite mixture model}. In this setting, we can rewrite the density $\sum_{k=1}^K \lambda_k f_k(\cdot-\mu_k)$ as 
\begin{align}
    p=\phi_{\sigma}\ast \underbrace{\sum_{k=1}^K \lambda_k G_k(\cdot-\mu_k)}_{:=G}=\phi_{\sigma}\ast G \label{eq:finite mixture density}
\end{align}
where   the mixing measure $G$ satisfies \ref{assm:npmix:sep}  
\begin{align*}
    \underset{j\neq k}{\operatorname{min}}\, \operatorname{dist} \left(\operatorname{supp}( G_j(\cdot-\mu_j)),\,\operatorname{supp}( G_k(\cdot-\mu_k))\right) > \underset{k}{\operatorname{max}}\, \operatorname{diam}(G_k).
\end{align*}
Let $\widehat{p}_n$ be a density estimator of $p$ in \eqref{eq:finite mixture density} satisfying $\mathbb{E}\|\widehat{p}_n-p\|_1\rightarrow 0$ (see e.g. Appendix \ref{sec:kde uniform L1}). We shall first approximate $\widehat{p}_n$ by finite mixtures of Gaussians. Let $M>0$ be a constant such that $\operatorname{supp}(G)\subset [-M,M]$ ($M$ can be chosen based on \ref{assumption:D1} or sufficiently large based on data) and consider  
\begin{align}
    \widehat{Q}_{n}&=\underset{Q\in \mix_{L_n}}{\operatorname{arg\,min}}\, \|Q-\widehat{p}_{n}\|_1,\label{eq:Q_L,n}
\end{align}
where
\begin{align*}
    \mix_{L_n} &=\Big\{Q=\sum_{\ell=1}^{L_n} w_{\ell}\mathcal{N}(a_{\ell},\sigma^2): a_{\ell}\in [-M,M] \,\,,   w_{\ell}\geq 0,\,\,  \forall \ell,\,\, \sum_{\ell=1}^{L_n} w_{\ell}=1\Big\}
\end{align*}
and $L_n$ is any sequence of positive integers converging to infinity. As the set $\mix_{L_n}$ is nonconvex, we let $\widehat{Q}_{n}$ be any minimizer if there is more than one. In particular, we can write
\begin{align*}
    \widehat{Q}_{n}=\phi_{\sigma}\ast \sum_{\ell=1}^{L_n} w_{\ell} \delta_{a_{\ell}}=: \phi_{\sigma} \ast\widehat{G}_{n}.
\end{align*}
The following lemma shows that $\widehat{G}_{n}$ is an approximation of $G$ and quantifies the approximation error in the $W_1$ metric defined as in \eqref{eq:def wass}: 
\begin{lemma}\label{lemma:W distance between G hat and G}
Let $Q_n=\operatorname{arg\,min}_{Q\in \mix_{L_n}} \|Q-p\|_1$. We have 
\begin{align*}
    W_1(\widehat{G}_{n},G)\leq C_M \big[-\log \|\widehat{Q}_n-p\|_1\big]^{-1/2}\leq C_M \left[-\log (\|Q_n-p\|_1+\|\widehat{p}_n-p\|_1)\right]^{-1/2},
\end{align*}
where $C_M$ is a constant depending only on $M$. 
\end{lemma}
Here $Q_{L_n}$ is the projection of the true density onto $\mix_{L_n}$ and the error $\|Q_{n}-p\|_1$ goes to zero as $L_n$ goes to infinity. This is the so called saturation rate, which has been studied extensively \citep{genovese2000rates,ghosal2001entropies}, but for completeness we include a proof for our $L^1$ case in Lemma \ref{lemma:approx mixing measure}. In particular, this together with the fact that $\mathbb{E}\|\widehat{p}_n-p\|_1\rightarrow 0$ implies  $W_1(\widehat{G}_{n},G)\rightarrow 0$ in probability. 

The next step is then to group the atoms of $\widehat{G}_{n}$ so that each cluster approximates precisely one of the $G_k(\cdot-\mu_k)$'s. We note that if the $a_{\ell}$'s lie exactly in $\operatorname{supp}(G)$, then the separation condition \ref{assm:npmix:sep} would allow us to group them simply by single linkage clustering. However, this is not known a priori due to the fact that $\widehat{Q}_{n}$ is defined by projecting the density estimator $\widehat{p}_n$, which does not necessarily have the same structure as $p$. In fact the best one can say is that ``most'' of the atoms lie in the support of $G$ in the sense of the following lemma: 
\begin{lemma} \label{lemma:total weights of outlier}
Let $\eta>0$ and $A_{\eta}=\{\ell: \operatorname{dist}(a_{\ell},\, \operatorname{supp}(G))>\eta\}$. Then $\sum_{\ell\in A_{\eta}}w_{\ell}\leq  \eta^{-1}W_1(\widehat{G}_{n},G).$
\end{lemma}
Lemma \ref{lemma:total weights of outlier} says that the total weights of those $a_{\ell}$'s that are of distance $\eta$ away from $\operatorname{supp}(G)$ are small, and goes to zero if we choose $\eta$ to be a sequence converging to zero slower than $W_1(\widehat{G}_{n},G)$. Therefore instead of directly clustering the $a_{\ell}$'s, we shall introduce a thresholding step to get rid of these potential ``outliers''. A simple thresholding on the weights $w_{\ell}$ may not work, however, since each individual weight could still be very small (think of the case where $w_{\ell}=1/L_n$ for all $\ell$). Instead we borrow ideas from density-based clustering by lifting the mixing measure $\widehat{G}_{n}$ to a density and searching for its high density regions (the ``smooth'' step). As an overview of our strategy below, we will first obtain after a suitable thresholding a family of ``preliminary sets'' $\widehat{E}_k$'s that roughly locate each support, defined as
\begin{align}
\label{eq:def:supp}
    S_k:=\operatorname{supp}(G_k(\cdot-\mu_k)).
\end{align}
This is the ``denoise'' step.
These sets are then used to construct a partition $\{E_k\}_{k=1}^K$ of $\mathbb{R}$ satisfying $S_k\subset E_k$ and we shall approximate each $G_k(\cdot-\mu_k)$ by $\widehat{G}_n(\,\cdot\,|E_k)$.

To start with, let $I=\frac12\mathbf{1}_{[-1,1]}$ with $I_{\delta}(\cdot)=\delta^{-1}I(\cdot/\delta)$ and define $\widehat{g}_{n}=\widehat{G}_{n}\ast I_{\delta_n}$, where $\delta_n\rightarrow 0$ is a sequence to be determined. The following result gives a thresholding on $\widehat{g}_{n}$ that allows one to locate the $S_k$'s. Let $D=\operatorname{max}_k \operatorname{diam}(S_k)$ and $S_k(\eta)=\{x: \operatorname{dist}(x,S_k)<\eta\}$ be the $\eta$-enlargement of $S_k$.
\begin{lemma}\label{lemma:thresholding and Ek hat}
Let $t_{n}\geq 2^{-1}\delta_n^{-2}W_1(\widehat{G}_{n},G)$.   If 
\begin{align}
    3\delta_n<\underset{j\neq k}{\operatorname{min}}\, \operatorname{dist}(S_j,S_k)\label{eq:thresholding assumption 0}
\end{align}
  and 
\begin{align}
    (D+4\delta_n)t_{n}+\delta_n^{-1}W_1(\widehat{G}_{n},G)<\operatorname{min}_k \lambda_k \label{eq:thresholding assumption}
\end{align}
then the level set $\{x: \widehat{g}_n(x)>t_n\}$ can be partitioned into $K$ sets $\widehat{E}_k$ as follows:
\begin{enumerate}
\item $\{x: \widehat{g}_n(x)>t_n\}=\biguplus_{k=1}^K \widehat{E}_k$ , where $\uplus$ denotes disjoint union;
\item $\widehat{E}_k\neq \emptyset$;
\item $\widehat{E}_k\subset S_k(2\delta_n)$.
\end{enumerate}
\end{lemma}
\noindent 
This result suggests that under \eqref{eq:thresholding assumption 0} and \eqref{eq:thresholding assumption}, the high density regions are localized around the true support $S_k$'s. In particular, the assumption \ref{assm:npmix:sep} implies that as $n\rightarrow\infty$
\begin{align*}
    \underset{j\neq k}{\operatorname{min}}\, \operatorname{dist} (S_j(2\delta_n),\,S_k(2\delta_n)) &=\underset{j\neq k}{\operatorname{min}}\, \operatorname{dist} (S_j,\,S_k)-4\delta_n\\
    &> \underset{k}{\operatorname{max}}\,\operatorname{diam}(S_k)+4\delta_n= \underset{k}{\operatorname{max}}\, \operatorname{diam}(S_k(2\delta_n))
\end{align*}
so that the separation between the $\widehat{E}_k$'s are eventually larger than their diameters. Therefore single linkage clustering recovers the $\widehat{E}_k$'s and in particular locates the $S_k$'s.   As in the proof of Theorem \ref{prop:npmix uniform consistency}, we can set $\delta_n=d_n^{1/4}$ and  $t_n=d_n^{1/2}$, where $d_n$ is a (high probability) uniform upper bound of $W_1(\widehat{G}_n,G)$ given by Lemma \ref{lemma:W distance between G hat and G} that converges to zero. These together with  assumptions \ref{assumption:D1} and \ref{assumption:D2} imply that \eqref{eq:thresholding assumption 0} and \eqref{eq:thresholding assumption} are satisfied uniformly over $\mathscr{U}$ when $n$ is large, meaning that the construction is also uniform. In other words, the choices of $t_n,d_n$ and a sufficient sample size $n$ can be determined based on the family $\mathscr{U}$ without the need to tune each individual model.

However, the $\widehat{E}_k$'s could still be much smaller than the true supports and the following modification uses these $\widehat{E}_k$'s  to construct another collection of sets $\{E_k\}_{k=1}^K$ that properly covers the supports. 

\begin{lemma}\label{lemma:Ek}
Let $\xi>0$ be so that 
\begin{align}
\label{eq:lemma:Ek}
    \underset{j\neq k}{\operatorname{min}}\,\operatorname{dist}(S_j,S_k)> D+4\xi.
\end{align}
If $2\delta_n<\xi$, then there exists a partition $\{E_k\}_{k=1}^K$ of $\mathbb{R}$ constructed from $\{\widehat{E}_k\}_{k=1}^K$ so that $S_k(\xi)\subset E_k$. 
\end{lemma}

\noindent 
  The construction is detailed in its proof and is equivalent to a Voronoi tessellation of $\mathbb{R}$ based on the $\widehat{E}_k$'s.   Notice that \eqref{eq:lemma:Ek} is implied by \ref{assm:npmix:sep} and is satisfied uniformly over $\mathscr{U}$ for some $\xi>0$ due to \ref{assumption:D3}.
The sets $\{E_k\}_{k=1}^K$ also form a partition for the atoms of $\widehat{G}_n$ and we shall consider $\widehat{G}_n(\,\cdot\given E_k)$ as an approximation of $G_k(\cdot-\mu_k)$. More precisely, we define our estimators as 
\begin{align}
\label{eq:def:npmixest}
        \widehat{f}_{n,k}&=\widehat{F}_{n,k}(\cdot+\widehat{\mu}_{n,k}),\quad \quad \widehat{\lambda}_{n,k}=\widehat{G}_{n}(E_k),
\end{align}
where 
\begin{align*}
    \widehat{F}_{n,k}= \phi_{\sigma}\ast \widehat{G}_{n}(\,\cdot\given E_k),\quad \quad \widehat{\mu}_{n,k}=\int_{\mathbb{R}} x \widehat{F}_{n,k}(x)dx.
\end{align*}
Notice that $\widehat{F}_{n,k}$ is supposed to approximate $f_k(\cdot-\mu_k)$ and hence we need an additional shifting step to recover $f_k$. The key building block is the following result. We remark that in practice we can only recover the $E_k$'s up to a permutation, as we have presented in Theorem \ref{prop:npmix uniform consistency}. 
\begin{prop}\label{lemma:conditioning}
Suppose $S_k(\xi)\subset E_k$ for some $\xi>0$ and $\widehat{\lambda}_{n,k}$, $\widehat{f}_{n,k}$ are defined as in \eqref{eq:def:npmixest}. Then
\begin{align*}
    \underset{k}{\operatorname{max}}\, \left(|\widehat{\lambda}_{n,k}-\lambda_k|\vee \|\widehat{f}_{n,k}-f_k\|_1\right) \leq  C_{M,\xi,\sigma} (\alpha_{n}+\sqrt{\alpha_{n}}),
\end{align*}
where $C_{M,\xi,\sigma}$ is a constant depending only on $M,\xi,\sigma$ and 
\begin{align*}
    \alpha_{n}=(\underset{k}{\operatorname{min}}\, \lambda_k)^{-1}\big(-\log \|\widehat{Q}_n-p\|_1 \big)^{-1/2}.
\end{align*}
\end{prop}

\begin{remark}
Proposition \ref{lemma:conditioning} can be seen as a modulus-of-continuity result in the sense that the component-wise error $|\widehat{\lambda}_{n,k}-\lambda_k|$ and $\|\widehat{f}_{n,k}-f_k\|_1$ is dominated by the overall error $\|\widehat{Q}_n-p\|_1$, which together with Lemma \ref{lemma:W distance between G hat and G} gives consistent estimators. The main technical difficulty comes from the need to bound $W_1(\widehat{G}_n(\,\cdot\given E_k), G(\,\cdot\given E_k))$ in terms of $W_1(\widehat{G}_n,G)$ as the support of $G$ is not necessarily discrete. We circumvent this issue by exploiting the relation between Wasserstein and total variation distances \cite[see e.g.][Theorem 6.15]{villani2008optimal} and noticing that 
\begin{align*}
    \|\widehat{G}_n(E_k)\widehat{G}_n(\,\cdot\given E_k)-G(E_k)G(\,\cdot\given E_k)\|_1&=\int_{E_k}d|\widehat{G}_n-G|(\theta)\nonumber\\
    &\leq \int_{\mathbb{R}}d|\widehat{G}_n-G|(\theta)=\|\widehat{G}_n-G\|_1, %
\end{align*}
where we have abused the notation to treat $\widehat{G}_n$ and $G$ as densities. 
\qed
\end{remark}

\begin{remark}
The idea to approximate $p$ with a mixture of Gaussians was introduced in \cite{aragam2020identifiability} and used to construct a pointwise consistent estimator. The construction used there fails to provide a uniformly consistent estimator, and the reason is explained by the smoothing and denoising steps described in Lemmas~\ref{lemma:total weights of outlier}, \ref{lemma:thresholding and Ek hat}, and~\ref{lemma:Ek} above. Without the additional structure of the convolutional model $f_{k}=\phi_{\sigma}\ast G_{k}$, it is difficult to control the ``outlier'' atoms in $\wh{G}_{n}$ uniformly. A key technical step in our analysis is to bridge this gap by first lifting $\wh{G}_{n}$ to a density before thresholding, and then modifying the resulting clusters using the separation condition \eqref{eq:lemma:Ek}. \qed
\end{remark}

\subsection{Discussion of Condition~\ref{assumption:D3}}
\label{sec:npmix:cond}

We end this section with a discussion on the separation condition \ref{assumption:D3}, which is the main structural assumption in this section. Condition \ref{assumption:D3} ensures that \eqref{eq:lemma:Ek} is satisfied uniformly over $\mathscr{U}$ for some $\xi>0$, which is an important step in guaranteeing the uniformity of the estimation procedure. To simplify the discussion we shall focus on \ref{assm:npmix:sep} in the following.

The rationale of this assumption is that it is the minimal requirement that allows single linkage clustering on the supports $G_k(\cdot-\mu_k)$'s to correctly identify each of them. In particular, if say each $G_k(\cdot-\mu_k)$ consists only of two atoms that are of distance $D$ apart, then it is necessary that the inter-support distance to be larger than $D$. However, if the mixing measure $G$ has certain structure (such as admitting a continuous density), then this additional information can be used to identify the supports under weaker conditions than \ref{assm:npmix:sep}. Below we present one such result, which removes the dependence on the maximum diameter in the lower bound:
\begin{lemma}\label{lemma:g continuous density}
Suppose $G$ has a density $g$ such that $b\leq g\leq B$   over its support   for some $b,B>0$. Suppose that $g$ is H\"older continuous of order $\beta\in(0,1)$ in its support. Suppose further that 
\begin{align}
\label{eq:lemma:g continuous density}
    \underset{j\neq k}{\operatorname{min}}\, \operatorname{dist}(S_j,S_k)=:4\xi>0.
\end{align}
The for $n$ large enough one can construct sets $\{E_k\}_{k=1}^K$ so that $S_k(\xi)\subset E_k$. 
\end{lemma}
\noindent
In comparison with \eqref{eq:lemma:Ek}, the lower bound in \eqref{eq:lemma:g continuous density} no longer depends on $D$ and can be arbitrarily small. 
The idea is that since $g$ is uniformly bounded below, its support automatically splits into $K$ connected components even under the weakest separation assumption. The construction of the $E_k$'s then follows the same steps as above. Finally the conclusion $S_k(\xi)\subset E_k$ allows direct application of Proposition \ref{lemma:conditioning} to yield consistent estimators.

\section{Generalizations}\label{sec:gen}
 
We have deliberately focused on a simple setting with univariate covariates and responses as well as Gaussian convolutions in order to emphasize an important point: Uniform consistency is a difficult matter \emph{even under stronger assumptions in a simplified setting}. That is, the difficulties are intrinsic to mixtures, and not other concerns such as the curse of dimensionality or regularity. In this section, we pause to discuss various generalizations that are of interest in practice.

\subsection{Different error densities}
If we drop the requirement of uniformity, then these results can be generalized to different, possibly non-convolutional error densities by using a simple distance-based estimator similar to \eqref{eq:MDE}. This generalization is also the basis of Proposition~\ref{prop:pointwise consistency convolutional} in Section~\ref{sec:consistent estimation special case}. We briefly outline this result here.

Before describing the generalization, let us emphasize how these assumptions were used in the previous sections:
The construction in Section~\ref{sec:npmor} crucially relies on the existence of uniformly consistent estimators of 
the mixture model at $x_{0}$, defined by $\lambda_{k}$ and $f$. In order to define uniformly consistent estimators $(\widehat{\lambda}_{k},\widehat{f})$ for $(\lambda_{k},f)$, in Section~\ref{sec:npmix} we carefully exploited the structure of the convolutional Gaussian model and more specifically, the assumption \ref{assumption:D3}. 
In this step of identifying (and estimating) the error densities and mixing proportions we have always allowed the error densities to be different---it is only in the step of inferring the regression functions (via Lemma~\ref{lemma:Hausdorff error regression function}) that we impose the additional restriction that the error densities are the same. The reason is that if different error densities are allowed, then the mixture at each $x$ is no longer from a simple translation family (i.e., $\{f(\cdot-\mu)\}_{\mu\in\mathbb{R}}$) but $\{f_k(\cdot-\mu)\}_{k\in[K],\mu\in\mathbb{R}}$ whose identifiability is more subtle. 

To extend the identifiability results in Section~\ref{sec:model} to different, possibly non-convolutional error densities $f_{k}$, we make use of the following result, whose proof is based on a classical result due to \cite{teicher1963}:
\begin{lemma}\label{lemma:identifiability K translation}
Suppose $\operatorname{min}_{j\neq k}|\lambda_j-\lambda_k|>0$ and  $\{f_k\}_{k=1}^K$ is a collection of densities so that
\begin{align}
    \underset{t\rightarrow\infty}{\operatorname{lim}}\, \frac{\varphi_{j}(t)}{\varphi_k(t)}\in \{0,\infty\}\quad \quad \forall j\neq k, \label{eq:assumption independence}
\end{align}
where $\varphi_k$ is the characteristic function of $f_k$. Then the equality
\begin{align}
    \sum_{k=1}^K \lambda_k f_k(y-\mu_k)= \sum_{k=1}^K \lambda_k f_k(y-m_k) \quad \quad \forall y\in\mathbb{R} \label{eq:identifiability K translation}
\end{align}
implies $\mu_k=m_k$ for all $k$. 
\end{lemma}

With this in mind, consider the following estimator:
\begin{align}
    (\widehat{m}_{n,1}(x),\ldots,\widehat{m}_{n,K}(x))^T
    \in \underset{\theta\in [-B,B]^K}{\operatorname{arg\,min}}\, \Big\|\sum_{k=1}^K \widehat{\lambda}_{n,k}\widehat{f}_{n,k}(\cdot-\theta_k)-\widehat{p}_{n}(\,\cdot\given x)\Big\|_1, \label{eq:MDE2}
\end{align}
where $B$ is a constant chosen so that $\operatorname{max}_k \|m_k\|_{L^{\infty}[a,b]}\leq B$ and $(\widehat{p}_{n}(\,\cdot\given x),\{\widehat{\lambda}_{n,k}\}_{k=1}^K,\{\widehat{f}_{n,k}\}_{k=1}^K)$ are (pointwise) 
consistent estimators satisfying with probability one
\begin{align}
    \|\widehat{p}_{n}(\,\cdot\given x)-p(\,\cdot\given x)\|_1\xrightarrow{n\rightarrow\infty} 0 \quad \quad \forall x\in(a,b) \label{eq:generalization consistency 1}\\
    \underset{k}{\operatorname{max}}\,\left(|\widehat{\lambda}_{n,k}-\lambda_k| \vee \|\widehat{f}_{n,k}-f_k\|_1\right) \xrightarrow{n\rightarrow\infty} 0.  \label{eq:generalization consistency 2}
\end{align}
In Appendix \ref{sec:pointwise consistency CDE} we include a proof of \eqref{eq:generalization consistency 1} for a standard kernel density estimator and in Remark~\ref{rem:prev:est} we discuss various examples of consistent estimators satisfying \eqref{eq:generalization consistency 2}.   
We will also assume the following natural generalization of \ref{assumption:joint:holder} to different $f_{k}$:
\begin{enumerate}[label=(A$5^{\prime}$)]
    \item\label{assumption:joint:holder:gen} The joint density $p_X(x)\sum_{k=1}^K \lambda_kf_k(y-m_k(x))$ is $\beta$-H\"older continuous for some $\beta>0$. 
\end{enumerate}
Then we have the following result:
\begin{prop}\label{prop:pointwise consistent regression function}
With probability one 
\begin{align*}
    \underset{k}{\operatorname{max}}\, \|\widehat{m}_{n,k}-m_k\|_{L^{1}[a,b]}\xrightarrow{n\rightarrow \infty} 0,
\end{align*}
where $\widehat{m}_{n,k}$ are given by \eqref{eq:MDE2}.
\end{prop}
\noindent
To prove Proposition~\ref{prop:pointwise consistent regression function}, we rely on the following crucial fact (Lemma \ref{lemma:pointwise consistency convergence}): Under the assumptions of Lemma~\ref{lemma:identifiability K translation}, the minimum-distance estimator is asymptotically unique. The proof can be found in Appendix~\ref{appen:proof of pointwise consistency}.

\begin{remark}
\label{rem:compare}
In related work, \cite{kitamura2018nonparametric} studied nonparametric identifiability in mixed regression models under similar assumptions to \eqref{eq:assumption independence}. 
Moreover, both this work and ours assume knowledge of a specific point $x_{0}\in\R$ around which ``local mixtures'' can be identified.
While their work is mostly focused on identifiability, they also define a pointwise consistent estimator for the special case $K=2$ and obtain pointwise rates of convergence for this estimator. 
While our proofs do imply rates for our estimators, we have made no attempt to optimize these upper bounds. Moreover, it is known that pointwise and uniform rates in mixture models can differ even in parametric models; see \cite{heinrich2018strong}. \qed
\end{remark}

\begin{remark}
\label{rem:uniform:difficulty}
We remark that there are two main technical difficulties for boosting this result to uniform consistency. The first comes from the need of uniformly consistent estimators as in \eqref{eq:generalization consistency 1} and \eqref{eq:generalization consistency 2}. As discussed in Section \ref{sec:unif:npmix:fail}, existing estimators \citep[e.g.][]{aragam2020identifiability} are not uniformly consistent. The second lies in the fact that for the minimum distance estimator defined in \eqref{eq:MDE2}, the best one can hope for is recovery of the true parameter on the level of the conditional density. In other words, we can only control the error  $\|\sum_{k=1}^K \lambda_kf_k(\cdot-\widehat{m}_{n,k}(x)) -\sum_{k=1}^K \lambda_kf_k(y-m_k(x))\|_1$ and need a modulus of continuity result to lift such error estimates to that for $|\widehat{m}_{n,k}(x)-m_k(x)|$.  Such results have been proved for the case $f_k\equiv f$ in \cite{heinrich2018strong} but remains open when the $f_k$'s could be different. We leave such investigations for future work.  \qed
\end{remark}

\begin{remark}
\label{rem:prev:est}
We conclude by discussing several cases in which (pointwise) consistent estimators $(\widehat{\lambda}_{n,k},\widehat{f}_{n,k})$ exist for each $k$:
\begin{itemize}
\item \cite{aragam2020identifiability} constructs a consistent estimator based on \emph{clusterability} and \emph{regularity} conditions; roughly speaking these conditions require that the mixture components $f_{k}$ are well-separated;
\item \cite{bordes2006} 
and \cite{hunter2007} 
consider mixtures of symmetric distributions and each construct consistent estimators;
\item \cite{chandra1977mixtures} and \cite{fisher1970estimating} provide general recipes for constructing consistent estimates of identifiable mixing distributions in general settings. See also \cite{rao1992}. 
\end{itemize}
Any of these estimators can be plugged into \eqref{eq:MDE2} and used in Proposition~\ref{prop:pointwise consistent regression function}. \qed
\end{remark}

\subsection{Different points of separation}\label{sec:gen:different x0}
  In our discussion so far, we have assumed the point of separation $x_0$ to be the same for the families of mixed regression models that we have considered. We remark that it is possible to extend our results to allow different points of separation. More precisely, we have the following generalization of Theorem \ref{thm:identifiability convolutional}, whose proof can be found in Appendix~\ref{appen:different sep}. 

\begin{thm}\label{thm:identifiability differnt x0}
The model $\Phi(\cup_{x_0\in[a,b]}\mor_{x_0})$ is identifiable. 
\end{thm}

The main ideas are similar to those for Theorem \ref{thm:identifiability convolutional}, by modifying the proof of Proposition \ref{prop:identifiability convolutional} to account for the possibly different $x_0$'s. Likewise, Theorem \ref{thm:uniform consistency} can be extended to the family $\mathcal{U}$ that satisfies  \ref{assumption:unpmor:regularity}, \ref{assumption:unpmor:lambda}, \ref{assumption:unpmor:FT}, \ref{assumption:unpmor:sep of reg} and 
\begin{enumerate}[label=(C$3^{\prime}$)]
    \item  $\underset{\vartheta\in\mathcal{U}}{\operatorname{inf}}\, \underset{x_0\in[a,b]}{\operatorname{max}}\, \big[|m_j(x_0)-m_k(x_0)|- 2\operatorname{diam}(\operatorname{supp}(G_0))\big]>0$. \label{assumption:C5 prime}
\end{enumerate}
In other words, the condition \ref{assumption:C5 prime} means that each model has a point of separation that can be different from instance to instance, and the amount of separation is uniformly lower bounded over $\mathcal{U}$. The estimation procedure and the proofs in Section \ref{sec:npmor} readily generalize to this setting, if in the step of estimating the error density and mixing proportions we work with the conditional density estimator $\widehat{p}_n(\,\cdot\given x_0)$ at a point of separation $x_0$   (or $\widehat{p}_n(\cdot\,|\,x_0\pm h_n)$ when $x_0$ is near boundary).   However, this does require knowledge of $x_0$ for each model in the class, and a provable procedure for finding such $x_0$ is still an important future direction. Empirically one can search for such points of separation by finding the point that maximizes the separation of the data as in \ref{assumption:C5 prime}. For instance, consider for each $x$ all the data pairs $D_x:=\{(X_i,Y_i): |X_i-x|\leq \delta\}$ for some small $\delta>0$. By running a clustering algorithm on $D_x$ and computing the centers of the resulting clusters, one can obtain a measure of separation $\operatorname{Sep}(x)$ at the point $x$ based on the distances between these centers. The point that maximizes $\operatorname{Sep}(x)$ would then be a candidate for the point of separation.

\subsection{ Other generalizations}\label{sec:gen:other}
Most of our results can be readily generalized to the setting of multivariate covariates $X$ and responses $Y$ with suitable modifications. For instance, Assumption (A4) can be replaced by assuming that the points of intersection have measure zero and partition the space into countably many pieces, since these are the relevant properties used in   the proof of Theorem \ref{thm:identifiability convolutional} \cite[see also][proof of Theorem 1]{huang2013}. The construction of $E_k$'s in Lemma 5.5 is equivalent to a Voronoi tessellation that holds in higher dimensions and continues to give similar guarantees under possibly stronger separation assumptions. Secondly, we have focused on convolutional Gaussian error densities $\phi_{\sigma}\ast G_0$ purely for simplicity:
The proposed estimators can be generalized to other convolutional families $\varphi\ast G_0$ with different source densities $\varphi$. For example, a similar analysis can be carried out under technical assumptions such as $|\mathcal{F}\varphi|>0$, $\|\varphi^{\prime}\|_1<\infty$, $\int x^2\varphi(x)dx<\infty$, together with a saturation result as in Lemma \ref{lemma:approx mixing measure}. Finally, our assumption that the covariate $X$ is supported on a compact interval $[a,b]$ is only a technical one that allows us to establish $L^1[a,b]$ consistency for estimating the regression functions. In the case where $X$ is supported on all of $\mathbb{R}$ with density $p_X$, we can show instead that $\int_{\mathbb{R}}|\widehat{m}_{n,k}(x)-m_k(x)| p_X(x)dx\rightarrow 0$ in probability, i.e., consistency in the $p_X$-weighted $L^1$ norm, which is a norm that has been used in other contexts (e.g. \cite{li2021minimax} and the references therein).

\section{Discussion}\label{sec:discussion}

We have undertaken a systematic study of uniform consistency in nonparametric mixture models, including both vanilla mixtures and mixed regression.
We constructed uniformly consistent estimators for mixed regression (Theorem~\ref{thm:uniform consistency}) and vanilla mixtures (Theorem~\ref{prop:npmix uniform consistency}) in nonparametric settings. In particular, our results make only mild nonparametric assumptions on the regression functions, error densities, and/or mixture components. Various extensions to weaker separation conditions as well as non-convolutional error densities have been outlined as well. Furthermore, the analysis highlights several subtleties in bootstrapping existing pointwise results to uniform results. In particular, the importance of the convolution structure and the resulting separation assumptions, as well as the (perhaps surprising) pivotal role played by having distinct weights in the model. We also illustrated how uniform consistency can easily break without these assumptions.
These results provide insight and justification into nonparametric latent variable models, for which mixtures are arguably the simplest case.   As our focus has been primarily theoretical, given the relevance of flexible, nonparametric models in practice, an important next step is to instantiate our models in practical applications.

\bibliography{bib} 

\begin{thebibliography}{94}
\providecommand{\natexlab}[1]{#1}
\providecommand{\url}[1]{\texttt{#1}}
\expandafter\ifx\csname urlstyle\endcsname\relax
  \providecommand{\doi}[1]{doi: #1}\else
  \providecommand{\doi}{doi: \begingroup \urlstyle{rm}\Url}\fi

\bibitem[Allman et~al.(2009)Allman, Matias, and
  Rhodes]{allman2009identifiability}
E.~S. Allman, C.~Matias, and J.~A. Rhodes.
\newblock Identifiability of parameters in latent structure models with many
  observed variables.
\newblock \emph{The Annals of Statistics}, 37\penalty0 (6A):\penalty0
  3099--3132, 2009.

\bibitem[Aragam et~al.(2020)Aragam, Dan, Xing, and
  Ravikumar]{aragam2020identifiability}
B.~Aragam, C.~Dan, E.~P. Xing, and P.~Ravikumar.
\newblock Identifiability of nonparametric mixture models and {B}ayes optimal
  clustering.
\newblock \emph{The Annals of Statistics}, 48\penalty0 (4):\penalty0
  2277--2302, 2020.

\bibitem[Arora and Kannan(2005)]{arora2005learning}
S.~Arora and R.~Kannan.
\newblock Learning mixtures of separated nonspherical {G}aussians.
\newblock \emph{The Annals of Applied Probability}, 15\penalty0 (1A):\penalty0
  69--92, 2005.

\bibitem[Balakrishnan et~al.(2017)Balakrishnan, Wainwright, and
  Yu]{balakrishnan2017statistical}
S.~Balakrishnan, M.~J. Wainwright, and B.~Yu.
\newblock Statistical guarantees for the {EM} algorithm: From population to
  sample-based analysis.
\newblock \emph{The Annals of Statistics}, 45\penalty0 (1):\penalty0 77--120,
  2017.

\bibitem[Beran(1977)]{beran1977minimum}
R.~Beran.
\newblock Minimum {H}ellinger distance estimates for parametric models.
\newblock \emph{The Annals of Statistics}, 5\penalty0 (3):\penalty0 445--463,
  1977.

\bibitem[Bordes et~al.(2006)Bordes, Mottelet, and Vandekerkhove]{bordes2006}
L.~Bordes, S.~Mottelet, and P.~Vandekerkhove.
\newblock Semiparametric estimation of a two-component mixture model.
\newblock \emph{The Annals of Statistics}, 34\penalty0 (3):\penalty0
  1204--1232, 2006.

\bibitem[Cai and Jin(2010)]{cai2010optimal}
T.~T. Cai and J.~Jin.
\newblock Optimal rates of convergence for estimating the null density and
  proportion of nonnull effects in large-scale multiple testing.
\newblock \emph{The Annals of Statistics}, 38\penalty0 (1):\penalty0 100--145,
  2010.

\bibitem[Cai et~al.(2019)Cai, Ma, and Zhang]{cai2019chime}
T.~T. Cai, J.~Ma, and L.~Zhang.
\newblock Chime: Clustering of high-dimensional {G}aussian mixtures with {EM}
  algorithm and its optimality.
\newblock \emph{The Annals of Statistics}, 47\penalty0 (3):\penalty0
  1234--1267, 2019.

\bibitem[Castelli and Cover(1995)]{castelli1995}
V.~Castelli and T.~M. Cover.
\newblock On the exponential value of labeled samples.
\newblock \emph{Pattern Recognition Letters}, 16\penalty0 (1):\penalty0
  105--111, 1995.

\bibitem[Castelli and Cover(1996)]{castelli1996}
V.~Castelli and T.~M. Cover.
\newblock The relative value of labeled and unlabeled samples in pattern
  recognition with an unknown mixing parameter.
\newblock \emph{IEEE Transactions on Information Theory}, 42\penalty0
  (6):\penalty0 2102--2117, 1996.

\bibitem[Chae and Walker(2020)]{chae2020wasserstein}
M.~Chae and S.~G. Walker.
\newblock Wasserstein upper bounds of the total variation for smooth densities.
\newblock \emph{Statistics \& Probability Letters}, 163:\penalty0 108771, 2020.

\bibitem[Chandra(1977)]{chandra1977mixtures}
S.~Chandra.
\newblock On the mixtures of probability distributions.
\newblock \emph{Scandinavian Journal of Statistics}, pages 105--112, 1977.

\bibitem[Chen(1995)]{chen1995}
J.~Chen.
\newblock Optimal rate of convergence for finite mixture models.
\newblock \emph{The Annals of Statistics}, pages 221--233, 1995.

\bibitem[Chen et~al.(2014)Chen, Yi, and Caramanis]{chen2014convex}
Y.~Chen, X.~Yi, and C.~Caramanis.
\newblock A convex formulation for mixed regression with two components:
  Minimax optimal rates.
\newblock In \emph{Conference on Learning Theory}, pages 560--604. PMLR, 2014.

\bibitem[Chen et~al.(2016)Chen, Genovese, Tibshirani, and
  Wasserman]{chen2016nonparametric}
Y.-C. Chen, C.~R. Genovese, R.~J. Tibshirani, and L.~Wasserman.
\newblock Nonparametric modal regression.
\newblock \emph{The Annals of Statistics}, 44\penalty0 (2):\penalty0 489--514,
  2016.

\bibitem[Compiani and Kitamura(2016)]{compiani2016using}
G.~Compiani and Y.~Kitamura.
\newblock Using mixtures in econometric models: A brief review and some new
  results.
\newblock \emph{The Econometrics Journal}, 19\penalty0 (3):\penalty0 C95--C127,
  2016.

\bibitem[Cozman et~al.(2003)Cozman, Cohen, and Cirelo]{cozman2003}
F.~G. Cozman, I.~Cohen, and M.~C. Cirelo.
\newblock Semi-supervised learning of mixture models.
\newblock In \emph{Proceedings of the 20th International Conference on Machine
  Learning (ICML-03)}, pages 99--106, 2003.

\bibitem[Dan et~al.(2018)Dan, Leqi, Aragam, Ravikumar, and Xing]{dan2018ssl}
C.~Dan, L.~Leqi, B.~Aragam, P.~K. Ravikumar, and E.~P. Xing.
\newblock The sample complexity of semi-supervised learning with nonparametric
  mixture models.
\newblock \emph{Advances in Neural Information Processing Systems}, 31, 2018.

\bibitem[De~Gooijer and Zerom(2003)]{de2003conditional}
J.~G. De~Gooijer and D.~Zerom.
\newblock On conditional density estimation.
\newblock \emph{Statistica Neerlandica}, 57\penalty0 (2):\penalty0 159--176,
  2003.

\bibitem[Deely and Kruse(1968)]{deely1968}
J.~J. Deely and R.~L. Kruse.
\newblock Construction of sequences estimating the mixing distribution.
\newblock \emph{The Annals of Mathematical Statistics}, 39\penalty0
  (1):\penalty0 286--288, 02 1968.

\bibitem[Doss et~al.(2020)Doss, Wu, Yang, and Zhou]{doss2020optimal}
N.~Doss, Y.~Wu, P.~Yang, and H.~H. Zhou.
\newblock Optimal estimation of high-dimensional {G}aussian mixtures.
\newblock \emph{arXiv preprint arXiv:2002.05818}, 2020.

\bibitem[Efromovich(2005)]{efromovich2005estimation}
S.~Efromovich.
\newblock Estimation of the density of regression errors.
\newblock \emph{The Annals of Statistics}, 33\penalty0 (5):\penalty0
  2194--2227, 2005.

\bibitem[Efromovich(2007)]{efromovich2007conditional}
S.~Efromovich.
\newblock Conditional density estimation in a regression setting.
\newblock \emph{The Annals of Statistics}, 35\penalty0 (6):\penalty0
  2504--2535, 2007.

\bibitem[Efron(2004)]{efron2004large}
B.~Efron.
\newblock Large-scale simultaneous hypothesis testing: The choice of a null
  hypothesis.
\newblock \emph{Journal of the American Statistical Association}, 99\penalty0
  (465):\penalty0 96--104, 2004.

\bibitem[Elmore et~al.(2005)Elmore, Hall, and Neeman]{elmore2005}
R.~Elmore, P.~Hall, and A.~Neeman.
\newblock An application of classical invariant theory to identifiability in
  nonparametric mixtures.
\newblock In \emph{Annales de l'institut Fourier}, volume~55, pages 1--28,
  2005.

\bibitem[Erola et~al.(2020)Erola, Bj{\"o}rkegren, and Michoel]{erola2020model}
P.~Erola, J.~L. Bj{\"o}rkegren, and T.~Michoel.
\newblock Model-based clustering of multi-tissue gene expression data.
\newblock \emph{Bioinformatics}, 36\penalty0 (6):\penalty0 1807--1813, 2020.

\bibitem[Fan(1991)]{fan1991optimal}
J.~Fan.
\newblock On the optimal rates of convergence for nonparametric deconvolution
  problems.
\newblock \emph{The Annals of Statistics}, pages 1257--1272, 1991.

\bibitem[Feng and Dicker(2018)]{feng2015nonparametric}
L.~Feng and L.~H. Dicker.
\newblock Approximate nonparametric maximum likelihood for mixture models: A
  convex optimization approach to fitting arbitrary multivariate mixing
  distributions.
\newblock \emph{Computational Statistics \& Data Analysis}, 122:\penalty0
  80--91, 2018.

\bibitem[Fisher and Yakowitz(1970)]{fisher1970estimating}
L.~Fisher and S.~Yakowitz.
\newblock Estimating mixing distributions in metric spaces.
\newblock \emph{Sankhy{\=a}: The Indian Journal of Statistics, Series A}, pages
  411--418, 1970.

\bibitem[Fraley and Raftery(2002)]{fraley2002}
C.~Fraley and A.~E. Raftery.
\newblock Model-based clustering, discriminant analysis, and density
  estimation.
\newblock \emph{Journal of the American statistical Association}, 97\penalty0
  (458):\penalty0 611--631, 2002.

\bibitem[Fr{\"u}hwirth-Schnatter(2006)]{fruhwirth2006}
S.~Fr{\"u}hwirth-Schnatter.
\newblock \emph{Finite mixture and Markov switching models}.
\newblock Springer Science \& Business Media, 2006.

\bibitem[Gassiat and Rousseau(2016)]{gassiat2016nonparametric}
E.~Gassiat and J.~Rousseau.
\newblock Nonparametric finite translation hidden markov models and extensions.
\newblock \emph{Bernoulli}, 22\penalty0 (1):\penalty0 193--212, 2016.

\bibitem[Gassiat et~al.(2020)Gassiat, Le~Corff, and
  Leh{\'e}ricy]{gassiat2020identifiability}
E.~Gassiat, S.~Le~Corff, and L.~Leh{\'e}ricy.
\newblock Identifiability and consistent estimation of nonparametric
  translation hidden {M}arkov models with general state space.
\newblock \emph{Journal of Machine Learning Research}, 21:\penalty0 115--1,
  2020.

\bibitem[Genovese and Wasserman(2000)]{genovese2000rates}
C.~R. Genovese and L.~Wasserman.
\newblock Rates of convergence for the {G}aussian mixture sieve.
\newblock \emph{The Annals of Statistics}, 28\penalty0 (4):\penalty0
  1105--1127, 2000.

\bibitem[Ghosal and Van Der~Vaart(2007)]{ghosal2007}
S.~Ghosal and A.~Van Der~Vaart.
\newblock Posterior convergence rates of {D}irichlet mixtures at smooth
  densities.
\newblock \emph{The Annals of Statistics}, 35\penalty0 (2):\penalty0 697--723,
  2007.

\bibitem[Ghosal and Van Der~Vaart(2001)]{ghosal2001entropies}
S.~Ghosal and A.~W. Van Der~Vaart.
\newblock Entropies and rates of convergence for maximum likelihood and {B}ayes
  estimation for mixtures of normal densities.
\newblock \emph{The Annals of Statistics}, pages 1233--1263, 2001.

\bibitem[Gin{\'e} and Guillou(2002)]{gine2002rates}
E.~Gin{\'e} and A.~Guillou.
\newblock Rates of strong uniform consistency for multivariate kernel density
  estimators.
\newblock In \emph{Annales de l'Institut Henri Poincare (B) Probability and
  Statistics}, volume~38, pages 907--921. Elsevier, 2002.

\bibitem[Hall and Zhou(2003)]{hall2003nonparametric}
P.~Hall and X.-H. Zhou.
\newblock Nonparametric estimation of component distributions in a multivariate
  mixture.
\newblock \emph{The Annals of Statistics}, 31\penalty0 (1):\penalty0 201--224,
  2003.

\bibitem[Hall et~al.(2005)Hall, Neeman, Pakyari, and Elmore]{hall2005mixture}
P.~Hall, A.~Neeman, R.~Pakyari, and R.~Elmore.
\newblock Nonparametric inference in multivariate mixtures.
\newblock \emph{Biometrika}, 92\penalty0 (3):\penalty0 667--678, 2005.

\bibitem[Hand and Joshi(2018)]{hand2018convex}
P.~Hand and B.~Joshi.
\newblock A convex program for mixed linear regression with a recovery
  guarantee for well-separated data.
\newblock \emph{Information and Inference: A Journal of the IMA}, 7\penalty0
  (3):\penalty0 563--579, 2018.

\bibitem[Heinrich and Kahn(2018)]{heinrich2018strong}
P.~Heinrich and J.~Kahn.
\newblock Strong identifiability and optimal minimax rates for finite mixture
  estimation.
\newblock \emph{The Annals of Statistics}, 46\penalty0 (6A):\penalty0
  2844--2870, 2018.

\bibitem[Ho and Nguyen(2016)]{ho2016convergence}
N.~Ho and X.~Nguyen.
\newblock Convergence rates of parameter estimation for some weakly
  identifiable finite mixtures.
\newblock \emph{The Annals of Statistics}, 44\penalty0 (6):\penalty0
  2726--2755, 2016.

\bibitem[Ho and Nguyen(2019)]{ho2019singularity}
N.~Ho and X.~Nguyen.
\newblock Singularity structures and impacts on parameter estimation in finite
  mixtures of distributions.
\newblock \emph{SIAM Journal on Mathematics of Data Science}, 1\penalty0
  (4):\penalty0 730--758, 2019.

\bibitem[Ho et~al.(2022{\natexlab{a}})Ho, Feller, Greif, Miratrix, and
  Pillai]{feller2016weak}
N.~Ho, A.~Feller, E.~Greif, L.~Miratrix, and N.~Pillai.
\newblock Weak separation in mixture models and implications for principal
  stratification.
\newblock In \emph{International Conference on Artificial Intelligence and
  Statistics}, pages 5416--5458. PMLR, 2022{\natexlab{a}}.

\bibitem[Ho et~al.(2022{\natexlab{b}})Ho, Yang, and Jordan]{ho2019convergence}
N.~Ho, C.-Y. Yang, and M.~I. Jordan.
\newblock Convergence rates for {G}aussian mixtures of experts.
\newblock \emph{Journal of Machine Learning Research}, 23\penalty0
  (323):\penalty0 1--81, 2022{\natexlab{b}}.

\bibitem[Huang and Yao(2012)]{huang2012mixture}
M.~Huang and W.~Yao.
\newblock Mixture of regression models with varying mixing proportions: A
  semiparametric approach.
\newblock \emph{Journal of the American Statistical Association}, 107\penalty0
  (498):\penalty0 711--724, 2012.

\bibitem[Huang et~al.(2013)Huang, Li, and Wang]{huang2013}
M.~Huang, R.~Li, and S.~Wang.
\newblock Nonparametric mixture of regression models.
\newblock \emph{Journal of the American Statistical Association}, 108\penalty0
  (503):\penalty0 929--941, 2013.

\bibitem[Hunter and Young(2012)]{hunter2012semiparametric}
D.~R. Hunter and D.~S. Young.
\newblock Semiparametric mixtures of regressions.
\newblock \emph{Journal of Nonparametric Statistics}, 24\penalty0 (1):\penalty0
  19--38, 2012.

\bibitem[Hunter et~al.(2007)Hunter, Wang, and Hettmansperger]{hunter2007}
D.~R. Hunter, S.~Wang, and T.~P. Hettmansperger.
\newblock Inference for mixtures of symmetric distributions.
\newblock \emph{The Annals of Statistics}, pages 224--251, 2007.

\bibitem[Ishwaran(1996)]{ishwaran1996identifiability}
H.~Ishwaran.
\newblock Identifiability and rates of estimation for scale parameters in
  location mixture models.
\newblock \emph{The Annals of Statistics}, 24\penalty0 (4):\penalty0
  1560--1571, 1996.

\bibitem[Jacobs et~al.(1991)Jacobs, Jordan, Nowlan, and Hinton]{jacobs1991}
R.~A. Jacobs, M.~I. Jordan, S.~J. Nowlan, and G.~E. Hinton.
\newblock Adaptive mixtures of local experts.
\newblock \emph{Neural computation}, 3\penalty0 (1):\penalty0 79--87, 1991.

\bibitem[Jiang and Tanner(1999{\natexlab{a}})]{jiang1999}
W.~Jiang and M.~A. Tanner.
\newblock Hierarchical mixtures-of-experts for exponential family regression
  models: Approximation and maximum likelihood estimation.
\newblock \emph{The Annals of Statistics}, pages 987--1011, 1999{\natexlab{a}}.

\bibitem[Jiang and Tanner(1999{\natexlab{b}})]{jiang1999identifiability}
W.~Jiang and M.~A. Tanner.
\newblock On the identifiability of mixtures-of-experts.
\newblock \emph{Neural Networks}, 12\penalty0 (9):\penalty0 1253--1258,
  1999{\natexlab{b}}.

\bibitem[Jordan and Jacobs(1994)]{jordan1994}
M.~I. Jordan and R.~A. Jacobs.
\newblock Hierarchical mixtures of experts and the {EM} algorithm.
\newblock \emph{Neural computation}, 6\penalty0 (2):\penalty0 181--214, 1994.

\bibitem[Kampffmeyer et~al.(2019)Kampffmeyer, L{\o}kse, Bianchi, Livi, Salberg,
  and Jenssen]{kampffmeyer2019deep}
M.~Kampffmeyer, S.~L{\o}kse, F.~M. Bianchi, L.~Livi, A.-B. Salberg, and
  R.~Jenssen.
\newblock Deep divergence-based approach to clustering.
\newblock \emph{Neural Networks}, 113:\penalty0 91--101, 2019.

\bibitem[Kitamura and Laage(2018)]{kitamura2018nonparametric}
Y.~Kitamura and L.~Laage.
\newblock Nonparametric analysis of finite mixtures.
\newblock \emph{arXiv preprint arXiv:1811.02727}, 2018.

\bibitem[Kivva et~al.(2021)Kivva, Rajendran, Ravikumar, and
  Aragam]{kivva2021learning}
B.~Kivva, G.~Rajendran, P.~Ravikumar, and B.~Aragam.
\newblock Learning latent causal graphs via mixture oracles.
\newblock \emph{Advances in Neural Information Processing Systems},
  34:\penalty0 18087--18101, 2021.

\bibitem[Kivva et~al.(2022)Kivva, Rajendran, Ravikumar, and
  Aragam]{kivva2022identifiability}
B.~Kivva, G.~Rajendran, P.~Ravikumar, and B.~Aragam.
\newblock Identifiability of deep generative models without auxiliary
  information.
\newblock \emph{Advances in Neural Information Processing Systems}, 35, 2022.

\bibitem[Koltchinskii(2000)]{koltchinskii2000empirical}
V.~I. Koltchinskii.
\newblock Empirical geometry of multivariate data: A deconvolution approach.
\newblock \emph{The Annals of statistics}, pages 591--629, 2000.

\bibitem[Kruijer et~al.(2010)Kruijer, Rousseau, and Van
  Der~Vaart]{kruijer2010adaptive}
W.~Kruijer, J.~Rousseau, and A.~Van Der~Vaart.
\newblock Adaptive {B}ayesian density estimation with location-scale mixtures.
\newblock \emph{Electronic Journal of Statistics}, 4:\penalty0 1225--1257,
  2010.

\bibitem[Kwon and Caramanis(2020)]{kwon2020converges}
J.~Kwon and C.~Caramanis.
\newblock {EM converges for a mixture of many linear regressions}.
\newblock In \emph{International Conference on Artificial Intelligence and
  Statistics}, pages 1727--1736. PMLR, 2020.

\bibitem[Kwon et~al.(2021)Kwon, Ho, and Caramanis]{kwon2021minimax}
J.~Kwon, N.~Ho, and C.~Caramanis.
\newblock On the minimax optimality of the {EM} algorithm for learning
  two-component mixed linear regression.
\newblock In \emph{International Conference on Artificial Intelligence and
  Statistics}, pages 1405--1413. PMLR, 2021.

\bibitem[Li et~al.(2022)Li, Neykov, and Balakrishnan]{li2021minimax}
M.~Li, M.~Neykov, and S.~Balakrishnan.
\newblock Minimax optimal conditional density estimation under total variation
  smoothness.
\newblock \emph{Electronic Journal of Statistics}, 16\penalty0 (2):\penalty0
  3937--3972, 2022.

\bibitem[Li and Liang(2018)]{li2018learning}
Y.~Li and Y.~Liang.
\newblock Learning mixtures of linear regressions with nearly optimal
  complexity.
\newblock In \emph{Conference On Learning Theory}, pages 1125--1144. PMLR,
  2018.

\bibitem[Makkuva et~al.(2019)Makkuva, Viswanath, Kannan, and
  Oh]{makkuva2019breaking}
A.~Makkuva, P.~Viswanath, S.~Kannan, and S.~Oh.
\newblock Breaking the gridlock in mixture-of-experts: Consistent and efficient
  algorithms.
\newblock In \emph{International Conference on Machine Learning}, pages
  4304--4313. PMLR, 2019.

\bibitem[Melnykov and Maitra(2010)]{melnykov2010finite}
V.~Melnykov and R.~Maitra.
\newblock Finite mixture models and model-based clustering.
\newblock \emph{Statistics Surveys}, 4:\penalty0 80--116, 2010.

\bibitem[Moulines et~al.(1997)Moulines, Cardoso, and
  Gassiat]{moulines1997maximum}
E.~Moulines, J.-F. Cardoso, and E.~Gassiat.
\newblock Maximum likelihood for blind separation and deconvolution of noisy
  signals using mixture models.
\newblock In \emph{1997 IEEE International Conference on Acoustics, Speech, and
  Signal Processing}, volume~5, pages 3617--3620. IEEE, 1997.

\bibitem[Nguyen and McLachlan(2019)]{nguyen2019approximations}
H.~D. Nguyen and G.~McLachlan.
\newblock On approximations via convolution-defined mixture models.
\newblock \emph{Communications in Statistics-Theory and Methods}, 48\penalty0
  (16):\penalty0 3945--3955, 2019.

\bibitem[Nguyen et~al.(2016)Nguyen, Lloyd-Jones, and McLachlan]{nguyen2016}
H.~D. Nguyen, L.~R. Lloyd-Jones, and G.~J. McLachlan.
\newblock A universal approximation theorem for mixture-of-experts models.
\newblock \emph{Neural Computation}, 28\penalty0 (12):\penalty0 2585--2593,
  2016.

\bibitem[Nguyen(2013)]{nguyen2013}
X.~Nguyen.
\newblock Convergence of latent mixing measures in finite and infinite mixture
  models.
\newblock \emph{The Annals of Statistics}, 41\penalty0 (1):\penalty0 370--400,
  2013.

\bibitem[Pan et~al.(2002)Pan, Lin, and Le]{pan2002model}
W.~Pan, J.~Lin, and C.~T. Le.
\newblock Model-based cluster analysis of microarray gene-expression data.
\newblock \emph{Genome biology}, 3\penalty0 (2):\penalty0 1--8, 2002.

\bibitem[Rao(1992)]{rao1992}
B.~L. S.~P. Rao.
\newblock \emph{Identifiability in Stochastic Models: Characterization of
  Probability Distributions}.
\newblock Elsevier, 1992.

\bibitem[Regev and Vijayaraghavan(2017)]{regev2017learning}
O.~Regev and A.~Vijayaraghavan.
\newblock On learning mixtures of well-separated {G}aussians.
\newblock In \emph{2017 IEEE 58th Annual Symposium on Foundations of Computer
  Science (FOCS)}, pages 85--96. IEEE, 2017.

\bibitem[Ritchie et~al.(2020)Ritchie, Vandermeulen, and
  Scott]{ritchie2020consistent}
A.~Ritchie, R.~A. Vandermeulen, and C.~Scott.
\newblock Consistent estimation of identifiable nonparametric mixture models
  from grouped observations.
\newblock \emph{Advances in Neural Information Processing Systems},
  33:\penalty0 11676--11686, 2020.

\bibitem[Ritter(2014)]{ritter2014}
G.~Ritter.
\newblock \emph{Robust cluster analysis and variable selection}.
\newblock CRC Press, 2014.

\bibitem[Saha and Guntuboyina(2020)]{saha2020nonparametric}
S.~Saha and A.~Guntuboyina.
\newblock {On the nonparametric maximum likelihood estimator for Gaussian
  location mixture densities with application to Gaussian denoising}.
\newblock \emph{The Annals of Statistics}, 48\penalty0 (2):\penalty0 738--762,
  2020.

\bibitem[Si et~al.(2014)Si, Liu, Li, and Brutnell]{si2014model}
Y.~Si, P.~Liu, P.~Li, and T.~P. Brutnell.
\newblock Model-based clustering for rna-seq data.
\newblock \emph{Bioinformatics}, 30\penalty0 (2):\penalty0 197--205, 2014.

\bibitem[Teicher(1960)]{teicher1960}
H.~Teicher.
\newblock On the mixture of distributions.
\newblock \emph{The Annals of Mathematical Statistics}, pages 55--73, 1960.

\bibitem[Teicher(1961)]{teicher1961}
H.~Teicher.
\newblock Identifiability of mixtures.
\newblock \emph{The Annals of Mathematical Statistics}, 32\penalty0
  (1):\penalty0 244--248, 1961.

\bibitem[Teicher(1963)]{teicher1963}
H.~Teicher.
\newblock Identifiability of finite mixtures.
\newblock \emph{The Annals of Mathematical Statistics}, pages 1265--1269, 1963.

\bibitem[Teicher(1967)]{teicher1967}
H.~Teicher.
\newblock Identifiability of mixtures of product measures.
\newblock \emph{The Annals of Mathematical Statistics}, 38\penalty0
  (4):\penalty0 1300--1302, 1967.

\bibitem[Vandekerkhove(2013)]{vandekerkhove2013estimation}
P.~Vandekerkhove.
\newblock Estimation of a semiparametric mixture of regressions model.
\newblock \emph{Journal of Nonparametric Statistics}, 25\penalty0 (1):\penalty0
  181--208, 2013.

\bibitem[Vandermeulen and Scott(2019)]{vandermeulen2019operator}
R.~A. Vandermeulen and C.~D. Scott.
\newblock An operator theoretic approach to nonparametric mixture models.
\newblock \emph{The Annals of Statistics}, 47\penalty0 (5):\penalty0
  2704--2733, 2019.

\bibitem[Villani(2008)]{villani2008optimal}
C.~Villani.
\newblock \emph{Optimal Transport: Old and New}, volume 338.
\newblock Springer Science \& Business Media, 2008.

\bibitem[Wu and Yang(2020)]{wu2020optimal}
Y.~Wu and P.~Yang.
\newblock Optimal estimation of {G}aussian mixtures via denoised method of
  moments.
\newblock \emph{The Annals of Statistics}, 48\penalty0 (4):\penalty0
  1981--2007, 2020.

\bibitem[Xiang and Yao(2018)]{xiang2018semiparametric}
S.~Xiang and W.~Yao.
\newblock Semiparametric mixtures of nonparametric regressions.
\newblock \emph{Annals of the Institute of Statistical Mathematics},
  70\penalty0 (1):\penalty0 131--154, 2018.

\bibitem[Yakowitz and Spragins(1968)]{yakowitz1968}
S.~J. Yakowitz and J.~D. Spragins.
\newblock On the identifiability of finite mixtures.
\newblock \emph{The Annals of Mathematical Statistics}, pages 209--214, 1968.

\bibitem[Yao and Li(2014)]{yao2014new}
W.~Yao and L.~Li.
\newblock A new regression model: Modal linear regression.
\newblock \emph{Scandinavian Journal of Statistics}, 41\penalty0 (3):\penalty0
  656--671, 2014.

\bibitem[Yao et~al.(2012)Yao, Lindsay, and Li]{yao2012local}
W.~Yao, B.~G. Lindsay, and R.~Li.
\newblock Local modal regression.
\newblock \emph{Journal of Nonparametric Statistics}, 24\penalty0 (3):\penalty0
  647--663, 2012.

\bibitem[Yen et~al.(2018)Yen, Lee, Chang, Zhong, Ravikumar, and
  Lin]{yen2018mixlasso}
I.~E. Yen, W.-C. Lee, S.-E. Chang, K.~Zhong, P.~Ravikumar, and S.-D. Lin.
\newblock Mixlasso: Generalized mixed regression via convex atomic-norm
  regularization.
\newblock In \emph{Proceedings of the 32nd International Conference on Neural
  Information Processing Systems}, pages 10891--10899, 2018.

\bibitem[Yi et~al.(2014)Yi, Caramanis, and Sanghavi]{yi2014}
X.~Yi, C.~Caramanis, and S.~Sanghavi.
\newblock Alternating minimization for mixed linear regression.
\newblock In \emph{International Conference on Machine Learning}, pages
  613--621, 2014.

\bibitem[Young and Hunter(2010)]{young2010mixtures}
D.~S. Young and D.~R. Hunter.
\newblock Mixtures of regressions with predictor-dependent mixing proportions.
\newblock \emph{Computational Statistics \& Data Analysis}, 54\penalty0
  (10):\penalty0 2253--2266, 2010.

\bibitem[Zeevi et~al.(1998)Zeevi, Meir, and Maiorov]{zeevi1998}
A.~J. Zeevi, R.~Meir, and V.~Maiorov.
\newblock Error bounds for functional approximation and estimation using
  mixtures of experts.
\newblock \emph{IEEE Transactions on Information Theory}, 44\penalty0
  (3):\penalty0 1010--1025, 1998.

\bibitem[Zhang(1990)]{zhang1990fourier}
C.-H. Zhang.
\newblock Fourier methods for estimating mixing densities and distributions.
\newblock \emph{The Annals of Statistics}, pages 806--831, 1990.

\end{thebibliography}
\bibliographystyle{abbrvnat}

\appendix

\section{Proofs}\label{sec:proofs}

\subsection{A guide to the proofs}

This appendix collects all the technical proofs of results from the main paper. There is one source of anachronism which we highlight up front: Strictly speaking, Theorem~\ref{thm:uniform consistency} depends on Theorem~\ref{prop:npmix uniform consistency}, and so for proper logical flow, the proof of Theorem~\ref{prop:npmix uniform consistency} should be read first. However, the reader willing to accept Theorem~\ref{prop:npmix uniform consistency} on a first reading can safely proceed directly with the proof of Theorem~\ref{thm:uniform consistency}. This is how we have presented the proofs, and should cause no confusion. Readers interested in strict logical flow are advised to begin reading Appendix~\ref{app:proof:sec5} before Appendix~\ref{app:proof:sec4}.

As a further guide to the interested reader, we provide the following overview.
\begin{itemize}
    \item The proofs for Section~\ref{sec:model} in Appendix~\ref{sec:proof identifiability} are self-contained and may be read in any order.
    \item An overview of the main ideas behind the proof of Theorem~\ref{thm:uniform consistency} can be found in Section~\ref{sec:est}.
    The necessary technical proofs are found in Appendix~\ref{app:proof:sec4}.
    \item An overview of the main ideas behind the proof of Theorem~\ref{prop:npmix uniform consistency} can be found in Sections~\ref{sec:npmix:toy}-\ref{sec:npmix:est}.
    The necessary technical proofs are found in Appendix~\ref{app:proof:sec5}.
    \item The proofs for Section~\ref{sec:gen} in Appendix~\ref{sec:proof for section 6} are self-contained and may be read in any order.
\end{itemize}
Finally, for the most part, proofs are presented in the order they appear in the main text.

\subsection{Proofs for Section \ref{sec:model}}\label{sec:proof identifiability}

\subsubsection{Proof of Proposition \ref{prop:identifiability convolutional}}\label{appen:proof of identifiability}
We need to show that 
\begin{align*}
    \sum_{k=1}^K \lambda_k f_k(y-\mu_k)=\sum_{k=1}^K \pi_k g_k(y-m_k)
\end{align*}
implies $\{(f_k,\lambda_k,\mu_k)\}_{k=1}^K=\{(g_k,\pi_k,m_k)\}_{k=1}^K$.
Using the fact that $f_k=\phi_{\sigma}\ast G_k$ and $g_k=\phi_{\sigma}\ast H_k$, we get 
\begin{align*}
    \phi_{\sigma}\ast \sum_{k=1}^K \lambda_kG_k(\cdot-\mu_k)=\phi_{\sigma}\ast \sum_{k=1}^K \pi_kH_k(\cdot-m_k),
\end{align*}
where $G_k(\cdot-\mu_k)$ is the translation of $G_k$ by $\mu_k$ (and similarly for the $H_k$'s). 
By \cite[][Theorem 2]{nguyen2013} we have 
\begin{align}
    \sum_{k=1}^K \lambda_kG_k(\cdot-\mu_k)=\sum_{k=1}^K \pi_kH_k(\cdot-m_k) \label{eq:mixing measure equal}
\end{align}
and in particular their supports are equal, i.e., 
\begin{align}
    \bigcup_{k=1}^K E_k=\bigcup_{k=1}^K F_k, \label{eq:supp equal}
\end{align}
where $E_k=\operatorname{supp}(G_k(\cdot-\mu_k))$ and $F_k=\operatorname{supp}(H_k(\cdot-m_k))$. We claim that for each $k$, $F_k\subset E_{\Pi(k)}$ for some $\Pi(k)$. Suppose otherwise that both $F_{i^{\ast}}\cap E_{j^{\ast}}$ and $F_{i^{\ast}}\cap E_{k^{\ast}}$ are nonempty for some $i^{\ast}$ and $j^{\ast}\neq k^{\ast}$. The assumption on the separation implies that 
\begin{align}
    \underset{j\neq k}{\operatorname{min}}\, \operatorname{dist}(E_j,E_k) > \underset{k}{\operatorname{max}}\,\operatorname{diam}(\operatorname{supp}(G_k)),\quad \underset{j\neq k}{\operatorname{min}}\, \operatorname{dist}(F_j,F_k) > \underset{k}{\operatorname{max}}\,\operatorname{diam}(\operatorname{supp} (H_k)).\label{eq:identifiability proof separation}
\end{align}
Hence $\operatorname{diam}(\operatorname{supp}( H_{i^{\ast}}))=\operatorname{diam}(F_{i^{\ast}})\geq \operatorname{dist}(E_{j^{\ast}},E_{k^{\ast}})>\operatorname{max}_k\operatorname{diam}(\operatorname{supp} (G_k))$. Then for any $k\neq i^{\ast}$, we see that $F_k$ does not intersect with either $E_{j^{\ast}}$ or $E_{k^{\ast}}$ because otherwise 
\begin{align*}
    \operatorname{dist}(F_{i^{\ast}},F_k) &\leq \operatorname{max} \{\operatorname{diam}(E_{j^{\ast}}), \operatorname{diam}(E_{k^{\ast}})\}\\
    &=\operatorname{max}\{\operatorname{diam}(\operatorname{supp} (G_{j^{\ast}})),\operatorname{diam}(\operatorname{supp} (G_{k^{\ast}}))\}<\operatorname{diam}(\operatorname{supp}( H_{i^{\ast}})), 
\end{align*}
a contradiction to \eqref{eq:identifiability proof separation}.  Therefore it follows that the remaining $K-1$ sets $\{F_k\}_{k=2}^K$ have union that is contained in $\bigcup_{k\neq j^{\ast},k^{\ast}} E_k$, a union of $K-2$ disjoint sets. In particular, there will be two sets $F_{m^{\ast}}$ and $F_{\ell^{\ast}}$ that intersect a same $E_{k}$. But a similar argument as above implies that 
\begin{align*}
    \operatorname{dist}(F_{m^{\ast}},F_{\ell^{\ast}})\leq \operatorname{diam}(E_k)=\operatorname{diam}(\operatorname{supp} (G_k))<\operatorname{diam}(\operatorname{supp} (H_{i^{\ast}})),
\end{align*}
again a contradiction to \eqref{eq:identifiability proof separation}. Therefore our original claim holds that for any $k$, $F_k\subset E_{\Pi(k)}$ for some $\Pi(k)$ and the argument above also implies $\Pi$ is a bijection. Since the $E_k$'s are disjoint, the equality \eqref{eq:supp equal} implies $F_k=E_{\Pi(k)}$ and evaluating \eqref{eq:mixing measure equal} at $F_k$ implies $\pi_k=\lambda_{\Pi(k)}$. As a consequence by restricting \eqref{eq:mixing measure equal} to their common support $F_k$, we have 
\begin{align}
    H_k(\cdot-m_k)=G_{\Pi(k)}(\cdot-\mu_{\Pi(k)}). \label{eq:individual mixing measure equal}
\end{align}
The assumption $\int_{\mathbb{R}}xf_k(x)dx=0$ translates to  $\int_{\mathbb{R}}\theta dG_k(\theta)=0$ and hence \eqref{eq:individual mixing measure equal} implies $m_k=\mu_{\Pi(k)}$ by computing the first moment. Therefore it follows that $H_k=G_{\Pi(k)}$ and $g_k=f_{\Pi(k)}$, which concludes the proof since $\Pi$ is a bijection.   
\qed

\subsubsection{Proof of Theorem \ref{thm:identifiability convolutional}}
We need to show that if
\begin{align*}
    p_X(x)\sum_{k=1}^K \lambda_k f(y-m_k(x))= q_X(x)\sum_{k=1}^K \pi_kg(y-\mu_k(x)) \quad \quad \forall \,\,(x,y)\in[a,b]\times \mathbb{R}
\end{align*}
then $(p_X,f,\{\lambda_k\}_{k=1}^K, \{m_k\}_{k=1}^K)=(q_X,g,\{\pi_k\}_{k=1}^K, \{\mu_k\}_{k=1}^K)$.
By integrating the above equation with respect to $y$ we see that $p_X=q_X$. Since $p_X=q_X>0$ over $[a,b]$ by assumption, the above equation reduces to 
\begin{align}
    \sum_{k=1}^K \lambda_k f(y-m_k(x))= \sum_{k=1}^K \pi_kg(y-\mu_k(x)) \quad \quad \forall \,\,(x,y)\in[a,b]\times \mathbb{R}.  \label{eq:identifiability outline conditional density}
\end{align}
Setting $x=x_0$ in the above equation we get by Proposition \ref{prop:identifiability convolutional} that $(f,\{\lambda_k\}_{k=1}^K)=(g,\{\pi_k\}_{k=1}^K)$ and the equation reduces to (possibly up to permutation of the indices)
\begin{align}
     \sum_{k=1}^K \lambda_k f(y-m_k(x))= \sum_{k=1}^K \lambda_k f(y-\mu_k(x)) \quad \quad \forall \,\,(x,y)\in[a,b]\times \mathbb{R}. \label{eq:translation family}
\end{align}
By assumption, the regression functions of both models intersect only at countably many points. We shall abuse the notation and let $Z$ denote the set of all intersection points for both models,   which is still countable and partitions the real line into disjoint interval $(z_k,z_{k+1})$'s. Fix $x\in (z_k,z_{k+1})$   in \eqref{eq:translation family}, by identifiability of translation family \cite[][Proposition~6]{yakowitz1968} at least two elements in $\{m_k(x),\mu_k(x)\}_{k=1}^K$ are equal. Since  $\{m_k(x)\}_{k=1}^K$ and $\{\mu_k(x)\}_{k=1}^K$ are both distinct, this implies $m_{k^{\ast}}(x)=\mu_{\Pi(k^{\ast})}(x)$ for some $k^{\ast}$ and \eqref{eq:translation family} reduces to
\begin{align*}
    (\lambda_{k^{\ast}}-\lambda_{\Pi(k^{\ast})}) f(y-m_{k^{\ast}}(x))+\sum_{k\neq k^{\ast}} \lambda_k f(y-m_k(x))-\sum_{k\neq \Pi(k^{\ast})} \lambda_k f(y-\mu_k(x))=0.
\end{align*}
Now since $m_{k^{\ast}}(x)$ is different from the remaining $m_k(x)$'s and $\mu_k(x)$'s, we conclude again by identifiability of translation family that $\lambda_{k^{\ast}}=\lambda_{\Pi(k^{\ast})}$, which implies $k^{\ast}=\Pi(k^{\ast})$ since the $\lambda_k$'s are distinct. Now repeating the above process we will end up with $m_k(x)=\mu_k(x)$ for all $k$ and for all $x\in Z^c$. Since $Z$ is countable   and in particular has measure zero, this continues to hold for $x\in Z$ by the continuity of the regression functions. This concludes the proof. \qed

\subsubsection{Proof of Proposition \ref{prop:pointwise consistency convolutional}}

Notice that this is the pointwise version Theorem \ref{thm:uniform consistency}, but the stronger almost sure convergence is claimed under the weaker regularity assumption \ref{assumption:joint:holder} on the joint density.  This is due to the difference between Lemmas \ref{lemma:convergence of KDE} and \ref{lemma:cde bound}, where the former suffices for pointwise consistency. Therefore with the almost sure convergence of the density estimator, the result can be proved in the same way as Theorem \ref{thm:uniform consistency} and is omitted. 
\qed

\subsection{Proofs for Section \ref{sec:npmor}}
\label{app:proof:sec4}

In this appendix we prove our first main result (Theorem~\ref{thm:uniform consistency}). Theorem~\ref{thm:uniform consistency} is ultimately a corollary of two results of independent interest given by Propositions~\ref{prop:step 2 uniform consistency} and~\ref{prop:step 3 uniform consistency}. A complete proof of the theorem can be found at the end of this appendix, after the proofs of these two propositions and other supporting lemmas have been presented.

\subsubsection{Proof of Proposition \ref{prop:step 1 uniform consistency}}

Lemma \ref{lemma:cde bound} bounds the error $\operatorname{sup}_{x\in[a,b]}\mathbb{E}\|\widehat{p}_n(\,\cdot\given x)-p(\,\cdot\given x)\|_1$ in terms of quantities depending on $p_X$ and the $m_k$'s. The result then follows from a uniform control of these quantities imposed by our assumptions on $p_X$ and \ref{assumption:unpmor:regularity}.   
\qed

\subsubsection{Proof of Proposition \ref{prop:step 2 uniform consistency}}
\textbf{Case 1: $x_0\in(a,b)$.}
We shall directly apply Theorem \ref{prop:npmix uniform consistency} to the vanilla mixture $p(\,\cdot\given x_0)$. To check the assumptions, notice that the mixing measure underlying $p(\,\cdot\given x_0)$ is 
\begin{align*}
    G=\sum_{k=1}^K \lambda_k G_0(\cdot-m_k(x_0))
\end{align*}
so that assumptions \ref{assumption:unpmor:regularity}  and \ref{assumption:unpmor:lambda} imply \ref{assumption:D1} and \ref{assumption:D2}. Furthermore, 
\begin{align*}
    &\operatorname{dist}\big(\operatorname{supp}(G_0(\cdot-m_j(x_0))), \operatorname{supp}(G_0(\cdot-m_k(x_0)))\big)\\
    &= \operatorname{max}\{0,|m_j(x_0)-m_k(x_0)|-\operatorname{diam}(\operatorname{supp}(G_0))\}
\end{align*}
so that \ref{assumption:unpmor:separation} implies \ref{assumption:D3}. Finally, the estimators in Theorem \ref{prop:npmix uniform consistency} rely on a uniformly consistent density estimator of $p(\,\cdot\given x_0)$ (not explicitly mentioned in the statements but in the details of the procedure). This is then given by Proposition  \ref{prop:step 1 uniform consistency}, which implies 
\begin{align*}
    \underset{\mathcal{V}_{x_0}}{\operatorname{sup}}\,\, \mathbb{E}\|\widehat{p}_n(\,\cdot\given x_0)-p(\,\cdot\given x_0)\|_1\xrightarrow{n\rightarrow\infty}{0}
\end{align*}
  since $x_0\in(a,b)$.   
Therefore Theorem \ref{prop:npmix uniform consistency} applies. 
The fact that we can take $\Pi$ to be the identity has been discussed in Remark \ref{remark:relabel}.

  \textbf{Case 2: $x_0=a$ or $x_0=b$.} We consider the case $x_0=a$; the case $x_0=b$ is similar. Instead of working directly with $p(\cdot\,|\,x_0)$, we shall apply Proposition \ref{lemma:conditioning} to the sequences $p(\cdot\,|\,x_0+h_n)$ (or $p(\cdot\,|\,x_0-h_n)$ when $x_0=b$). By Proposition \ref{prop:step 1 uniform consistency}, we have
\begin{align}
    \underset{\mathcal{V}_{x_0}}{\operatorname{sup}}\,\, \mathbb{E} \|\widehat{p}_n(\cdot\,|\, x_0+h_n)-p(\cdot\,|\,x_0+h_n)\|_1 \xrightarrow{n\rightarrow\infty} 0. \label{eq:kde boundary}
\end{align}
Next we show that $x_0+h_n$ is also a point of separation for all large $n$. Indeed, \ref{assumption:unpmor:separation} implies that 
\begin{align*}
    \underset{j\neq k}{\operatorname{min}}\,\, |m_j(x_0)-m_k(x_0)| \geq 2\operatorname{diam}(\operatorname{supp}(G_0))+\xi,\qquad \forall\,\, \vartheta\in\mathcal{V}_{x_0},
\end{align*}
for some slack $\xi>0.$ We have 
\begin{align*}
     &|m_j(x_0+h_n)-m_k(x_0+h_n)|\\
    &=|m_j(x_0+h_n)-m_j(x_0)+m_j(x_0)-m_k(x_0)+m_k(x_0)-m_k(x_0+h_n)|\\
    &\geq |m_j(x_0)-m_k(x_0)|-|m_j(x_0+h_n)-m_j(x_0)|-|m_k(x_0)-m_k(x_0+h_n)|\\
    &\geq 2\operatorname{diam}(\operatorname{supp}(G_0))+\xi-2\,\underset{\ell}{\operatorname{max}}\,\|m_{\ell}^{\prime}\|_{L^{\infty}[a,b]}h_n,\qquad \forall j\neq k.
\end{align*}
Therefore when $h_n\leq  \xi(4 \operatorname{max}_{\ell}\|m_{\ell}^{\prime}\|_{L^{\infty}[a,b]})^{-1}$ say, the collections $\mathscr{U}_n:=\{p(\cdot\,|\,x_0+h_n),\vartheta\in\mathcal{V}_{x_0}\}$ satisfy \ref{assumption:D3} with a uniform slack $\xi/2$. Similarly, \ref{assumption:unpmor:regularity} implies the $\mathscr{U}_n$'s satisfy \ref{assumption:D1} for all $n$ large with the same upper bound since $|\mu_k|=|m_k(x_0+h_n)|\leq \operatorname{max}_k \|m_k\|_{L^{\infty}[a,b]}+h_n\operatorname{max}_k \|m_k^{\prime}\|_{L^{\infty}[a,b]}$. Finally \ref{assumption:D2} is satisfied as before. Hence we can apply the estimation procedure to the sequence $\widehat{p}_n(\cdot\,|\, x_0+h_n)$ (or the sequence of families $\mathscr{U}_n$) and conclude by Proposition \ref{lemma:conditioning} the desired uniform consistency result. 
 
\qed

\subsubsection{Proof of Lemma \ref{lemma:Hausdorff error regression function}}
Let 
\begin{align*}
    \widehat{V}=\sum_{k=1}^K \lambda_k\delta_{\widehat{m}_{n,k}(x)}, \quad \quad V=\sum_{k=1}^K \lambda_k \delta_{m_k(x)}.
\end{align*}
By Lemma \ref{lemma:moc}, we have 
\begin{align}
     W_1(\widehat{V},V)\leq C_B F_{J,\sigma}(\|\widehat{V}\ast f-V\ast f\|_1),  \label{eq:moc MoR}
\end{align}
where $C_B$ is a constant depending only on $B$ and $F_{J,\sigma}$ is a strictly increasing function that depends only on $J,\sigma$ and satisfies $F_{J,\sigma}(d)\xrightarrow{d\rightarrow 0}0$. Now we have 
\begin{align*}
    \|\widehat{V}\ast f-V\ast f\|_1
    &=\left\|\sum_{k=1}^K \lambda_k f(\cdot-\widehat{m}_{n,k}(x))- \sum_{k=1}^K \lambda_k f(\cdot-m_k(x))\right\|_1\\
    &\leq \left\|\sum_{k=1}^K \lambda_k f(\cdot-\widehat{m}_{n,k}(x))- \sum_{k=1}^K \widehat{\lambda}_{n,k} \widehat{f}_n(\cdot-\widehat{m}_{n,k}(x))\right\|_1\\
    &\quad +\left\|\sum_{k=1}^K \widehat{\lambda}_{n,k} \widehat{f}_n(\cdot-\widehat{m}_{n,k}(x))-\sum_{k=1}^K \lambda_k f(\cdot-m_k(x))\right\|_1=:E_1+E_2.
\end{align*}
For $E_1$ we have
\begin{align}
    E_1&
    \leq \left\|\sum_{k=1}^K \lambda_k f(\cdot-\widehat{m}_{n,k}(x))- \sum_{k=1}^K \lambda_k \widehat{f}_n(\cdot-\widehat{m}_{n,k}(x))\right\|_1\nonumber\\
    &\quad +\left\|\sum_{k=1}^K \lambda_k \widehat{f}_n(\cdot-\widehat{m}_{n,k}(x))-\sum_{k=1}^K \widehat{\lambda}_{n,k} \widehat{f}_n(\cdot-\widehat{m}_{n,k}(x))\right\|_1
    \leq \|\widehat{f}_n-f\|_1+\sum_{k=1}^K |\widehat{\lambda}_{n,k}-\lambda_k|.\label{eq:E1 uc}
\end{align}
For $E_2$ we notice that since $\sum_{k=1}^K \lambda_k f(y-m_k(x))$ is the true conditional density $p(\,\cdot\given x)$, we have 
\begin{align}
    E_2&=\left\|\sum_{k=1}^K \widehat{\lambda}_{n,k} \widehat{f}_n(y-\widehat{m}_{n,k}(x))-p(\,\cdot\given x)\right\|_1\nonumber \\
    &\leq \left\|\sum_{k=1}^K \widehat{\lambda}_{n,k} \widehat{f}_n(y-\widehat{m}_{n,k}(x))-\widehat{p}_{n}(\,\cdot\given x)\right\|_1 +\|\widehat{p}_{n}(\,\cdot\given x)-p(\,\cdot\given x)\|_1\nonumber\\
    &=:E_3+\|\widehat{p}_{n}(\,\cdot\given x)-p(\,\cdot\given x)\|_1. \label{eq:E2 uc}
\end{align}
The term $E_3$ is exactly in the form of the definition \eqref{eq:MDE}, and can be bounded by replacing $\widehat{m}_{n,k}(x)$ with $m_k(x)$ (since $\widehat{m}_{n,k}(x)$ is a minimizer). So we have 
\begin{align}
    E_3 &\leq \left\|\sum_{k=1}^K \widehat{\lambda}_{n,k} \widehat{f}_n(y-m_k(x))- \widehat{p}_{n}(\,\cdot\given x)\right\|_1\nonumber\\
    &\leq \left\|\sum_{k=1}^K \widehat{\lambda}_{n,k} \widehat{f}_n(y-m_k(x))- p(\,\cdot\given x)\right\|_1
    +\|\widehat{p}_{n}(\,\cdot\given x)-p(\,\cdot\given x)\|_1\nonumber\\
    &=\left\|\sum_{k=1}^K \widehat{\lambda}_{n,k} \widehat{f}_n(y-m_k(x))-\sum_{k=1}^K \lambda_k f(y-m_k(x)) \right\|_1+\|\widehat{p}_{n}(\,\cdot\given x)-p(\,\cdot\given x)\|_1\nonumber \\
    &\leq \|\widehat{f}_n-f\|_1 + \sum_{k=1}^K |\widehat{\lambda}_{n,k}-\lambda_k|+\|\widehat{p}_{n}(\,\cdot\given x)-p(\,\cdot\given x)\|_1. \label{eq:E3 uc}
\end{align}
Therefore combining \eqref{eq:E1 uc}, \eqref{eq:E2 uc} and \eqref{eq:E3 uc} we have 
\begin{align*}
    \|\widehat{V}\ast f-V\ast f\|_1 &\leq 2\left(\|\widehat{f}_n-f\|_1 + \sum_{k=1}^K |\widehat{\lambda}_{n,k}-\lambda_k|+\|\widehat{p}_{n}(\,\cdot\given x)-p(\,\cdot\given x)\|_1\right),
\end{align*}
which together with \eqref{eq:moc MoR} and the fact that $F_J$ is strictly increasing implies $W_1(\widehat{V},V) \leq e_n(x)$ (after redefining $F_{J,\sigma}$ as a dilation of itself). Now \cite[][Lemma~2]{wu2020optimal} implies that 
\begin{align*}
    d_{\operatorname{Haus}}\left(\{\widehat{m}_{n,k}(x)\}_{k=1}^K,\{m_k(x)\}_{k=1}^K\right) \leq \frac{W_1(\widehat{V},V)}{\operatorname{min}_k \lambda_k}\leq \frac{e_n(x)}{\operatorname{min}_k \lambda_k},
\end{align*}
where $d_{\operatorname{Haus}}(A,B)=\operatorname{max}\{\operatorname{sup}_{x\in A}\operatorname{dist}(x,B), \operatorname{sup}_{y\in B}\operatorname{dist}(y,A)\}$ is the Hausdorff distance between $A$ and $B$.  
Then for a fixed $k$, there exists $\Pi(k)\in \{1,\ldots,K\}$ so that
\begin{align}
    |\widehat{m}_{n,\Pi(k)}(x)-m_k(x)|\leq \frac{e_n(x)}{\operatorname{min}_k \lambda_k}, \label{eq:mean bound reduction}
\end{align}
proving the first part of the lemma. 
We notice that $\Pi$ is single-valued; otherwise the pigeonhole principle would imply that for some $j^{\ast}\neq k^{\ast}$ there is a common $\widehat{m}_{n,\Pi(j^{\ast})}=\widehat{m}_{n,\Pi(k^{\ast})}$ so that 
\begin{align*}
    |m_{j^{\ast}}(x)-m_{k^{\ast}}(x)|\leq |m_{j^{\ast}}(x)-\widehat{m}_{n,\Pi(j^{\ast})}(x)|+|\widehat{m}_{n,\Pi(k^{\ast})}(x)-m_{k^{\ast}}(x)|\leq \frac{2e_n(x)}{\operatorname{min}_k \lambda_k},
\end{align*}
which contradicts \eqref{eq:assump reduction}. This also implies that $\Pi$ is injective and is hence a permutation. 
Let 
\begin{align*}
    r_n(x)=\frac{2e_n(x)}{\operatorname{min}_k\lambda_k\wedge \operatorname{min}_{j\neq k}|\lambda_j-\lambda_k|}.
\end{align*}
For any $i\neq j$, we have $i=\Pi(i^{\ast})$ and $j=\Pi(j^{\ast})$ for some $i^{\ast}$, $j^{\ast}$, and 
\begin{align*}
    &|\widehat{m}_{n,i}(x)-\widehat{m}_{n,j}(x)|\\
    &=|\widehat{m}_{n,\Pi(i^{\ast})}(x)-\widehat{m}_{n,\Pi(j^{\ast})}(x)|\\
    &\geq |m_{i^{\ast}}(x)-m_{j^{\ast}}(x)|-|m_{i^{\ast}}(x)-\widehat{m}_{n,\Pi(i^{\ast})}(x)|-|m_{j^{\ast}}(x)-\widehat{m}_{n,\Pi(j^{\ast})}(x)|\\
    &> \frac{6e_n(x)}{\operatorname{min}_k\lambda_k\wedge \operatorname{min}_{j\neq k}|\lambda_j-\lambda_k|}-\frac{2e_n(x)}{\operatorname{min}_k\lambda_k}\geq  2r_n(x).
\end{align*}
In particular the $\widehat{m}_{n,k}$'s are also well-separated and
hence for a fixed $k$ we have 
\begin{align*}
    [m_k(x)-r_n(x),m_k(x)+r_n(x)] \cap \{\widehat{m}_{n,1}(x),\ldots,\widehat{m}_{n,K}(x)\}&=\widehat{m}_{n,\Pi(k)}(x)\\
    [\widehat{m}_{n,\Pi(k)}(x)-r_n(x),\widehat{m}_{n,\Pi(k)}(x)+r_n(x)] \cap \{m_1(x),\ldots,m_K(x)\}&=m_k(x).
\end{align*}
Then \cite[][Lemma~3]{wu2020optimal}  applied with $\delta=r_n(x)$ implies
\begin{align*}
    \underset{k}{\operatorname{max}}\, |\lambda_k -\lambda_{\Pi(k)}|\leq \frac{W_1(\widehat{V},V)}{r_n(x)}\leq \frac{e_n(x)}{r_n(x)}\leq \frac12\underset{j\neq k}{\operatorname{min}}\,|\lambda_j-\lambda_k|,
\end{align*}
which implies that the left hand side equals zero and $\Pi$ is the identity. Therefore \eqref{eq:mean bound reduction} reads 
\begin{align}
    \underset{k}{\operatorname{max}}\, |\widehat{m}_{n,k}(x)-m_k(x)| &\leq \frac{e_n(x)}{\operatorname{min}_j\lambda_j},
\end{align} 
as desired.
\qed

\subsubsection{Proof of Proposition \ref{prop:step 3 uniform consistency}}
  Under \ref{assumption:unpmor:regularity}-\ref{assumption:unpmor:separation}, Proposition \ref{prop:step 2 uniform consistency} implies that   
\begin{align*}
    \underset{\mathcal{V}_{x_0}}{\operatorname{sup}}\,\,\mathbb{P}\left(\|\widehat{f}_n-f\|_1 \vee \underset{k}{\operatorname{max}}\,|\widehat{\lambda}_{n,k}-\lambda_k| >\varepsilon_n\right) \xrightarrow{n\rightarrow\infty} 0
\end{align*}
for some sequence $\varepsilon_n\rightarrow 0$. This together with Proposition \ref{prop:step 1 uniform consistency} implies that the error $e_n(x)$ in Lemma \ref{lemma:Hausdorff error regression function}   (which holds uniformly under \ref{assumption:unpmor:FT})   satisfies for some sequence $\gamma_n\rightarrow 0$ that    
\begin{align*}
    \underset{\mathcal{V}_{x_0}}{\operatorname{sup}}\,\, \mathbb{P}\Big(\underset{x\in[a+h_n,b-h_n]}{\operatorname{sup}} e_n(x)>\gamma_n\Big) \leq \underset{\mathcal{V}_{x_0}}{\operatorname{sup}} \underset{x\in[a+h_n,b-h_n]}{\operatorname{sup}} \mathbb{P}(e_n(x)>\gamma_n)\xrightarrow{n\rightarrow\infty} 0.
\end{align*}
Therefore it suffices to consider the event $\{\operatorname{sup}_{x\in [a+h_n,b-h_n]}  e_n(x)\leq \gamma_n\}$.   
Denote
\begin{align*}
    &\Lambda=\underset{\vartheta\in\mathcal{U}_{x_0}}{\operatorname{inf}}\, \underset{k}{\operatorname{min}}\, \lambda_k ,\qquad S_{\lambda}=\underset{\vartheta\in\mathcal{U}_{x_0}}{\operatorname{inf}}\, \underset{j\neq k}{\operatorname{min}}\, |\lambda_j-\lambda_k|,\quad \\
    &B=\underset{\vartheta\in\mathcal{U}_{x_0}}{\operatorname{sup}}\,\big( \underset{k}{\operatorname{max}}\, \|m_k\|_{L^{\infty}[a,b]} \vee \underset{k}{\operatorname{max}}\, \|m_k^{\prime}\|_{L^{\infty}[a,b]}\big)\\
    &\mathcal{O}_n=\left\{x: \exists \,j\neq k ,\,  |m_j(x)-m_k(x)|\leq 6\Lambda^{-1}S_{\lambda}^{-1}\gamma_n \right\}.
\end{align*}
We have by Lemma \ref{lemma:Hausdorff error regression function}
 \begin{align}
    &\int_{[a,b]} |\widehat{m}_{n,k}(x)-m_k(x)| dx \nonumber\\
    &= \int_{[a+h_n,b-h_n]\cap \mathcal{O}_n^c}  |\widehat{m}_{n,k}(x)-m_k(x)| dx+\int_{[a+h_n,b-h_n]^c \cup \mathcal{O}_n}  |\widehat{m}_{n,k}(x)-m_k(x)| dx  \nonumber \\
    & \leq \frac{(b-a)\gamma_n}{\operatorname{min}_k\lambda_k}+2B (|\mathcal{O}_n|+2h_n)\xrightarrow{n\rightarrow\infty} 0 \label{eq:step 3 2}
\end{align} 
uniformly over $\mathcal{V}_{x_0}$, where we have used \ref{assumption:unpmor:sep of reg} in the last step. The result then follows.
\qed

\subsubsection{Proof of Theorem \ref{thm:uniform consistency}}
 
Theorem \ref{thm:uniform consistency} now follows as an immediate corollary to Propositions~\ref{prop:step 2 uniform consistency} and~\ref{prop:step 3 uniform consistency}. In more detail, it is enough to construct uniformly consistent estimators of the error density $f$, the mixture weights $\lambda_k$, and the regression functions $m_k$. We recall the main steps in constructing these estimators here (see Section~\ref{sec:est} for details):
\begin{enumerate}
    \item First, estimate the conditional density $p(\cdot\given x)$ via KDE (cf. Section~\ref{sec:step 1}).
    \item Next, estimate $\lambda_{k}$ and $f$ by interpreting $p(\cdot\given x_{0})$ as a vanilla mixture model as in \eqref{eq:defn:npmor}, and using the project-smooth-denoise construction to define estimators $\widehat{f}_n$ and $\{\widehat{\lambda}_{n,k}\}_{k=1}^K$ (cf. Section~\ref{sec:step 2}). 
    \item Finally, estimate the regression functions $m_k$ via minimum distance estimators $\{\widehat{m}_{n,k}\}_{k=1}^K$ (cf. Section~\ref{sec:step 3}).
\end{enumerate}
By \ref{assumption:unpmor:regularity}-\ref{assumption:unpmor:separation}, Proposition \ref{prop:step 2 uniform consistency} establishes uniform consistency of $\widehat{f}_n$, $\{\widehat{\lambda}_{n,k}\}_{k=1}^K$. Moreover, by \ref{assumption:unpmor:regularity}-\ref{assumption:unpmor:sep of reg}, Proposition \ref{prop:step 3 uniform consistency} establishes uniform consistency of $\{\widehat{m}_{n,k}\}_{k=1}^K$. 
 
\qed

\subsection{Proofs for Section \ref{sec:npmix}}\label{app:proof:sec5}

In this appendix we prove our second main result (Theorem~\ref{prop:npmix uniform consistency}). The proof of this theorem can be found at the end of this appendix, after the proofs of the technical lemmas have been presented.

\subsubsection{Proof of Lemma \ref{lemma:W distance between G hat and G}}

By \cite[][Theorem~2(2)]{nguyen2013} applied in their notation with $f=\phi_{\sigma}$ so that $\beta=2$, we have 
\begin{align*}
    W_1(\widehat{G}_{n},G)\leq C_M \left[-\log \|\widehat{Q}_{n}-p\|_1\right]^{-1/2}.
\end{align*}
Here we remark that the constant in their theorem actually depends on the compact set $\Theta$ in their notation and explains the subscript in $C_M$ in our case. Now recalling the definitions of $\widehat{Q}_{n}$ and $Q_n$, we have 
\begin{align*}
    \|\widehat{Q}_{n}-p\|_1 &\leq \|\widehat{Q}_{n}-\widehat{p}_n\|_1+\|\widehat{p}_n-p\|_1\\
    &\leq \|Q_n-\widehat{p}_n\|_1+\|\widehat{p}_n-p\|_1 
     \leq \|Q_n-p\|_1+2\|\widehat{p}_n-p\|_1.
\end{align*}
The desired result follows.
\qed

\subsubsection{Proof of Lemma \ref{lemma:total weights of outlier}}

Let $\widehat{X}_{n},X$ be any coupling between $\widehat{G}_{n},G$. Then
\begin{align*}
    W_1(\widehat{G}_{n},G)=\mathbb{E}|\widehat{X}_{n}-X| &=\sum_{\ell=1}^{L_n} w_{\ell} \mathbb{E}\left[|\widehat{X}_{n}-X| \,\Big|\, \widehat{X}_{n}=a_{\ell}\right]\geq \eta\sum_{\ell \in A_{\eta}} w_{\ell} .
\end{align*}
\qed

\subsubsection{Proof of Lemma \ref{lemma:thresholding and Ek hat}}

By the definition of $\widehat{g}_n$, we have
\begin{align*}
    \widehat{g}_n(x)&=\sum_{\ell \in A_{\delta_n}} w_{\ell} I_{\delta_n}(x-a_{\ell})+ \sum_{\ell \notin A_{\delta_n}}  w_{\ell} I_{\delta_n}(x-a_{\ell}) \\
    &\leq I_{\delta_n}(0)\sum_{\ell \in A_{\delta_n}} w_{\ell}+\sum_{\ell \notin A_{\delta_n}}  w_{\ell} I_{\delta_n}(x-a_{\ell})
    \leq 2^{-1}\delta_n^{-2} W_1(\widehat{G}_{n},G) +\sum_{\ell \notin A_{\delta_n}}  w_{\ell} I_{\delta_n}(x-a_{\ell}), 
\end{align*}
where $A_{\delta_n}=\{\ell:\operatorname{dist}(a_{\ell},\operatorname{supp}(G))>\delta_n\}$ as in Lemma \ref{lemma:total weights of outlier}. Now for any $x$ such that $\operatorname{dist}(x,\operatorname{supp}(G))>2\delta_n$, we have $|x-a_{\ell}|>\delta_n$ for any $\ell\notin A_{\delta}$ and so the second sum in the above equals zero. In particular we have shown that \begin{align}
    \widehat{g}_n \leq 2^{-1}\delta_n^{-2} W_1(\widehat{G}_{n},G) \quad\text{on}\quad \{x:\operatorname{dist}(x,\operatorname{supp}(G))>2\delta_n\}, \label{eq:outlier density level}
\end{align}
which suggests a threshold $t_n\geq 2^{-1}\delta_n^{-2}W_1(\widehat{G}_n,G)$ and we can write 
\begin{align*}
   \{x: \widehat{g}_n(x)>t_n\}&   =\{x: \widehat{g}_n(x)>t_n\}\cap \{x:\operatorname{dist}(x,\operatorname{supp}(G))\leq 2\delta_n\}\\
  & =\bigcup_{k=1}^K \{x: \widehat{g}_n(x)>t_n\}\cap S_k(2\delta_n)=:\bigcup_{k=1}^K \widehat{E}_k.
\end{align*}
It now remains to show that the $\widehat{E}_k$'s are nonempty.

By \cite[][Lemma~2.1]{chae2020wasserstein} we have 
\begin{align*}
    \|\widehat{G}_{n}\ast I_{\delta_n}-G\ast I_{\delta_n}\|_1 \leq \underset{s\neq t}{\operatorname{sup}}\frac{\|I_{\delta_n}(\cdot-s)-I_{\delta_n}(\cdot-t)\|_1}{|s-t|}W_1(\widehat{G}_{n},G) \leq  \delta_n^{-1}W_1(\widehat{G}_{n},G),
\end{align*}
which in particular gives 
\begin{align} 
    \left|\int_{S_k(2\delta_n)}\widehat{g}_n(x)dx-\int_{S_k(2\delta_n)}g_n(x)dx\right|\leq \delta_n^{-1}W_1(\widehat{G}_{n},G), \label{eq:weight approx 1}
\end{align}
where $g_n(x)=\int_{\mathbb{R}} I_{\delta_n}(x-y)dG(y)$ is the density of $G\ast I_{\delta_n}$. 
Now notice that $\int_{S_k(2\delta_n)} g_n\,dx=\lambda_k$. Indeed we have 
\begin{align*}
    \int_{S_k(2\delta_n)}g_n\,dx &= \sum_{j=1}^K \int_{S_j} \int_{S_k(2\delta_n)} I_{\delta_n}(x-y) dx dG(y) \\
    &= \int_{S_k} \int_{S_k(2\delta_n)} I_{\delta_n}(x-y) dx dG(y) +\sum_{j\neq k} \int_{S_j} \int_{S_k(2\delta_n)} I_{\delta_n}(x-y) dx dG(y)\\
    &=:i_1+i_2.
\end{align*}
The separation assumption \eqref{assm:npmix:sep} together with that   $3\delta_n<\operatorname{min}_{j\neq k}\operatorname{dist}(S_j,S_k)$   implies $|x-y|>\delta_n$ for $x\in S_k(2\delta_n)$, $y\in S_j$ with $j\neq k$ and hence $i_2=0$. Similarly 
\begin{align*}
    i_1=\int_{S_k} \int_{\mathbb{R}} I_{\delta_n}(x-y) dx dG(y) = \int_{S_k} dG(y)=\lambda_k.
\end{align*}
Therefore \eqref{eq:weight approx 1} implies 
\begin{align}
    \left|\int_{S_k(2\delta_n)}\widehat{g}_n(x)dx-\lambda_k\right|\leq \delta_n^{-1}W_1(\widehat{G}_{n},G). \label{eq:weight approx 2}
\end{align}
Now if $\widehat{E}_k=\emptyset$ for some $k$, i.e., $\widehat{g}_n\leq t_n$ on $S_k(2\delta_n)$, then \eqref{eq:weight approx 2} implies 
\begin{align*}
    \lambda_k -\delta_n^{-1}W_1(\widehat{G}_{n},G) \leq \int_{S_k(2\delta_n)}\widehat{g}_n(x)dx \leq (D+4\delta_n)t_n,
\end{align*}
contradicting \eqref{eq:thresholding assumption}. This concludes the proof. 
\qed

\subsubsection{Proof of Lemma \ref{lemma:Ek}}

Let $\hat{a}_k=\operatorname{inf} \widehat{E}_k$ and $\hat{b}_k=\operatorname{sup} \widehat{E}_k$. Define
\begin{align*}
    a_k=\frac12(\hat{b}_{k-1}+\hat{a}_k),\quad b_k=\frac12(\hat{b}_k+\hat{a}_{k+1}),\quad k=2,\ldots,K-1
\end{align*}
with $a_1=-\infty$ and $b_K=\infty$ and let $E_k=[a_k,b_k)$. It is clear that $E_k$'s form a partition of $\mathbb{R}$. To see that $S_k(\xi)\subset E_k$, it suffices to show that $a_k\leq \operatorname{inf}S_k-\xi$ and $b_k\geq \operatorname{sup}S_k+\xi$. By definition of $a_k$ and the fact that $\widehat{E}_k\subset S_k(2\delta)$ we have 
\begin{align*}
    a_k & \leq \frac12 \left(\sup S_{k-1}(2\delta) +\sup S_k(2\delta)\right)\\
    & \leq \frac12 \left(\sup S_{k-1}+\sup S_k+4\delta\right)\\
    & \leq \frac12 \left[(\inf S_k -D-4\xi) + \sup S_k +4\delta\right]\\
    & = \frac12 \left[\inf S_k +(\sup S_k-D)+ 4\delta-4\xi\right]
     \leq \inf S_k-\xi,
\end{align*}
where we have used in the last step the fact that $\sup S_k-D \leq \inf S_k$ and $2\delta<\xi$. The proof for $b_k$ is similar. 
\qed

\subsubsection{Proof of Proposition \ref{lemma:conditioning}}

To simplify the notation, we will suppress the dependence of $\widehat{G}_{n}, \widehat{F}_{n,k}, \widehat{\lambda}_{n,k}, \widehat{\mu}_{n,k}$ on $n$ and denote them as $\widehat{G}, \widehat{F}_k, \widehat{\lambda}_k,\widehat{\mu}_k$. Let $F_k:=f_k(\cdot-\mu_k)$, $\widetilde{G}_k:=G_k(\cdot-\mu_k)$ and $\widehat{G}_k:=\widehat{G}(\,\cdot\given E_k)$ so that $p=\sum_{k=1}^K \lambda_kF_k$, $G=\sum_{k=1}^K \lambda_k \widetilde{G}_k$ and $\widehat{G}=\sum_{k=1}^K \widehat{\lambda}_{k}\widehat{G}_k$. As a result of these notations, we can write $F_k=\phi_{\sigma}\ast \widetilde{G}_k$ and $\widehat{F}_k=\phi_{\sigma}\ast \widehat{G}_k$. 

The first step is to show that the analysis of $\|\widehat{f}_k-f_k\|_1$ reduces to bounding $\|\widehat{F}_k- F_k\|_1$.
Indeed, we have 
\begin{align*}
    \|\widehat{f}_{k}-f_k\|_1
    &\leq \|\widehat{F}_{k}(\cdot+\widehat{\mu}_{k})-F_k(\cdot+\widehat{\mu}_{k})\|_1+\|F_k(\cdot+\widehat{\mu}_{k})-F_k(\cdot+\mu_k)\|_1\\
    & =\|\widehat{F}_{k}-F_k\|_1 + \int \left|\int_0^1 (\widehat{\mu}_{k}-\mu_k) F_k^{\prime}(x+t(\widehat{\mu}_{k}-\mu_k))dt\right| dx\\
    & \leq \|\widehat{F}_{k}-F_k\|_1+ |\widehat{\mu}_{k}-\mu_k|\, \|F_k^{\prime}\|_1.
\end{align*}
Since $F_k=\int_{\mathbb{R}} \phi_{\sigma}(x-\theta)d\widetilde{G}_k(\theta)$, we have 
$F_k^{\prime}=\int_{\mathbb{R}} \left(\frac{\theta-x}{\sigma^2}\right)\phi_{\sigma}(x-\theta)d\widetilde{G}_k(\theta)$
and 
\begin{align*}
    \|F_k^{\prime}\|_1\leq \int_{\mathbb{R}} \int_{\mathbb{R}} \left|\frac{\theta-x}{\sigma^2}\right|\phi_{\sigma}(x-\theta) dx\, d\widetilde{G}_k(\theta)= \sigma^{-2}\int_{\mathbb{R}}|x|\phi_{\sigma}(x)dx=\sqrt{\frac{2}{\pi\sigma^2}}.
\end{align*}
Furthermore we have 
\begin{align*}
    |\widehat{\mu}_{k}-\mu_k|
    &\leq \int_{\mathbb{R}}\left|x \left(\widehat{F}_{k}^{1/2}+F_k^{1/2}\right)\right| \left|\widehat{F}_{k}^{1/2}-F_k^{1/2}\right|dx\\
    & \leq \sqrt{\int_{\mathbb{R}}x^2\left(\widehat{F}_{k}^{1/2}+F_k^{1/2}\right)^2dx}\sqrt{\int_{\mathbb{R}} \left(\widehat{F}_{k}^{1/2}-F_k^{1/2}\right)^2dx}\\
    & \leq \sqrt{2\int_{\mathbb{R}}x^2\left(\widehat{F}_{k}+F_k\right)dx}\sqrt{\|\widehat{F}_{k}-F_k\|_1}.
\end{align*}
Recall $\operatorname{supp}(\widetilde{G}_k)\cup \operatorname{supp}(\widehat{G}_k)\subset [-M,M]$, so that 
\begin{align*}
    \int_{\mathbb{R}} x^2  F_k  dx
    &= \int_{|\theta|\leq M} \int_{\mathbb{R}} x^2\phi_{\sigma}(x-\theta)dx\,d\widetilde{G}_k(\theta)\\
    &=\int_{|\theta|\leq M} \int_{\mathbb{R}}|x+\theta|^2 \phi_{\sigma}(x)dx\,d\widetilde{G}_k(\theta)\leq \sigma^2 +M^2
\end{align*}
and similarly for $\widehat{F}_k$. 
Therefore we have 
\begin{align}
    \|\widehat{f}_{k}-f_k\|_1 \leq \|\widehat{F}_{k}-F_k\|_1+ \sqrt{\frac{8(\sigma^2+M^2)}{\pi\sigma^2}}\sqrt{\|\widehat{F}_{k}-F_k\|_1}. \label{eq:gtl reduction 1}
\end{align}

Next we shall show that to bound $|\widehat{\lambda}_k-\lambda_k|$ and $\|\widehat{F}_{k}-F_k\|_1$, it suffices to bound 
$\|\widehat{\lambda}_{k}\widehat{F}_{k}-\lambda_k F_k\|_1$. 
Indeed we have 
\begin{align*}
    |\widehat{\lambda}_{k}-\lambda_k|=\left|\int \widehat{\lambda}_{k}\widehat{F}_{k}-\lambda_kF_k\right|\leq \|\widehat{\lambda}_{k}\widehat{F}_{k}-\lambda_k F_k\|_1
\end{align*}
and 
\begin{align*}
    \|\widehat{\lambda}_{k}\widehat{F}_{k}-\lambda_k F_k\|_1 &\geq \|\lambda_k \widehat{F}_{k}-\lambda_k F_k\|_1-\|\widehat{\lambda}_{k} \widehat{F}_{k}-\lambda_k \widehat{F}_{k}\|_1 \\
    &= \lambda_k\|\widehat{F}_{k}-F_k\|_1-|\widehat{\lambda}_{k}-\lambda_k|
\end{align*}
so that 
\begin{align}
    |\widehat{\lambda}_{k}-\lambda_k|+\lambda_k\|\widehat{F}_{k}-F_k\|_1 \leq 3\|\widehat{\lambda}_{k}\widehat{F}_{k}-\lambda_k  F_k\|_1.  \label{eq:gtl reduction 2}
\end{align}
Finally we shall bound $\|\widehat{\lambda}_{k}\widehat{F}_{k}-\lambda_k F_k\|_1$, which we decompose as three terms by introducing a mollifier
\begin{align*}
    \|\widehat{\lambda}_{k}\widehat{F}_{k}-\lambda_k F_k\|_1 &
    =\left\|\int_{E_k} \phi_{\sigma}(x-\theta) d\widehat{G}(\theta)-\int_{E_k} \phi_{\sigma}(x-\theta)d\widetilde{G}(\theta)\right\|_1\leq J_1+J_2+J_3,
\end{align*}
where
\begin{align*}
    J_1&:=\left\|\int_{E_k} \phi_{\sigma}(x-\theta) d\widehat{G}(\theta)-\int_{E_k} \phi_{\sigma}(x-\theta) d(\widehat{G}\ast H_{\delta})(\theta)\right\|_1 \\
    J_2&:= \left\|\int_{E_k} \phi_{\sigma}(x-\theta) d(\widehat{G}\ast H_{\delta})(\theta)-\int_{E_k} \phi_{\sigma}(x-\theta) d(\widetilde{G}\ast H_{\delta} )(\theta)\right\|_1 \\
    J_3&:= \left\|\int_{E_k} \phi_{\sigma}(x-\theta) d(\widetilde{G}\ast H_{\delta})(\theta)-\int_{E_k} \phi_{\sigma}(x-\theta) d\widetilde{G}(\theta)\right\|_1.
\end{align*}
Here $H$ is a symmetric density function with bounded first moment whose Fourier transform is supported in $[-1,1]$ and $H_{\delta}=\delta^{-1}H(\delta^{-1}\cdot)$.

\vspace{0.2cm}
\noindent{\textbf{Bound for $J_1$}}: Recall $\widehat{G}=\sum_{k=1}^K \widehat{\lambda}_k \widehat{G}_k$ with $\operatorname{supp}(\widehat{G}_k)\subset E_k$. Then $\widehat{G}\ast H_{\delta}=\sum_{k=1}^{K}\widehat{\lambda}_{k} \widehat{G}_{k}\ast H_{\delta}$ and
\begin{align*}
    \int_{E_k}\phi_{\sigma}(x-\theta) d\widehat{G}(\theta)&=\widehat{\lambda}_k \int_{E_k} \phi_{\sigma}(x-\theta)d\widehat{G}_k(\theta)=\widehat{\lambda}_k \int_{\Theta} \phi_{\sigma}(x-\theta)d\widehat{G}_k(\theta)\\
    \int_{E_k}\phi_{\sigma}(x-\theta) d(\widehat{G}\ast H_{\delta})(\theta)&=\widehat{\lambda}_k\int_{E_k}\phi_{\sigma}(x-\theta)d(\widehat{G}_k\ast H_{\delta})(\theta) \\
    &\quad +\sum_{j\neq k}\widehat{\lambda}_j\int_{E_k} \phi_{\sigma}(x-\theta)d(\widehat{G}_j\ast H_{\delta})(\theta).
\end{align*}
We have 
\begin{align*}
    J_1&\leq \left\|\widehat{\lambda}_k \int_{\mathbb{R}} \phi_{\sigma}(x-\theta)d\widehat{G}_k(\theta)-\widehat{\lambda}_k\int_{\mathbb{R}}\phi_{\sigma}(x-\theta)d(\widehat{G}_k\ast H_{\delta})(\theta)\right\|_1 \\
    &\quad + \left\|\widehat{\lambda}_k\int_{\mathbb{R}}\phi_{\sigma}(x-\theta)d(\widehat{G}_k\ast H_{\delta})(\theta)- \widehat{\lambda}_k\int_{E_k}\phi_{\sigma}(x-\theta)d(\widehat{G}_k\ast H_{\delta})(\theta)\right\|_1 \\
    &\quad + \left\|\sum_{j\neq k}\widehat{\lambda}_j\int_{E_k} \phi_{\sigma}(x-\theta)d(\widehat{G}_j\ast H_{\delta})(\theta)\right\|_1=:e_1+e_2+e_3.
\end{align*}
By \cite[][Lemma~1]{nguyen2013}, we have 
\begin{align}
    e_1\leq W_1(\widehat{G}_k,\widehat{G}_k\ast H_{\delta})\leq C\delta, \label{eq:e1 G hat}
\end{align}
where the last step can be proved as in \cite[][Theorem~2]{nguyen2013}: letting $\theta\sim \widehat{G}_k$ and $\varepsilon \sim H_{\delta}$ gives $W_1(\widehat{G}_k,\widehat{G}_k\ast H_{\delta})\leq\mathbb{E}\|\theta-(\theta+\varepsilon)\|_1\leq C\delta$. 
For $e_2$ we have 
\begin{align*}
    e_2 \leq \widehat{\lambda}_k \left\|\int_{E_k^c} \phi_{\sigma}(x-\theta)d(\widehat{G}_k\ast H_{\delta})(\theta)\right\|_1=\widehat{\lambda}_k(\widehat{G}_k\ast H_{\delta})(E_k^c),
\end{align*}
where
\begin{align*}
    (\widehat{G}_k\ast H_{\delta})(E_k^c)&=\int_{E_k^c} \int_{E_k} H_{\delta}(\theta-z)d\widehat{G}_k(z) d\theta \\
    &=\int_{E_k} \int_{E_k^c} H_{\delta}(\theta-z)d\theta d\widehat{G}_k(z) \\
    &=\int_{S_k(\xi/2)} \int_{E_k^c} H_{\delta}(\theta-z)d\theta d\widehat{G}_k(z)+ \int_{E_k\backslash S_k(\xi/2)} \int_{E_k^c} H_{\delta}(\theta-z)d\theta d\widehat{G}_k(z)\\
    &=:i_1+i_2.
\end{align*}
Recall that $E_k\supset S_k(\xi)$ so we have $\operatorname{dist}(E_k^c,S_k(\xi/2))\geq \frac{\xi}{2}$. Then for $z\in S_k(\xi/2)$
\begin{align*}
    \int_{E_k^c}H_{\delta}(\theta-z)d\theta \leq \int_{|x|>\xi/2} H_{\delta}(x)dx=\int_{|x|>\xi/2\delta}H(x)dx\leq \frac{2\delta}{\xi} \int_{|x|>\xi/2\delta}|x|H(x)dx\leq \frac{2C\delta}{\xi}
\end{align*}
and hence $i_1\leq 2C\xi^{-1}\delta$.
For $i_2$, we have 
\begin{align*}
    i_2 \leq \widehat{G}_k(E_k \backslash S_k(\xi/2))=\frac{\widehat{G}(E_k\backslash S_k(\xi/2))}{\widehat{\lambda}_k}.
\end{align*}
Therefore 
\begin{align}
    e_2 \leq  \widehat{\lambda}_k(\widehat{G}_k\ast H_{\delta})(E_k^c)\leq 2C\widehat{\lambda}_k \xi^{-1}\delta +\widehat{G}(E_k\backslash S_k(\xi/2)), \label{eq:e2 G hat}
\end{align}
and further that 
\begin{align}
    e_3\leq \sum_{j\neq k}\widehat{\lambda}_j (G_j\ast H_{\delta})(E_k)\leq \sum_{j\neq k}\widehat{\lambda}_j (G_j\ast H_{\delta})(E_j^c) \leq \sum_{j\neq k}2C\widehat{\lambda}_j\xi^{-1}\delta+\widehat{G}(E_j\backslash S_j(\xi/2)). \label{eq:e3 G hat}
\end{align}
Combining \eqref{eq:e1 G hat}, \eqref{eq:e2 G hat} and \eqref{eq:e3 G hat} we get 
\begin{align}
    J_1\leq C\xi^{-1}\delta+ \widehat{G}(A_{\xi/2})\leq C_{\xi}\left[\delta+W_1(\widehat{G},G)\right], \label{eq:I1 G hat}
\end{align}
where $A_{\eta}=\{x: \operatorname{dist}(x,\operatorname{supp}(G))>\eta\}$ as in Lemma \ref{lemma:total weights of outlier}. 

\vspace{0.2cm}
\noindent{\textbf{Bound for $J_3$}}: The term $J_3$ can be bounded similarly as 
\begin{align}
    J_3 \leq C_{\xi}\delta. \label{eq:I3 G hat}
\end{align}
Note that since the support of $G$ is $\bigcup_{k=1}^K S_k$, the corresponding error term $G(A_{\delta})$ is zero.

\vspace{0.2cm}
\noindent{\textbf{Bound for $J_2$}}:
\begin{align*}
    J_2&\leq \int_{\mathbb{R}} \int_{E_k} \phi_{\sigma}(x-\theta)d|\widehat{G}\ast H_{\delta}-G\ast H_{\delta}|(\theta) dx\\
    &\leq\int_{\mathbb{R}} \int_{\mathbb{R}} \phi_{\sigma}(x-\theta)d|\widehat{G}\ast H_{\delta}-G\ast H_{\delta}|(\theta) dx\\
    &=\int_{\mathbb{R}}|\widehat{G}\ast H_{\delta}(\theta)-G\ast H_{\delta}(\theta)|d\theta\\
    &=\int_{|\theta|\leq M+\xi} |\widehat{G}\ast H_{\delta}(\theta)-G\ast H_{\delta}(\theta)|d\theta + \int_{|\theta|>M+\xi}|\widehat{G}\ast H_{\delta}(\theta)-G\ast H_{\delta}(\theta)|d\theta.
\end{align*}
The second term can be bounded by noticing that 
\begin{align*}
    \int_{|\theta|>M+\xi} G\ast H_{\delta}(\theta)d\theta
    =\int_{|z|\leq M} \int_{|\theta|>M+\xi} H_{\delta}(\theta-z)d\theta dG(z)
    \leq \int_{|x|>\xi/\delta}H(x)dx \leq C\xi^{-1}\delta 
\end{align*}
and similarly for $\int_{|\theta|>M+\xi} \widehat{G}\ast H_{\delta}(\theta)d\theta$. The first term can be bounded using Cauchy-Schwarz by 
\begin{align*}
    \sqrt{\int_{|\theta|\leq M+\xi}1d\theta \int_{|\theta|\leq M+\xi}|\widehat{G}\ast H_{\delta}(\theta)-G\ast H_{\delta}(\theta)|^2d\theta}\leq \sqrt{2M+2\xi}\|\widehat{G}\ast H_{\delta}-G\ast H_{\delta}\|_2 
\end{align*}
Letting 
$h_{\delta}=\mathcal{F}^{-1}(\mathcal{F}H_{\delta}/\mathcal{F}\phi_{\sigma})$ (since $\mathcal{F}H_{\delta}$ is continuous and compactly supported, and $\mathcal{F}\phi_{\sigma}$ is never zero, $\mathcal{F}H_{\delta}/\mathcal{F}\phi_{\sigma}\in L^1$ and $h_{\delta}$ is well-defined), we have $H_{\delta}=\phi_{\sigma}\ast h_{\delta}$ and then
\begin{align*}
    \widehat{G}\ast H_{\delta}&=(\widehat{G}\ast \phi_{\sigma})\ast h_{\delta}=\widehat{Q} \ast h_{\delta}\\
    G\ast H_{\delta}&=(G\ast \phi_{\sigma})\ast h_{\delta}=p \ast h_{\delta},
\end{align*}
where we recall $\widehat{Q}$ is defined in \eqref{eq:Q_L,n} and $p$ is the true conditional density. 
Thus by Young's inequality we have
\begin{align*}
    \|\widehat{G}\ast H_{\delta}-G\ast H_{\delta}\|_2=\|\widehat{Q}\ast h_{\delta}-p\ast h_{\delta}\|_2\leq \|\widehat{Q}-p\|_1 \|h_{\delta}\|_2
\end{align*}
and by Plancherel's identity
\begin{align*}
    \|h_{\delta}\|^2_2=\left\|\frac{\mathcal{F}H_{\delta}}{\mathcal{F}\phi_{\sigma}}\right\|_2^2\leq C \int_{|w|<1/\delta} \exp\left(\sigma^2w^2\right)dw\leq C\exp\left(\sigma^2\delta^{-2}\right),
\end{align*}
where we have used the fact that $\mathcal{F}H$ is supported on $[-1,1]$, which implies that $\mathcal{F}H_{\delta}$ is supported on $[-1/\delta,1/\delta]$. 
Therefore we have 
\begin{align}
    J_2 \leq C_{M,\xi} \left[\delta+\|\widehat{Q}-p\|_1\exp\left(2^{-1}\sigma^2\delta^{-2}\right) \right]. \label{eq:I2 G hat}
\end{align}
and furthermore by combining \eqref{eq:I1 G hat}, \eqref{eq:I3 G hat}, \eqref{eq:I2 G hat} we have 
\begin{align*}
    \|\widehat{\lambda}_k\widehat{F}_k-\lambda_kF_k\|_1 \leq C_{M,\xi}\left[\delta+\|\widehat{Q}-p\|_1\exp\left(2^{-1}\sigma^2\delta^{-2}\right) +W_1(\widehat{G},G)\right]. 
\end{align*}
Setting $\delta^{-2}=-\log \|\widehat{Q}-p\|_1$, we get 
\begin{align}
    \|\widehat{\lambda}_k\widehat{F}_k-\lambda_kF_k\|_1 \leq C_{M,\xi,\sigma}\left[ \left(-\log \|\widehat{Q}-p\|_1 \right)^{-1/2}+W_1(\widehat{G},G)\right]. \label{eq:gtl reduction 3}
\end{align}
The results follows by combining \eqref{eq:gtl reduction 1}, \eqref{eq:gtl reduction 2}, \eqref{eq:gtl reduction 3} and Lemma \ref{lemma:W distance between G hat and G}.
\qed

\subsubsection{Proof of Lemma \ref{lemma:g continuous density}}

We first bound the error $\|\widehat{g}_n-g\|_1$. By \cite[][Lemma~2.1]{chae2020wasserstein} we have 
\begin{align*}
    \|\widehat{G}_{n}\ast I_{\delta_n}-G\ast I_{\delta_n}\|_1\leq \underset{s\neq t}{\operatorname{sup}}\, \frac{\|I_{\delta_n}(\cdot-s)-I_{\delta_n}(\cdot-t)\|_1}{|s-t|} W_1(\widehat{G}_{n},G)=\delta_n^{-1}W_1(\widehat{G}_{n},G).
\end{align*}
Moreover 
\begin{align*}
    \|G\ast I_{\delta_n}-G\|_1 
    &\leq \int_{|x|\leq M+\delta_n} \int_{|y|\leq \delta_n} |g(x-y)-g(x)|I_{\delta}(y)dy dx \\
    &=\int_{|x|\leq M-\delta_n} \int_{|y|\leq \delta_n} |g(x-y)-g(x)|I_{\delta}(y)dy dx \\
    &\quad +\int_{M-\delta_n<|x|\leq M+\delta_n} \int_{|y|\leq \delta_n} |g(x-y)-g(x)|I_{\delta}(y)dy dx=:e_1+e_2.
\end{align*}
By H\"older continuity of $g$ we have
\begin{align*}
    e_1\leq \int_{|x|\leq M-\delta_n} \int_{|y|\leq \delta} C_{H} |y|^{\beta}I_{\delta}(y)dy dx \leq 2MC_{H} \delta_n^{\beta},
\end{align*}
where $C_H$ is the H\"older constant of $g$. 
By boundedness of $g$ we have 
\begin{align*}
    e_2\leq \int_{M-\delta_n<|x|\leq M+\delta_n} \int_{|y|\leq \delta} 2BI_{\delta}(y)dy dx\leq 8B\delta_n.
\end{align*}
Therefore 
\begin{align*}
    \|\widehat{g}_n-g\|_1\leq \delta_n^{-1}W_1(\widehat{G}_{n},G)+2MC_H\delta_n^{\beta} +8B\delta_n=:E_{n}
\end{align*}
and then
\begin{align*}
    \operatorname{Leb}( N_{t_n}):=\operatorname{Leb}( \{|\widehat{g}_n-g|>t_n\})\leq \frac{E_{n}}{t_n}.
\end{align*}
In particular we have $\{g>2t_n\}\backslash N_{t_n} \subset\{\widehat{g}_n>t_n\}$ and by choosing $t_n$ going to zero faster than $E_n$, we see that $|\widehat{g}_n-g|\leq t_n$ over $N_{t_n}^c$, whose size shrinks to zero. Similarly as in \eqref{eq:outlier density level}, we have
\begin{align*}
    \widehat{g}_n(x)\leq 2^{-1}\delta_n^{-2}W_1(\widehat{G}_n,G) \quad \quad \operatorname{on}\,\, \{x:\operatorname{dist}(x,\operatorname{supp}(G))>2\delta_n\}.
\end{align*}
Let $d_n\rightarrow 0$ be an upper bound of $W_1(\widehat{G}_n,G)$ as in the proof of Theorem \ref{prop:npmix uniform consistency}. Pick $\frac{1}{2+\beta}<\alpha<\frac12$ and set  $\delta_n=d_n^{\alpha}$, $t_n=2^{-1}d_n^{1-2\alpha}$. 
We have 
\begin{align}
    \{g>2t_n\}\backslash N_{t_n} \subset\{\widehat{g}_n>t_n\} \subset \{x:\operatorname{dist}(x,\operatorname{supp}(G))\leq 2\delta_n\}, \label{eq:approximating support}
\end{align}
with 
\begin{align*}
     \operatorname{Leb} (N_{t_n})\leq 
     2d_n^{\alpha} +4C_HMd_n^{\alpha\beta+2\alpha-1}+ 16Bd_n^{3\alpha-1}
    \leq (2+4C_HM+16B)d_n^{\alpha(2+\beta)-1},
\end{align*}
where we have used the fact that $0<\alpha(2+\beta)-1<\operatorname{min}\{\alpha,3\alpha-1\}$. Now write 
\begin{align*}
    \{\widehat{g}_n>t_n\} = \bigcup_{k=1}^K\{\widehat{g}_n>t_n\}\cap S_k(2\delta_n)=:\bigcup_{k=1}^K \widehat{E}_k. 
\end{align*}
Notice that if further $2t_n<b$, then the left hand side of \eqref{eq:approximating support} equals $\operatorname{supp}(G)\backslash N_{t_n}$ and so we have
\begin{align}
    S_k\backslash N_{t_n} \subset \widehat{E}_k \subset S_k(2\delta_n).  \label{eq:approximating each support}
\end{align}
Now since $g\geq b>0$, each $S_k$ is a connected set. It follows then from \eqref{eq:approximating each support} that $\widehat{E}_k$ can be written as a union of connected sets whose pairwise separation is smaller than $\operatorname{Leb}(N_{t_n})+2\delta_n$. At the same time the separation between the $\widehat{E}_k$'s are greater than $4\xi-4\delta_n$. Therefore if further $\operatorname{Leb}(N_{t_n})+2\delta_n<4\xi-4\delta_n$, single linkage clustering recovers the sets $\{\widehat{E}_k\}_{k=1}^K$. 

Similarly as in the proof of Lemma \ref{lemma:Ek}, let $\hat{a}_k=\operatorname{inf} \widehat{E}_k$ and $\hat{b}_k=\operatorname{sup}\widehat{E}_k$. Define 
\begin{align*}
    a_k=\frac12(\hat{b}_{k-1}+\hat{a}_k),\quad b_k=\frac12(\hat{b}_k+\hat{a}_{k+1}),\quad k=2,\ldots,K-1
\end{align*}
with $a_1=-\infty$ and $b_K=\infty$ and let $E_k=[a_k,b_k)$. It is clear that $E_k$'s form a partition of $\mathbb{R}$. To see that $S_k(\xi)\subset E_k$, it suffices to show that $a_k\leq \operatorname{inf}S_k-\xi$ and $b_k\geq \operatorname{sup}S_k+\xi$. By definition of $a_k$ and the fact that $S_k\backslash N_{t_n} \subset\widehat{E}_k\subset S_k(2\delta)$ we have 
\begin{align*}
    a_k & \leq \frac12 \left(\sup S_{k-1}+2\delta_n +\inf S_k+\operatorname{Leb}(N_{t_n})\right)\\
    & \leq \frac12 \left(2\inf S_k-4\xi +2\delta_n+\operatorname{Leb}(N_{t_n})\right)
     \leq \inf S_k-\xi,
\end{align*}
if $2\delta_n+\operatorname{Leb}(N_{t_n})<2\xi$. The proof for $b_k$ is similar. 
\qed

\subsubsection{Proof of Theorem \ref{prop:npmix uniform consistency}}
  Condition \ref{assumption:D1} ensures that the mixing measures $G=\sum_{k=1}^K\lambda_k G_k(\cdot-\mu_k)$ over $\mathscr{U}$ are all contained in the interval $[-M,M]$ for some $M>0$.   Lemma \ref{lemma:kde L1 rate} then implies that we can find a sequence $b_n$ converging to zero so that $\operatorname{sup}_{\mathscr{U}} \mathbb{E}\|\widehat{p}_n-p\|_1\leq b_n^2$, and then 
\begin{align*}
    \underset{\mathscr{U}}{\operatorname{sup}}\,\,\mathbb{P}(\|\widehat{p}_n-p\|_1 > b_{n})\leq b_{n} \xrightarrow{n\rightarrow \infty}0.
\end{align*}
It suffices to consider the event $A=\{\|\widehat{p}_n-p\|_1\leq b_n\}$, over which Lemma \ref{lemma:W distance between G hat and G} implies for that
\begin{align*}
    W_1(\widehat{G}_{n},G)\leq C_M[-\log(b_n+c_n)]^{-1/2}:=d_{n},
\end{align*}
where $c_n$ is a uniform upper bound of the saturation rate over $\mathscr{U}$ given by Lemma \ref{lemma:approx mixing measure}. 
Now let $\delta_{n}=d_{n}^{1/4}$ so that $t_{n}=d_{n}^{1/2}$. Then for all large $n$, the assumptions in Lemma \ref{lemma:thresholding and Ek hat} and \ref{lemma:Ek} are satisfied uniformly over $\mathscr{U}$   thanks to the conditions \ref{assumption:D1} and \ref{assumption:D3},   and the construction applied to any model $\vartheta\in\mathscr{U}$ gives sets $E_k$'s satisfying $E_k\supset S_k(\xi)$, where
\begin{align*}
    4\xi:=\underset{\vartheta\in\mathscr{U}}{\operatorname{inf}}\,\left[\underset{j\neq k}{\operatorname{min}} \, \operatorname{dist} \left(\operatorname{supp} (G_j(\cdot-\mu_j)),\,\operatorname{supp}( G_k(\cdot-\mu_k))\right)- \underset{k}{\operatorname{max}}\,\operatorname{diam}(\operatorname{supp}(G_k))\right].
\end{align*}
Then Proposition \ref{lemma:conditioning} implies that on the event $\{\|\widehat{p}_n-p\|_1\leq b_n\}$ we have 
\begin{align*}
    \underset{\mathscr{U}}{\operatorname{sup}}\,\left[\underset{k}{\operatorname{max}}\, \left(|\widehat{\lambda}_{n,k}-\lambda_k|\vee \|\widehat{f}_{n,k}-f_k\|_1\right) \right]\leq C_{M,\xi,\sigma}(\alpha_{n}+\sqrt{\alpha_{n}}),
\end{align*}
where $\alpha_{n}=\Lambda^{-1} [-\log (c_n+b_n)]^{-1/2}$,   and $\Lambda=\operatorname{inf}_{\mathscr{U}}\,\operatorname{min}_k \lambda_k>0$ by \ref{assumption:D2}.   Setting $\varepsilon_{n}=C_{M,\xi,\sigma}(\alpha_{n}+\sqrt{\alpha_n})$ we have 
\begin{align*}
     &\underset{\mathscr{U}}{\operatorname{sup}}\,\, \mathbb{P} \left( \underset{k}{\operatorname{max}}\,\left(|\widehat{\lambda}_{n,k}-\lambda_k|\vee\|\widehat{f}_{n,k}-f_k\|_{1}\right)>\varepsilon_n\right)
     \leq \underset{\mathscr{U}}{\operatorname{sup}}\,\, \mathbb{P} \left(\|\widehat{p}_{n}-p\|_1>b_n\right)\leq b_n \xrightarrow{n\rightarrow \infty} 0.
\end{align*}

\qed

\subsection{Proofs for Section \ref{sec:gen}}\label{sec:proof for section 6}

\subsubsection{Proof of Lemma \ref{lemma:identifiability K translation}}

Define an order $\prec$ on $\{f_k\}_{k=1}^K$ by 
\begin{align*}
    f_j\prec f_k \quad \Longleftrightarrow \quad \underset{t\rightarrow \infty}{\operatorname{lim}}\, \frac{\varphi_j(t)}{\varphi_k(t)}=0.
\end{align*}
By the assumption we see that either $f_j\prec f_k$ or $f_k\prec f_j$ for $j\neq k$ and so $\prec$ is a total ordering on $\{f_k\}_{k=1}^K$. Therefore we can assume without loss of generality that the $f_k$'s are ordered with respect to $\prec$. 
Now taking the characteristic function of \eqref{eq:identifiability K translation} we have 
\begin{align*}
    \sum_{k=1}^K \lambda_k (e^{i\mu_kt}-e^{im_kt})\varphi_k(t)=0, \quad \quad \forall t\in\mathbb{R}.
\end{align*}
Dividing the above equation by $\varphi_K$ and setting $t\rightarrow \infty$, we obtain as a result of the ordering $\prec$ 
\begin{align*}
    \underset{t\rightarrow\infty}{\operatorname{lim}}\, \lambda_K (e^{i\mu_Kt}-e^{im_Kt})=0,
\end{align*}
which implies $\mu_K=m_K$. The result then follows by induction. 
\qed

\subsubsection{Proof of Proposition \ref{prop:pointwise consistent regression function}}\label{appen:proof of pointwise consistency}

Fix $x\in[a,b]$. Notice that the vectors $\widehat{\theta}_n(x)=(\widehat{m}_{n,1}(x),\ldots,\widehat{m}_{n,K}(x))^T$ and $\theta^{\ast}(x)=(m_1(x),\ldots,m_K(x))^T$ are minimizers of the following functionals respectively,
\begin{align*}
    \widehat{T}_{n,x}(\theta)&=\left\|\sum_{k=1}^K\widehat{\lambda}_{n,k}\widehat{f}_{n,k}(\cdot-\theta_k)-\widehat{p}_{n}(\,\cdot\given x)\right\|_1\\
    T_x(\theta)&=\left\|\sum_{k=1}^K\lambda_kf_k(\cdot-\theta_k)-p(\,\cdot\given x)\right\|_1.
\end{align*}
We need the following lemma (proved right after) for consistency of minimum distance estimators \cite[see also][Theorem 3]{beran1977minimum}.
\begin{lemma}\label{lemma:pointwise consistency convergence}
Fix $x\in\R$. Suppose 
\begin{enumerate}
\item $\|\widehat{p}_n(\,\cdot\given x)-p(\,\cdot\given x)\|_1\xrightarrow{n\rightarrow\infty} 0$
\item $\operatorname{max}_k(|\widehat{\lambda}_{n,k}-\lambda_k|\vee \|\widehat{f}_{n,k}-f_k\|_1)\xrightarrow{n\rightarrow\infty} 0 $
\end{enumerate}
Then any convergent subsequence of $\{\widehat{\theta}_n(x)\}_{n=1}^{\infty}$ converges to a minimizer of $T_x$. In particular if $T_x$ has a unique minimizer $\theta^{\ast}(x)$, then $\{\widehat{\theta}_n(x)\}_{n=1}^{\infty}$ converges to $\theta^{*}(x)$.
\end{lemma}

Notice that the assumptions of Lemma \ref{lemma:pointwise consistency convergence} are satisfied almost surely for the estimators in \eqref{eq:generalization consistency 1} and \eqref{eq:generalization consistency 2} when $x\in(a,b)$. This together with Lemma \ref{lemma:identifiability K translation}, which implies that the minimizers of $T_x$ are unique, gives 
\begin{align*}
    \underset{k}{\operatorname{max}}\, |\widehat{m}_{n,k}(x)-m_k(x)|\xrightarrow{n\rightarrow\infty} 0 \quad \quad \forall x\in(a,b)
\end{align*}
with probability one. Therefore dominated convergence theorem implies that (since the functions $\widehat{m}_{n,k}$'s are bounded)
\begin{align*}
    \underset{k}{\operatorname{max}}\, \|\widehat{m}_{n,k}-m_k\|_{L^1[a,b]}&\xrightarrow{n\rightarrow\infty} 0.
\end{align*}
\qed

\begin{proof}[Proof of Lemma \ref{lemma:pointwise consistency convergence}]

Denote 
\begin{align*}
    g_{\theta}=\sum_{k=1}^K \lambda_kf_k(\cdot-\theta_k)\quad \quad \widehat{g}_{n,\theta}=\sum_{k=1}^K \widehat{\lambda}_{n,k}\widehat{f}_{n,k}(\cdot-\theta_k).
\end{align*}
We first show continuity of $T_x$ in $\theta$, which then guarantees the existence of a minimizer in $\Theta$. Let $\theta_n \rightarrow \theta$. Then 
\begin{align*}
    |T_x(\theta_n)-T_x(\theta)|\leq \|g_{\theta_n}-g_{\theta}\|_1=2\int_{g_{\theta}\geq g_{\theta_n}} g_{\theta}-g_{\theta_n}=2\int (g_{\theta}-g_{\theta_n})\mathbf{1}_{g_{\theta}\geq g_{\theta_n}}.
\end{align*}
Since $f$ is continuous, we have $g_{\theta_n}\rightarrow g_{\theta}$ pointwise and the last quantity converges to zero by dominated convergence theorem. Therefore $T_x$ is continuous in $\theta$. Similarly we can show continuity of $\widehat{T}_{n,x}$ since each $\widehat{f}_{n,k}$ is continuous, which guarantees the existence of a minimizer as defined in \eqref{eq:MDE2}. 

Now we claim that 
\begin{align}
    \underset{\theta\in\Theta}{\operatorname{sup}}\, |\widehat{T}_{n,x}(\theta)-T_x(\theta)| \xrightarrow{n\rightarrow\infty} 0. \label{eq:h_n h2}
\end{align}
This will then imply  
\begin{align} \label{eq:min h_n& min h}
    |\underset{\theta \in \Theta}{\operatorname{min}}\,\widehat{T}_{n,x}(\theta) - \underset{\theta\in\Theta}{\operatorname{min}}\,T_x(\theta)| \xrightarrow{n\rightarrow\infty} 0. 
\end{align}
Since $\Theta$ is compact, $\{\widehat{\theta}_n\}_{n=1}^{\infty}:=\{\widehat{\theta}_n(x)\}_{n=1}^{\infty}$ has a convergence subsequence, still denoted as $\{\widehat{\theta}_n\}_{n=1}^{\infty}$, that converges to a point $\theta_0 \in \Theta$. By continuity of $T_x$, we have 
\begin{align*}
    |T_x(\widehat{\theta}_n)-T_x(\theta_0)|\xrightarrow{n\rightarrow\infty} 0,
\end{align*}
which together \eqref{eq:h_n h2} implies
\begin{align*}
    |\widehat{T}_{n,x}(\widehat{\theta}_n)-T_x(\theta_0)|\xrightarrow{n\rightarrow\infty} 0,
\end{align*}
which together with \eqref{eq:min h_n& min h} further implies \begin{align*}
    T_x(\theta_0)=\underset{\theta\in\Theta}{\operatorname{min}}\, T_x(\theta),
\end{align*}
i.e., $\theta_0$ is a minimizer of $T_x$. Now if $T_x$ has a unique minimizer $\theta^{\ast}$ then any convergent subsequence of $\{\hat{\theta}_n(x)\}_{n=1}^{\infty}$ converges to $\theta^*$, which implies convergence of the whole sequence to $\theta^*$.  

Now we show the claim \eqref{eq:h_n h2}, by considering  an intermediate quantity 
\begin{align*}
    T_{n,x}(\theta)=\|g_{\theta}-\widehat{p}_{n}(\,\cdot\given x)\|_1,
\end{align*}
and bound 
\begin{align*}
    \underset{\theta\in\Theta}{\operatorname{sup}}\, |\widehat{T}_{n,x}(\theta)-T_x(\theta)|\leq \underset{\theta\in\Theta}{\operatorname{sup}}\,|\widehat{T}_{n,x}(\theta)-T_{n,x}(\theta)|+\underset{\theta\in\Theta}{\operatorname{sup}}\, |T_{n,x}(\theta)-T_x(\theta)|.
\end{align*}
For the second term, notice that 
\begin{align}
    |T_{n,x}(\theta)-T_x(\theta)|\leq \|\widehat{p}_{n}(\,\cdot\given x)-p(\,\cdot\given x)\|_1 \xrightarrow{n\rightarrow\infty} 0, \label{eq:D2}
\end{align}
by assumption on $\widehat{p}_{n}(\,\cdot\given x)$. 
To bound the first term, we have 
\begin{align*}
    |\widehat{T}_{n,x}(\theta)-T_{n,x}(\theta)| &\leq \|\widehat{g}_{n,\theta}-g_{\theta}\|_1 \leq \|\widehat{g}_{n,\theta}-\widetilde{g}_{n,\theta}\|_1+\|\widetilde{g}_{n,\theta}-g_{\theta}\|_1,
\end{align*}
where
\begin{align*}
    \widetilde{g}_{n,\theta}=\sum_{k=1}^K \widehat{\lambda}_{n,k}f_k(\cdot-\theta_k).
\end{align*}
The claim then follows from the observations that 
\begin{align*}
    \|\widehat{g}_{n,\theta}-\widetilde{g}_{n,\theta}\|_1&=\left\|\sum_{k=1}^K \widehat{\lambda}_{n,k}\widehat{f}_{n,k}(\cdot-\theta_k)-\sum_{k=1}^K\widehat{\lambda}_{n,k}f_k(\cdot-\theta_k)\right\|_1\leq \underset{k}{\operatorname{max}}\,\|\widehat{f}_{n,k}-f_k\|_1\xrightarrow{n\rightarrow\infty} 0,
\end{align*}
and 
\begin{align*}
    \|\widetilde{g}_{n,\theta}-g_{\theta}\|_1=\left\|\sum_{k=1}^K\widehat{\lambda}_{n,k}f_k(\cdot-\theta_k)-\sum_{k=1}^K\lambda_kf_k(\cdot-\theta_k)\right\|_1 \leq \sum_{k=1}^K |\widehat{\lambda}_{n,k}-\lambda_k|\xrightarrow{n\rightarrow\infty} 0.
\end{align*}
\end{proof}

\subsubsection{Proof of Theorem \ref{thm:identifiability differnt x0}}\label{appen:different sep}

Suppose we are given two models in  $\Phi(\cup_{x_0\in[a,b]}\mor_{x_0})$ with 
\begin{align*}
    \Phi(\mor_{x_0}) \ni p_X(x)\sum_{k=1}^K \lambda_k f(y-m_k(x))=
    q_X(x)\sum_{k=1}^K \pi_kg(y-\mu_k(x))\in\Phi(\mor_{x_1}).
\end{align*}
Then we need to show that $(p_X,f,\{\lambda_k\}_{k=1}^K, \{m_k\}_{k=1}^K)=(q_X,g,\{\pi_k\}_{k=1}^K, \{\mu_k\}_{k=1}^K)$.
By integrating the above equation with respect to $y$ we see that $p_X=q_X$. Since $p_X=q_X>0$ over $[a,b]$ by assumption, the above equation reduces to 
\begin{align*}
    \sum_{k=1}^K \lambda_k f(y-m_k(x))= \sum_{k=1}^K \pi_kg(y-\mu_k(x)) \quad \quad \forall \,\,(x,y)\in[a,b]\times \mathbb{R}.  
\end{align*}
Now we shall first prove that $(f,\{\lambda_k\}_{k=1}^K)=(g,\{\pi_k\}_{k=1}^K)$ by exploiting the points of separation. Let $G$ and $H$ be the such that $f=\phi_{\sigma}\ast G$ and $g=\phi_{\sigma}\ast H$. The above equation, in terms of mixing measures, implies that 
\begin{align}
    \sum_{k=1}^K \lambda_k G(\cdot-m_k(x))=\sum_{k=1}^K \pi_k H(\cdot-\mu_k(x)), \label{eq:identifiability generalization 1}
\end{align}
and in particular their supports are equal, i.e., 
\begin{align}
    \bigcup_{k=1}^K E_k(x)= \bigcup_{k=1}^K F_k(x),  \label{eq:identifiability generalization 2}
\end{align}
where $E_k(x)=\operatorname{supp}(G(\cdot-m_k(x)))$ and $F_k(x)=\operatorname{supp}(H(\cdot-\mu_k(x)))$. Since $x_0$ is a point of separation of the first model, we have 
\begin{align}
    \underset{j\neq k}{\operatorname{min}}\,\operatorname{dist}(E_j(x_0),E_k(x_0)) > \operatorname{diam}(\operatorname{supp}(G))\label{eq:identifiability generalization 4}
\end{align}
and similarly for the second model
\begin{align*}
    \underset{j\neq k}{\operatorname{min}}\,\operatorname{dist}(F_j(x_1),F_k(x_1)) > \operatorname{diam}(\operatorname{supp}(H)).
\end{align*}
Notice that the difference with the case where $x_0=x_1$ is that we cannot say anything about $F_k(x_0)$'s or $E_k(x_1)$'s yet. However, we can make the following claim: it holds that either
\begin{enumerate}
    \item for any $k$, $F_k(x_0)\subset E_{\Pi_0(k)}(x_0)$ for some $\Pi_0(k)$, or 
    \item for any $k$, $E_k(x_1)\subset F_{\Pi_1(k)}(x_1)$ for some $\Pi_1(k)$. 
\end{enumerate}
To see why the claim is true, suppose neither item 1 nor 2 is true. Failure of item 1 implies that there exists $k^*$ such that both $F_{k^*}(x_0)\cap E_{i^*}(x_0)$ and $F_{k^*}(x_0)\cap E_{j^*}(x_0)$ are nonempty for some $i^*\neq j^*.$  This implies that 
\begin{align*}
    \operatorname{diam}(\operatorname{supp}(H))=\operatorname{diam}(F_{k^*}(x_0)) \geq \operatorname{dist}(E_{i^*}(x_0),E_{j^*}(x_0))> \operatorname{diam}(\operatorname{supp}(G))
\end{align*}
where we used \eqref{eq:identifiability generalization 4} in the last step. Similarly, failure of item 2 would then imply that $\operatorname{diam}(\operatorname{supp}(G))>\operatorname{diam}(\operatorname{supp}(H))$, which then leads to a contradiction. 

We note that either one of the above conditions would give that the error densities and mixing proportions are equal as in the proof of Proposition \ref{prop:identifiability convolutional}. For completeness we repeat the details here, by supposing for the moment that item 1 holds. Since the $E_k(x_0)$'s are disjoint, $\Pi_0$ cannot be multivalued. Moreover, $\Pi_0$ must be onto because otherwise  \eqref{eq:identifiability generalization 2} would be violated, which further implies that $\Pi_0$ is in fact bijective. Therefore by disjointness of the $E_k(x_0)$'s again, we conclude that $F_k(x_0)=E_{\Pi_0(k)}(x_0)$. Restricting \eqref{eq:identifiability generalization 1} to the common support $F_k(x_0)$, we get that $\pi_k=\lambda_{\Pi_0(k)}$ and 
\begin{align}
    H(\cdot-\mu_k(x_0))=G(\cdot-m_{\Pi_0(k)}(x_0)). \label{eq:identifiability generalization 3}. 
\end{align}
The assumption that the error density $f$ has mean zero, i.e., $\int_{\mathbb{R}}xf(x)dx=0$, translates to $\int_{\mathbb{R}}\theta dG(\theta)=0$. Therefore \eqref{eq:identifiability generalization 3} implies that $\mu_k(x_0)=m_{\Pi_0(x_0)}$ by computing first moments and hence $H=G$, which gives $g=f.$

Now it remains to show that the regression functions are equal and this part follows exactly the same argument in the proof of Theorem \ref{thm:identifiability convolutional} that is based on the identifiability of translation families. 
\qed

\section{Technical Lemmas}\label{sec:technical lemma}

\begin{lemma}[$L^1$ version of Lemma 3.1 in \cite{ghosal2001entropies}] \label{lemma:approx mixing measure}
Let $G$ be a mixing measure supported on $[-M,M]$. For $k\geq \left(\frac{9e}{2}\right)^6 \vee \big(\frac{2M}{\sigma}\big)^3$, there exists a discrete mixing measure $G_k$ supported on $[-M,M]$ with at most $2k-1$ atoms such that 
\begin{align*}
    \|\phi_{\sigma}\ast G-\phi_{\sigma}\ast G_k\|_1 \leq C\left[k^{-1/3}\exp\left(-\frac{k^{2/3}}{8}\right) +\exp\left(\left(-\frac{k}{3}+\frac13\right)\log k\right) \right]=:s_{k},
\end{align*}
where $C$ is a universal constant. As a result, if $Q_L=\operatorname{arg\,min}_{Q\in \mix_L} \|Q-\phi_{\sigma}\ast G\|_1$ then for all $L$ large
\begin{align*}
    \|Q_L-p\|_1 \leq C s_{\lfloor\frac{L+1}{2}\rfloor}.
\end{align*}
\end{lemma}

\begin{proof}
By \cite[][Lemma~A.1]{ghosal2001entropies}, since $G$ is compacted supported, there exists a discrete measure $G_k$ with at most $2k-1$ atoms so that 
\begin{align}
    \int \theta^{\ell} dG(\theta)=\int \theta^{\ell} dG_k(\theta) ,\quad \quad \ell=1,\ldots,2k-2. \label{eq:moment equality}
\end{align}
We have 
\begin{align*}
    \|\phi_{\sigma}\ast G-\phi_{\sigma}\ast G_k\|_1&=\int_{|x|\leq R_k} |\phi_{\sigma}\ast G-\phi_{\sigma}\ast G_k|dx+\int_{|x|>R_k} |\phi_{\sigma}\ast G-\phi_{\sigma}\ast G_k|dx\\
    &=:I_1+I_2, 
\end{align*}
where $R_k:=\sigma k^{1/3} \geq 2M$. 
We have 
\begin{align*}
    \phi_{\sigma}\ast G&=\int_{-M}^{M} \phi_{\sigma}(x-\theta)dG(\theta)\leq \phi_{\sigma}(x-M),\quad \quad x>R_k\\
    \phi_{\sigma}\ast G&=\int_{-M}^{M} \phi_{\sigma}(x-\theta)dG(\theta)\leq \phi_{\sigma}(x+M),\quad \quad x<-R_k
\end{align*}
and hence 
\begin{align*}
    \int_{|x|>R_k} \phi_{\sigma}\ast G \,dx&\leq \int_{R_k}^{\infty} \phi_{\sigma}(x-M)dx +\int_{-\infty}^{-R_k}\phi_{\sigma}(x+M)dx\\
    &= 2\int_{\frac{R_k-M}{\sigma}}^{\infty} \frac{1}{\sqrt{2\pi}}e^{-\frac{x^2}{2}}dx \leq \frac{2}{\sqrt{2\pi}} \frac{\sigma}{R_k-M}e^{-\frac{|R_k-M|^2}{2\sigma^2}},
\end{align*}
where we have used the bound $\int_{a}^{\infty}e^{-\frac{x^2}{2}}dx\leq a^{-1}e^{-\frac{a^2}{2}}$ in the last step. Since $G_k$ is also supported on $[-M,M]$, we have similarly
\begin{align*}
    \int_{|x|>R_k} \phi_{\sigma}\ast G_k \,dx\leq \frac{2}{\sqrt{2\pi}} \frac{\sigma}{R_k-M}e^{-\frac{|R_k-M|^2}{2\sigma^2}}
\end{align*}
and hence 
\begin{align}
    \int_{|x|>R_k}|\phi_{\sigma}\ast G-\phi_{\sigma}\ast G_k|dx\leq \frac{4\sigma}{\sqrt{2\pi}} \frac{e^{-\frac{|R_k-M|^2}{2\sigma^2}}}{R_k-M} \leq C k^{-1/3}\exp\left(-\frac{k^{2/3}}{8}\right), \label{eq:bound |x|>M}
\end{align}
where we have used $R_k-M\geq R_k/2$ in the last step 
To bound $I_1$, we first approximate $\phi_{\sigma}$ using Taylor expansion. For $y<0$ and using that $k!\geq k^ke^{-k}$, we have 
\begin{align*}
    \left|e^y-\sum_{j=0}^{k-1} \frac{y^j}{j!}\right|\leq \frac{|y|^k}{k!}\leq \frac{(e|y|)^k}{k^k}.
\end{align*}
Setting $y=-\frac{x^2}{2\sigma^2}$, we have 
\begin{align*}
    \left|\phi_{\sigma}(x)-\frac{1}{\sqrt{2\pi\sigma^2}}\sum_{j=0}^{k-1}\frac{(-2\sigma^{-2}x^2)^j}{j!}\right|\leq \frac{(e2\sigma^{-2}x^2)^k}{\sqrt{2\pi\sigma^2}k^k}.
\end{align*}
Now for $|x|\leq R_k$, 
\begin{align*}
    |\phi_{\sigma}\ast G-\phi_{\sigma}\ast G_k|&\leq \left|\int \sum_{j=0}^{k-1}\frac{1}{\sqrt{2\pi\sigma^2}} \frac{(-2\sigma^{-2}|x-\theta|^2)^j}{j!}d(G-G_k)(\theta)
    \right| \\
    &\quad + 2\underset{|\theta|\leq M}{\operatorname{sup}} \left|\phi_{\sigma}(x-\theta)-\sum_{j=0}^{k-1}\frac{1}{\sqrt{2\pi\sigma^2}} \frac{(-2\sigma^{-2}|x-\theta|^2)^j}{j!}\right|\\
    &\leq\left|\int \sum_{j=0}^{k-1} \frac{(-2\sigma^{-2})^{j}}{\sqrt{2\pi\sigma^2}j!}  \sum_{\ell=0}^{2j} {2j \choose \ell}\theta^{\ell}x^{2j-\ell} d(G-G_k)(\theta)\right|\\
    &\quad + 2\underset{|\theta|\leq M}{\operatorname{sup}} \frac{(2e\sigma^{-2}|x-\theta|^2)^k}{\sqrt{2\pi\sigma^2}k^k}.
\end{align*}
The first term in the last display equals zero by \eqref{eq:moment equality}. Therefore we have 
\begin{align}
    \int_{|x|\leq R_k} |\phi_{\sigma}\ast G-\phi_{\sigma}\ast G_k|&\leq 2R_k  \underset{|\theta|\leq M, |x|\leq R_k}{\operatorname{sup}}\frac{(2e\sigma^{-2}|x-\theta|^2)^k}{\sqrt{2\pi\sigma^2}k^k}\nonumber \\
    &\leq 2R_k \frac{(2e\sigma^{-2}(R_k+M)^2)^k}{\sqrt{2\pi\sigma^2}k^k}\leq C\exp\left(\left(-\frac{k}{3}+\frac13\right)\log k\right), \label{eq:bound |x|<M}
\end{align}
where we have used $R_k+M\leq 3R_k/2$ in the last step. 
Now combining \eqref{eq:bound |x|>M} and \eqref{eq:bound |x|<M}, we have 
\begin{align*}
    \|\phi_{\sigma}\ast G-\phi_{\sigma}\ast G_k\|_1 \leq C\left[k^{-1/3}\exp\left(-\frac{k^{2/3}}{8}\right) +\exp\left(\left(-\frac{k}{3}+\frac13\right)\log k\right) \right],
\end{align*}
with $C$ a universal constant, provided that $k\geq \left(\frac{9e}{2}\right)^6 \vee \left(\frac{2a}{\sigma}\right)^3$.
\end{proof}

\begin{lemma}\label{lemma:moc}
Suppose $f=\phi_{\sigma}\ast G_0$ is a density satisfying $|(\mathcal{F}G_0)(w)|\geq J(w)>0$ for some function $J$. Suppose $\widehat{V}$ and $V$ are two mixing measures whose supports are contained in $[-B,B]$ for some $B>0$. Then 
\begin{align*}
    W_1(\widehat{V},V)\leq C_B F_{J,\sigma}(\|\widehat{V}\ast f-V\ast f\|_1),
\end{align*}
where $C_B$ is a constant depending only on $B$ and $F_{J,\sigma}$ is a strictly increasing function that depends only on $J,\sigma$ and satisfies $F_{J,\sigma}(d)\xrightarrow{d\rightarrow 0}0$.
\end{lemma}
\begin{proof}
The proof is similar to \cite[][Theorem~2]{nguyen2013}. We include the details for completeness.
Let $H$ be a density function with bounded first moment whose Fourier transform is supported in $[-1,1]$ and $H_{\delta}=\delta^{-1}H(\delta^{-1}\cdot)$. We have 
\begin{align*}
    W_1(\widehat{V},V)\leq W_1(\widehat{V},\widehat{V}\ast H_{\delta})+ W_1(\widehat{V}\ast H_{\delta},V\ast H_{\delta})+W_1(V\ast H_{\delta},V),
\end{align*}
where $W_1(\widehat{V},\widehat{V}\ast H_{\delta})$ and $W_1(V\ast H_{\delta},V)$ are both bounded by $C\delta$. By \cite[][Theorem~6.15]{villani2008optimal} we have
\begin{align*}
    W_1(\widehat{V}\ast H_{\delta},V\ast H_{\delta})&\leq \int_{\mathbb{R}} |\theta||\widehat{V}\ast H_{\delta}(\theta)-V\ast H_{\delta}(\theta)|d\theta\\
    &=\int_{|\theta|\leq B+1}|\theta| |\widehat{V}\ast H_{\delta}(\theta)-V\ast H_{\delta}(\theta)|d\theta \\
    &\quad + \int_{|\theta|>B+1}|\theta||\widehat{V}\ast H_{\delta}(\theta)-V\ast H_{\delta}(\theta)|d\theta.
\end{align*}
The second term can be bounded by noticing that 
\begin{align*}
    \int_{|\theta|>B+1} |\theta|V\ast H_{\delta}(\theta)d\theta
    &=\int_{|z|\leq M} \int_{|\theta|>B+1} |\theta| H_{\delta}(\theta-z)d\theta dV(z)\\
    &\leq \int_{|z|\leq B} \int_{|\theta+z|>B+1} |\theta| H_{\delta}(\theta)d\theta dV(z)\\
    &\quad+\int_{|z|\leq B} \int_{|\theta+z|>B+1} |z| H_{\delta}(\theta)d\theta dV(z)\\
    &\leq \int_{\mathbb{R}}|\theta|H_{\delta}(\theta)d\theta+\int_{|z|\leq B}|z|dV(z)\int_{|\theta|>1}H_{\delta}(\theta)d\theta\\
    &\leq C\delta+B\int_{|\theta|>1/\delta} \delta|\theta| H(\theta)d\theta  \leq C_B \delta
\end{align*}
and similarly for $\int_{|\theta|>B+1} \widehat{V}\ast H_{\delta}(\theta)d\theta$. The first term can be bounded using Cauchy-Schwarz by 
\begin{align*}
    \sqrt{\int_{|\theta|\leq B+1}|\theta|^2 d\theta \int_{|\theta|\leq B+1}|\widehat{V}\ast H_{\delta}(\theta)-V\ast H_{\delta}(\theta)|^2d\theta}\leq C_B\|\widehat{V}\ast H_{\delta}-V\ast H_{\delta}\|_2 
\end{align*}
Now letting
$h_{\delta}=\mathcal{F}^{-1}(\mathcal{F}H_{\delta}/\mathcal{F}f)$ (since $\mathcal{F}H_{\delta}$ is continuous and compactly supported, and $\mathcal{F}f$ is never zero, $\mathcal{F}H_{\delta}/\mathcal{F}f\in L^1$ and $h_{\delta}$ is well-defined), we have $H_{\delta}=f\ast h_{\delta}$ and then
\begin{align*}
    \widehat{V}\ast H_{\delta}&=(\widehat{V}\ast f)\ast h_{\delta}\\
    V\ast H_{\delta}&=(V\ast f)\ast h_{\delta}.
\end{align*}
Thus by Young's inequality we have
\begin{align*}
    \|\widehat{V}\ast H_{\delta}-V\ast H_{\delta}\|_2\leq \|\widehat{V}\ast f-V\ast f\|_1 \|h_{\delta}\|_2
\end{align*}
and by Plancherel's identity
\begin{align*}
    \|h_{\delta}\|^2_2=\left\|\frac{\mathcal{F}H_{\delta}}{\mathcal{F}f}\right\|_2^2\leq C \int_{|w|<1/\delta} (\mathcal{F}f)(w)^{-2}dw\leq C\int_{|w|<1/\delta} \exp(\sigma^2w^2)J(w)^{-2}dw,
\end{align*}
where we have used the fact that $\mathcal{F}H$ is supported on $[-1,1]$, which implies that $\mathcal{F}H_{\delta}$ is supported on $[-1/\delta,1/\delta]$. 
Therefore we have 
\begin{align*}
    W_1(\widehat{V},V) \leq C_B \left[\delta+\|\widehat{V}\ast f-V\ast f\|_1 \int_{|w|<1/\delta} \exp(\sigma^2w^2)J(w)^{-2}dw\right].
\end{align*}
Now define 
\begin{align*}
    A_{J,\sigma}(\delta)=\frac{\delta}{\int_{|w|<1/\delta} \exp(\sigma^2w^2)J(w)^{-2}dw}
\end{align*}
and notice that $A_{J,\sigma}$ is continuous, strictly increasing, $A_{J,\sigma}(\delta)\xrightarrow{\delta\rightarrow0}0$ and $A_{J,\sigma}(\delta)\xrightarrow{\delta\rightarrow\infty}\infty$. Therefore $A_{J,\sigma}$ is invertible with range $[0,\infty)$ so that we can set
$\delta=A_{J,\sigma}^{-1}(\|\widehat{V}\ast f-V\ast f\|_1)$ and get 
\begin{align*}
    W_1(\widehat{V},V)\leq C_B A_{J,\sigma}^{-1}(\|\widehat{V}\ast f-V\ast f\|_1)=:C_BF_{J,\sigma}(\|\widehat{V}\ast f-V\ast f\|_1).
\end{align*}
The fact that $F_{J,\sigma}$ is also strictly increasing and $F_{J,\sigma}(d)\xrightarrow{d\rightarrow 0}0$ is immediate. 
\end{proof}

\begin{lemma}\label{lemma:uniform transversality}
The conditions \eqref{eq:uniform transversality assumption} implies \ref{assumption:unpmor:sep of reg}.
\end{lemma}
\begin{proof}
Let $a_n$ be a sequence converging to zero. Notice that 
\begin{align*}
    \left\{x:\exists j\neq k ,\,  |m_j(x)-m_k(x)|\leq a_n \right\}=\bigcup_{j\neq k} \left\{x:|m_j(x)-m_k(x)|\leq a_n\right\}
\end{align*}
and it suffices to show $\left|\{x:|m_j(x)-m_k(x)|\leq a_n\}\right|\xrightarrow{n\rightarrow\infty} 0$ for all $j\neq k$. By the second assumption in \eqref{eq:uniform transversality assumption} we see that $\{x:|m_j(x)-m_k(x)|\leq a_n\}\subset Z_{jk}(\delta)$ when $n$ is large enough so that $a_n<\eta$. By the first assumption in \eqref{eq:uniform transversality assumption} we can write 
\begin{align*}
    Z_{jk}(\delta)=\bigcup_{i=1}^{N_{jk}} (x_i-\delta,x_i+\delta)
\end{align*}
with $Z_{jk}=\{x_i\}_{i=1}^{N_{jk}}$ and 
\begin{align}
    \left\{x:|m_j(x)-m_k(x)|\leq a_n\right\} &= \left\{x:|m_j(x)-m_k(x)|\leq a_n\right\} \cap Z_{jk}(\delta)\nonumber\\
   &=\bigcup_{i=1}^{N_{jk}}\left\{x:|m_j(x)-m_k(x)|\leq a_n\right\}\cap (x_i-\delta,x_i+\delta). \label{eq:universal transversality set decomp}
\end{align}
Since $N_{jk}\leq N$, it is enough to show that each component of \eqref{eq:universal transversality set decomp} goes to zero. Denoting $\Delta_{jk}(x)=m_j(x)-m_k(x)$, for any $y,z\in \{x:|m_j(x)-m_k(x)|\leq a_n\}\cap (x_i-\delta,x_i+\delta)$ we have 
\begin{align*}
    2a_n \geq |\Delta_{jk}(y)-\Delta_{jk}(z)|  = |\Delta_{jk}^{\prime}(\zeta)||y-z| \geq \eta |y-z|, 
\end{align*}
where we have used the fact that $\zeta$ is a point between $y$ and $z$ so that $\zeta\in(x_i-\delta,x_i+\delta)$ and $|\Delta_{jk}^{\prime}(\zeta)|\geq\eta$ by the third assumption in \eqref{eq:uniform transversality assumption}. Therefore we have 
\begin{align*}
    \left|\left\{x:|m_j(x)-m_k(x)|\leq a_n\right\}\cap (x_i-\delta,x_i+\delta)\right| \leq 2\eta^{-1}a_n \xrightarrow{n\rightarrow \infty} 0,
\end{align*}
which concludes the proof. 
\end{proof}

\section{Consistency of Density Estimators}\label{sec:density estimator}

\subsection{$L^1$ consistency of KDE for Sections \ref{sec:step 1} and \ref{sec:gen}}\label{sec:consistency cde}
In this section we construct a conditional density estimator that suffices for the purposes of Sections \ref{sec:step 1} and \ref{sec:gen}.
Let $\{(X_i,Y_i)\}_{i=1}^n$ be i.i.d. samples from the joint density
\begin{align*}
    p(x,y)=p_X(x)\sum_{k=1}^K \lambda_kf(y-m_k(x)), \quad \quad (x,y)\in[a,b]\times \mathbb{R}.
\end{align*}
Let $H=\frac12\mathbf{1}_{[-1,1]}$ (in general any kernel with bounded $p$-variation) and consider the kernel density estimator
\begin{align*}
    \widehat{p}_n(x,y)=\frac{1}{nh_n^2}\sum_{i=1}^n H\left(\frac{x-X_i}{h_n}\right)H\left(\frac{y-Y_i}{h_n}\right)
\end{align*}
for a suitable sequence $h_n\rightarrow 0$ to be determined. 
The marginal density of $\widehat{p}_n(x,y)$ is 
\begin{align*}
    \widehat{p}_{X,n}(x)=\frac{1}{nh_n}\sum_{i=1}^n H\left(\frac{x-X_i}{h_n}\right).
\end{align*}
and we define our conditional density estimator as
\begin{align}
   \widehat{p}_n(y|x) = \frac{\widehat{p}_n(x,y)}{\widehat{p}_{X,n}(x)}. \label{eq:conditional density estimator}
\end{align}

\subsubsection{Pointwise consistency for KDE in Section \ref{sec:gen}} \label{sec:pointwise consistency CDE}
The following Lemma gives pointwise consistency of \eqref{eq:conditional density estimator} under minimal assumptions. 
\begin{lemma}\label{lemma:convergence of KDE}
Suppose the joint density $p(x,y)$ is $\beta-$H\"older   over its support   (e.g. when $f,$ $p_X$, and the $m_k$'s are all $\beta-$H\"older) and the marginal density $p_X$ is bounded below by a positive constant. Let  $h_n$ satisfy the following scaling 
\begin{align}
    h_n\rightarrow 0, \quad \frac{nh_n^2}{|\log h_n|}\rightarrow\infty, \quad \frac{|\log h_n|}{\log \log n}\rightarrow \infty, \quad h_n^2 \leq ch_{2n}^2 \label{eq:hn scaling}
\end{align}
for some $c>0$. 
Then with probability one, 
\begin{align*}
    \underset{  (x,y)\in[a+h_n,b-h_n]\times\mathbb{R} }{\operatorname{sup}}\, |\widehat{p}_n(y|x)-p(y|x)| &\xrightarrow{n\rightarrow \infty} 0 
    .
\end{align*}
  Consequently, for each $x\in(a,b)$,   $\|\widehat{p}_n(\,\cdot\given x)-p(\,\cdot\given x)\|_1\xrightarrow{n\rightarrow \infty} 0.$
\end{lemma}
\begin{proof}
We have 
\begin{align}
    |\widehat{p}_n(y|x)-p(y|x)| \leq \frac{1}{\widehat{p}_{X,n}(x)}|\widehat{p}_n(x,y)-p(x,y)|+p(x,y)\left|\frac{1}{\widehat{p}_{X,n}(x)}-\frac{1}{p_X(x)}\right|. \label{eq:cde decomposition}
\end{align}
Therefore it suffices to study  $\operatorname{sup}_{(x,y)\in[a,b]\times\mathbb{R}}|\widehat{p}_{n}(x,y)-p(x,y)|$ and $\operatorname{sup}_{x\in[a,b]}|\widehat{p}_{X,n}(x)-p_X(x)|$.
Notice that $\widehat{p}_n(x,y)$ and $\widehat{p}_{X,n}(x)$ are KDEs with respect to kernels $\frac14\mathbf{1}_{[-1,1]^2}$ and $\frac12\mathbf{1}_{[-1,1]}$, which are both of bounded variation. By \cite[][Theorem~2.3]{gine2002rates}, the scaling \eqref{eq:hn scaling} implies that with probability one
\begin{align*}
    \underset{(x,y)\in[a,b]\times\mathbb{R}}{\operatorname{sup}}\,|\widehat{p}_n(x,y)-\mathbb{E}\widehat{p}_n(x,y)| &\xrightarrow{n\rightarrow\infty} 0, \\
    \underset{x\in [a,b]}{\operatorname{sup}}\,|\widehat{p}_{X,n}(x)-\mathbb{E}\widehat{p}_{X,n}(x)| &\xrightarrow{n\rightarrow\infty} 0. 
\end{align*} 
(Notice that the second condition on $h_n$ also imples $\frac{nh_n}{|\log h_n|}\rightarrow\infty$, which guarantees the convergence in one dimension.)
For the biases, we have for $(x,y)\in [a+h_n,b-h_n]\times\mathbb{R}$
\begin{align}
    &|\mathbb{E}\widehat{p}_n(x,y)-p(x,y)|\nonumber\\
    &=\left|\frac{1}{h_n^2}\int_{-h_n}^{h_n}\int_{-h_n}^{h_n} H\left(\frac{X}{h_n}\right)H\left(\frac{Y}{h_n}\right)\left[p(x+X,y+Y)-p(x,y)\right]dXdY\right|\nonumber\\
    & \leq \left|\frac{1}{h_n^2}\int_{-h_n}^{h_n}\int_{-h_n}^{h_n} H\left(\frac{X}{h_n}\right)H\left(\frac{Y}{h_n}\right)C_p(X^2+Y^2)^{\frac{\beta}{2}}dXdY\right| \leq C_p h_n^{\beta} \label{eq:bias holder}
\end{align}
and similarly $|\mathbb{E}\widehat{p}_n(x)-p(x)|\leq C_p h_n^{\beta}$.
Therefore with probability one 
\begin{align*}
    \underset{(x,y)\in[a+h_n,b-h_n]\times\mathbb{R}}{\operatorname{sup}}\,|\widehat{p}_n(x,y)-p(x,y)| &\xrightarrow{n\rightarrow\infty} 0, \\
    \underset{x\in [a+h_n,b-h_n]}{\operatorname{sup}}\,|\widehat{p}_{X,n}(x)-p_X(x)| &\xrightarrow{n\rightarrow\infty} 0,
\end{align*} 
and since the marginal density $p$ is assumed to be bounded below by a positive constant, \eqref{eq:cde decomposition} implies that with probability one 
\begin{align*}
    \underset{(x,y)\in[a+h_n,b-h_n]\times\mathbb{R}}{\operatorname{sup}}\, |\widehat{p}_n(y|x)-p(y|x)| \xrightarrow{n\rightarrow \infty} 0. 
\end{align*}
The second assertion follows from   the fact that any $x\in(a,b)$ is eventually contained in all $[a+h_n,b-h_n]$,   and the general observation that if $g_n\rightarrow g$ pointwise with $g_n,g$ densities, then we have 
\begin{align*}
    \int |g_n-g| = 2\int_{g\geq g_n} g-g_n =2\int (g-g_n)\mathbf{1}_{g\geq g_n}\xrightarrow{n\rightarrow\infty} 0 
\end{align*}
by dominated convergence theorem. 
\end{proof}

\subsubsection{Uniform consistency for KDE in Section \ref{sec:step 1}}\label{sec:cde uniform}
In this section we address uniform consistency of \eqref{eq:conditional density estimator}. In particular we need to keep track of the dependence of the error $\|\widehat{p}_n(\,\cdot\given x)-p(\,\cdot\given x)\|_1$ on the model parameters and are essentially finding rates  of convergence. Therefore we need stronger regularity assumptions on the joint density $p$. 

\begin{lemma}\label{lemma:cde bound}
Suppose $f=\phi_{\sigma}\ast G_0$ with $\operatorname{supp}(G_0)\subset [-r,r]$. Suppose $p_X$ and the $m_k$'s are all bounded and have bounded derivatives   over their supports. Suppose further $p_X\geq p_{\operatorname{min}}>0$. Let $h_n$ satisfy the scaling \eqref{eq:hn scaling}. Then
\begin{align*}
    \underset{x\in[a+h_n,b-h_n]}{\operatorname{sup}}\, &\mathbb{E} \|\widehat{p}_n(\,\cdot\given x)-p(\,\cdot\given x)\|_1 
    \leq C_{\sigma}p_{\operatorname{min}}^{-1}\Bigg(h_n\Big[\|p_X^{\prime}\|_{\infty}+(1\vee\underset{k}{\operatorname{max}}\, \|m_k^{\prime}\|_{\infty}) \|p_X\|_{\infty} \Big]\\
    &+\frac{\sqrt{\|p_X\|_{\infty}}}{\sqrt{nh_n^2}}\left[1\vee\big( h_n+2r+\underset{k}{\operatorname{max}}\, \|m_k\|_{\infty}\big)\right]+ \sqrt{\frac{\|p_X\|_{\infty}|\log h_n|}{nh_n}}\Bigg),
\end{align*}
where $C_{\sigma}$ is a constant depending only on $\sigma$. 
\end{lemma}
\begin{proof}
\begin{align}
    \int_\mathbb{R} \left|\widehat{p}_n(y|x)-p(y|x)\right| dy&=\int_{\mathbb{R}} \left|\frac{\widehat{p}_n(x,y)}{\widehat{p}_{X,n}(x)}-\frac{p(x,y)}{p_X(x)}\right|dy\nonumber \\
    & \leq \int_{\mathbb{R}}\left|\frac{\widehat{p}_n(x,y)}{\widehat{p}_{X,n}(x)}-\frac{\widehat{p}_n(x,y)}{p_X(x)}\right| dy +\int_{\mathbb{R}} \left|\frac{\widehat{p}_n(x,y)}{p_X(x)}-\frac{p(x,y)}{p_X(x)}\right|dy\nonumber\\
    & =  \frac{1}{p_X(x)}\int_{\mathbb{R}}\left|\widehat{p}_n(x,y)-p(x,y)\right|dy+\frac{|\widehat{p}_{X,n}(x)-p_X(x)|}{p_X(x)} , \label{eq:cde decomposition 2}
\end{align}
where we have used the fact that $\int_{\mathbb{R}}\widehat{p}_n(x,y)dy=\widehat{p}_{X,n}(x)$. For the same choice of $h_n$ as in \eqref{eq:hn scaling}, by \cite[][Theorem~2.3]{gine2002rates} and a similar bias analysis as in \eqref{eq:bias holder} when $p$ is differentiable, we have 
\begin{align}
    \underset{x \in[a+h_n,b-h_n]}{\operatorname{sup}}\, |\widehat{p}_{X,n}(x)-p_X(x)| \leq C \sqrt{\frac{\|p_X\|_{\infty}|\log h_n|}{nh_n}}+\|p_X^{\prime}\|_{\infty}h_n, \label{eq:KDE sup 1d rate}
\end{align}
where $C$ is a universal constant. 
So it suffices to bound the term 
\begin{align}
    \int_{\mathbb{R}}\left|\widehat{p}_n(x,y)-p(x,y)\right| dy\leq \int_{\mathbb{R}}|\widehat{p}_n(x,y)-\mathbb{E}\widehat{p}_n(x,y)| dy+\int_{\mathbb{R}}\left|\mathbb{E}\widehat{p}_n(x,y)-p(x,y)\right|dy, \label{eq:cde bv tradeoff}
\end{align}
where the expectation is with respect to the joint distribution of $(X,Y)$. The two integrals $I_1$ and $I_2$ can be interpreted as the variance and bias.

\vspace{0.2cm}
\noindent{\textbf{Bias}}:
Fix $(x,y)\in[a+h_n,b-h_n]\times \mathbb{R}$. Recall that 
\begin{align*}
    \widehat{p}_n(x,y)=\frac{1}{nh_n^2}\sum_{i=1}^n H\left(\frac{x-X_i}{h_n}\right) H\left(\frac{y-Y_i}{h_n}\right), \quad \quad H=\frac12\mathbf{1}_{[-1,1]}, 
\end{align*}
which gives 
\begin{align*}
    \mathbb{E}\widehat{p}_n(x,y)&=\frac{1}{h_n^2}\int_{\mathbb{R}} \int_{\mathbb{R}} H\left(\frac{x-X}{h_n}\right)H\left(\frac{y-Y}{h_n}\right)p(X,Y) dXdY\\
    &=\frac{1}{h_n^2} \int_{-h_n}^{h_n}\int_{-h_n}^{h_n} H\left(\frac{X}{h_n}\right)H\left(\frac{Y}{h_n}\right)p(x+X,y+Y) dXdY.
\end{align*}
Since $p$ is differentiable, we have when $|X|\leq h_n$
\begin{align}
    p(x+X,y+Y)-p(x,y)&=\int_0^1 \nabla p\left((x,y)+t(X,Y)\right)^T [X\,\, Y] dt \nonumber  \\
    &=\int_0^1 \left[\frac{\partial p}{\partial x}\left(x+tX,y+tY\right) X +\frac{\partial p}{\partial y}\left(x+tX,y+tY\right) Y\right] dt, \label{eq:taylor}
\end{align}
where 
\begin{align}
    \frac{\partial p}{\partial x} &= p_X^{\prime}(x)\sum_{k=1}^K\lambda_k f(y-m_k(x))-p_X(x)\sum_{k=1}^K \lambda_k f^{\prime} (y-m_k(x))m_k^{\prime}(x)\label{eq:px}\\
    \frac{\partial p}{\partial y} &=p_X(x)\sum_{k=1}^K\lambda_kf^{\prime}(y-m_k(x)). \label{eq:py}
\end{align}
So by \eqref{eq:taylor} we have 
\begin{align*}
    &\left\|p(x+X,\cdot+Y)-p(x,\cdot)\right\|_1 \\
    &\leq \int_0^1  X\left\|\frac{\partial p}{\partial x}\left(x+tX,\,\cdot+tY\right)\right\|_1+Y\left\|\frac{\partial p}{\partial y}\left(x+tX,\,\cdot+tY\right)\right\|_1dt,
\end{align*}
where by \eqref{eq:px} 
\begin{align*}
    \left\|\frac{\partial p}{\partial x}\left(x+tX,\,\cdot+tY\right)\right\|_1&\leq \sum_{k=1}^K\lambda_k \Big[\|p_X^{\prime}\|_{\infty}  \|f(\cdot-m_k(x+tX))\|_1 \\
    &\quad+ \|p_X\|_{\infty} \|m_k^{\prime}\|_{\infty} \|f^{\prime}(\cdot-m_k(x+tX))\|_1\Big]\\
     & \leq \|p_X^{\prime}\|_{\infty}+ \underset{k}{\operatorname{max}}\, \|m_k^{\prime}\|_{\infty} \|p_X\|_{\infty} \|f^{\prime}\|_{1},  
\end{align*}
and by \eqref{eq:py}
\begin{align*}
    \left\|\frac{\partial p}{\partial y}\left(x+tX,\,\cdot+tY\right)\right\|_1 &\leq 
    \|p_X\|_{\infty}\sum_{k=1}^K \lambda_k \|f^{\prime}(\cdot-m_k(x+tX))\|_1
     \leq \|p_X\|_{\infty}\|f^{\prime}\|_{1}. 
\end{align*}
It then follows that 
\begin{align*}
    \left\|p(x+X,\cdot+Y)-p(x,\cdot)\right\|_1& \leq |X|\big(\|p_X^{\prime}\|_{\infty}+\underset{k}{\operatorname{max}}\, \|m_k^{\prime}\|_{\infty} \|p_X\|_{\infty} \|f^{\prime}\|_{1}\big) \\
    &\quad+ |Y|\|p_X\|_{\infty}\|f^{\prime}\|_{1}.
\end{align*}
Therefore
\begin{align*}
    &\int_{\mathbb{R}} \left|\mathbb{E}p_n(x,y)-p(x,y)\right|dy\\
    &\leq \int_{\mathbb{R}} \left[\frac{1}{h_n^2} \int_{-h_n}^{h_n}\int_{-h_n}^{h_n} H\left(\frac{X}{h_n}\right)H\left(\frac{Y}{h_n}\right)\left|p(x+X,y+Y)-p(x,y)\right| dXdY\right] dy \nonumber\\
    &=\frac{1}{h_n^2} \int_{-h_n}^{h_n}\int_{-h_n}^{h_n} H\left(\frac{X}{h_n}\right)H\left(\frac{Y}{h_n}\right) \left\|p(x+X,\cdot+Y)-p(x,\cdot)\right\|_1  dXdY\nonumber \\
    & \leq \frac{1}{h_n^2} \int_{-h_n}^{h_n}\int_{-h_n}^{h_n} H\left(\frac{X}{h_n}\right)H\left(\frac{Y}{h_n}\right)|X|\Big(\|p_X^{\prime}\|_{\infty}+ \underset{k}{\operatorname{max}}\, \|m_k^{\prime}\|_{\infty} \|p_X\|_{\infty} \|f^{\prime}\|_{1}\Big) dXdY\nonumber\\
    &\quad + \frac{1}{h_n^2} \int_{-h_n}^{h_n}\int_{-h_n}^{h_n} H\left(\frac{X}{h_n}\right)H\left(\frac{Y}{h_n}\right)|Y|\|p_X\|_{\infty}\|f^{\prime}\|_{1} dXdY \nonumber\\
    & \leq h_n\Big[\|p_X^{\prime}\|_{\infty}+(1\vee\underset{k}{\operatorname{max}}\, \|m_k^{\prime}\|_{\infty}) \|p_X\|_{\infty} \|f^{\prime}\|_{1}\Big].
\end{align*}
Notice that 
\begin{align*}
    f^{\prime}(x)=\int_{-r}^r \left(\frac{\theta-x}{\sigma^2}\right)\phi_{\sigma}(x-\theta)dG_0(\theta)
\end{align*}
and 
\begin{align*}
    \|f^{\prime}\|_1\leq \int_{-r}^r\int_{\mathbb{R}} \frac{|x|}{\sigma^2}\phi_{\sigma}(x) dx dG_0(\theta)\leq C_{\sigma}
\end{align*}
for some constant $C_{\sigma}$. So for $x\in[a+h_n,b-h_n]$ we have 
\begin{align}
    \int_{\mathbb{R}} \left|\mathbb{E}p_n(x,y)-p(x,y)\right|dy \leq C_{\sigma}h_n\Big[\|p_X^{\prime}\|_{\infty}+(1\vee\underset{k}{\operatorname{max}}\, \|m_k^{\prime}\|_{\infty}) \|p_X\|_{\infty} \Big] \label{eq:cde bias}
\end{align}

\vspace{0.2cm}
\noindent{\textbf{Variance}}:
Fix $x\in[a,b]$. Applying Cauchy-Schwarz to $\mathbb{E}|\widehat{p}_n(x,y)-\mathbb{E}\widehat{p}_n(x,y)|$ (where both expectations are with respect to the joint law), we have 
\begin{align*}
    \mathbb{E} \left[\int_{\mathbb{R}} |\widehat{p}_n(x,y)-\mathbb{E}\widehat{p}_n(x,y)| dy \right]
    &=\int_{\mathbb{R}} \mathbb{E} |\widehat{p}_n(x,y)-\mathbb{E}\widehat{p}_n(x,y)| dy \\
    &\leq \int_{\mathbb{R}} \sqrt{\mathbb{E}|\widehat{p}_n(x,y)-\mathbb{E}\widehat{p}_n(x,y)|^2 }dy,
\end{align*}
where we compute 
\begin{align*}
    &\mathbb{E}|\widehat{p}_n(x,y)-\mathbb{E}\widehat{p}_n(x,y)|^2\\
    &=\mathbb{E}[\widehat{p}_n(x,y)]^2-[\mathbb{E}\widehat{p}_n(x,y)]^2\\
    &=\mathbb{E}\left[\frac{1}{n^2h_n^4}\sum_{i,j}H\left(\frac{x-X_i}{h_n}\right)H\left(\frac{y-Y_i}{h_n}\right)H\left(\frac{x-X_j}{h_n}\right)H\left(\frac{y-Y_j}{h_n}\right)\right]\\
    &\quad - \frac{1}{h_n^4}\left[\int_{\mathbb{R}} \int_{\mathbb{R}} H\left(\frac{x-X}{h_n}\right)H\left(\frac{y-Y}{h_n}\right)p(X,Y) dXdY\right]^2\\
    &=\frac{1}{n^2h_n^4}\sum_{i=1}^n \int_{\mathbb{R}} \int_{\mathbb{R}} H^2\left(\frac{x-X}{h_n}\right)H^2\left(\frac{y-Y}{h_n}\right)p(X,Y)dXdY \\
    &\quad + \frac{n(n-1)}{n^2h_n^4}\left[\int_{\mathbb{R}} \int_{\mathbb{R}} H\left(\frac{x-X}{h_n}\right)H\left(\frac{y-Y}{h_n}\right)p(X,Y) dXdY\right]^2\\
    &\quad - \frac{1}{h_n^4}\left[\int_{\mathbb{R}} \int_{\mathbb{R}} H\left(\frac{x-X}{h_n}\right)H\left(\frac{y-Y}{h_n}\right)p(X,Y) dXdY\right]^2\\
    &\leq \frac{1}{nh_n^2} \int_{\mathbb{R}} \int_{\mathbb{R}} \frac{1}{h_n^2} H^2\left(\frac{x-X}{h_n}\right)H^2\left(\frac{y-Y}{h_n}\right)p(X,Y)dXdY.
\end{align*}
Denoting $\mathcal{H}_{h_n}$ as the product kernel $h_n^{-2}H^2(h_n^{-1}\cdot)H^2(h_n^{-1}\cdot)$, we have shown that 
\begin{align*}
    \mathbb{E} \left[\int_{\mathbb{R}} |\widehat{p}_n(x,y)-\mathbb{E}\widehat{p}_n(x,y)| dy \right]
    &\leq \frac{1}{\sqrt{nh_n^2}}  \int_{\mathbb{R}} \sqrt{(\mathcal{H}_{h_n}\ast p)(x,y)} dy.
\end{align*}
Recall that $f=\phi_{\sigma}\ast G_0$ with $\operatorname{supp}(G_0)\subset [-r,r]$. Then for $|y|>2r$ we have 
\begin{align*}
    f(y)\leq \frac{1}{\sqrt{2\pi\sigma^2}}\exp\left(-\frac{y^2}{8\sigma^2}\right), 
\end{align*}
and for $|y|>2r+\underset{k}{\operatorname{max}}\, |m_k(x)|$
\begin{align*}
    p(x,y)=p_X(x)\sum_{k=1}^K\lambda_k f(y-m_k(x)) \leq  C_{\sigma}p_X(x) \exp(-c_{\sigma}y^2). 
\end{align*}
Therefore for $y>h_n+2r+\underset{k}{\operatorname{max}}\, \|m_k\|_{\infty}$, we have 
\begin{align*}
    (\mathcal{H}_{h_n}\ast p)(x,y)&= \int_{-1}^1 \int_{-1}^1 H^2(X)H^2(Y)p(x+h_nX,y+h_nY) dXdY\\
    &\leq C_{\sigma}\|p_X\|_{\infty} \int_{-1}^1 \int_{-1}^1 H^2(X)H^2(Y) \exp\left(-c_{\sigma}(y+h_nY)^2\right)dXdY\\
    &\leq C_{\sigma} \|p_X\|_{\infty} \exp\left(-c_{\sigma}(y-h_n)^2\right).
\end{align*}
Therefore 
\begin{align*}
    \int^{\infty}_{h_n+2r+\underset{k}{\operatorname{max}}\, \|m_k\|_{\infty}} \sqrt{(\mathcal{H}_{h_n}\ast p)(x,y)} dy \leq C_{\sigma}\sqrt{\|p_X\|_{\infty}}
\end{align*}
for a possibly different $C_{\sigma}$. 
Similarly,
\begin{align*}
    \int_{-\infty}^{-h_n-2r-\underset{k}{\operatorname{max}}\, \|m_k\|_{\infty}}\sqrt{(\mathcal{H}_{h_n}\ast p)(x,y)} dy \leq C_{\sigma}\sqrt{\|p_X\|_{\infty}}.
\end{align*}
Meanwhile since $(\mathcal{H}_{h_n}\ast p)(x,y)\leq \|p_X\|_{\infty}\|f\|_{\infty}\leq C_{\sigma}\|p_X\|_{\infty}$, then
\begin{align*}
    \int_{-h_n-2r-\underset{k}{\operatorname{max}}\, \|m_k\|_{\infty}}^{h_n+2r+\underset{k}{\operatorname{max}}\, \|m_k\|_{\infty}}\sqrt{(\mathcal{H}_{h_n}\ast p)(x,y)} dy\leq C_{\sigma}\sqrt{\|p_X\|_{\infty}}\big( h_n+2r+\underset{k}{\operatorname{max}}\, \|m_k\|_{\infty}\big).
\end{align*}
Hence we have
\begin{align}
    \mathbb{E} \left[\int_{\mathbb{R}} |\widehat{p}_n(x,y)-\mathbb{E}\widehat{p}_n(x,y)| dy \right] 
    \leq \frac{C_{\sigma}\sqrt{\|p_X\|_{\infty}}}{\sqrt{nh_n^2}}\left[1\vee\big( h_n+2r+\underset{k}{\operatorname{max}}\, \|m_k\|_{\infty}\big)\right]. \label{eq:cde variance}
\end{align}
Plugging \eqref{eq:cde bias}, \eqref{eq:cde variance} in \eqref{eq:cde bv tradeoff}, we have for $x\in[a+h_n,b-h_n]$
\begin{align*}
    \mathbb{E} \int_{\mathbb{R}} \left|\widehat{p}_n(x,y)-p(x,y)\right|dy& \leq C_{\sigma}h_n\Big[\|p_X^{\prime}\|_{\infty}+(1\vee\underset{k}{\operatorname{max}}\, \|m_k^{\prime}\|_{\infty}) \|p_X\|_{\infty} \Big]\\
    &\quad+\frac{C_{\sigma}\sqrt{\|p_X\|_{\infty}}}{\sqrt{nh_n^2}}\left[1\vee\big( h_n+2r+\underset{k}{\operatorname{max}}\, \|m_k\|_{\infty}\big)\right],
\end{align*}
which together with \eqref{eq:KDE sup 1d rate}, \eqref{eq:cde decomposition 2} and the fact that $p_X\geq p_{\operatorname{min}}>0$ we get
\begin{align*}
    \underset{x\in[a+h_n,b-h_n]}{\operatorname{sup}}\, \mathbb{E} \|\widehat{p}_n(\,\cdot\given x)-p(\,\cdot\given x)\|_1 \leq C_{\sigma}h_np_{\operatorname{min}}^{-1}\Big[\|p_X^{\prime}\|_{\infty}+(1\vee\underset{k}{\operatorname{max}}\, \|m_k^{\prime}\|_{\infty}) \|p_X\|_{\infty} \Big]\\
    +\frac{C_{\sigma}p_{\operatorname{min}}^{-1}\sqrt{\|p_X\|_{\infty}}}{\sqrt{nh_n^2}}\left[1\vee\big( h_n+2r+\underset{k}{\operatorname{max}}\, \|m_k\|_{\infty}\big)\right]+ C p_{\operatorname{min}}^{-1}\sqrt{\frac{\|p_X\|_{\infty}|\log h_n|}{nh_n}}.
\end{align*}

\end{proof}

\subsection{Uniform $L^1$-consistency for KDE in Section \ref{sec:npmix}} \label{sec:kde uniform L1}
In this section we provide a specific density estimator that suffices for the purpose of Section \ref{sec:npmix}, i.e.,  $\operatorname{sup}_{\mathscr{U}}\mathbb{E}\|\widehat{p}_n-p\|_1\xrightarrow{n\rightarrow\infty}0$. To do so we need to keep track of all the parameter dependency when bounding $\mathbb{E}\|\widehat{p}_n-p\|_1$, as in Lemma \ref{lemma:kde L1 rate} below.  Let $\{Y_i\}_{i=1}^n$ be i.i.d. samples from the density $p(y)=\sum_{k=1}^K\lambda_kf_k(y-\mu_k)=\phi_{\sigma}\ast G$. Consider the kernel density estimator 
\begin{align*}
    \widehat{p}_n(y)=\frac{1}{nh_n}\sum_{i=1}^N H\left(\frac{y-Y_i}{h_n}\right), \quad \quad H=\frac12\mathbf{1}_{[-1,1]}.
\end{align*}

\begin{lemma}\label{lemma:kde L1 rate}
Suppose $\operatorname{supp}(G)\subset [-M,M]$. Let $h_n$ satisfy $h_n\rightarrow 0$ and $nh_n\rightarrow\infty$. Then 
\begin{align*}
    \mathbb{E}\|\widehat{p}_n-p\|_1\leq C_{\sigma}\left[h_n+ \frac{1\vee( h_n+2M)}{\sqrt{nh_n}}\right],
\end{align*}
where $C_{\sigma}$ is a constant depending only on $\sigma$. 
\end{lemma}
\begin{proof}
We have
\begin{align*}
    \mathbb{E}\|\widehat{p}_n-p\|_1 \leq \mathbb{E}\|\widehat{p}-\mathbb{E}\widehat{p}_n\|_1+ \|\mathbb{E}\widehat{p}_n-p\|_1
\end{align*}
and it suffices to bound the variance and bias terms respectively. 

\vspace{0.2cm}
\noindent{\textbf{Bias}}:
We have 
\begin{align*}
    \mathbb{E}\widehat{p}_n(y)=\frac{1}{h_n}\int_{\mathbb{R}} H\left(\frac{y-Y}{h_n}\right)p(Y)dY =\frac{1}{h_n}\int_{\mathbb{R}}H\left(\frac{Y}{h_n}\right)p(y+Y)dY
\end{align*}
and then 
\begin{align*}
    \|\mathbb{E}\widehat{p}_n-p\|_1 \leq \frac{1}{h_n}\int_{\mathbb{R}}H\left(\frac{Y}{h_n}\right)\|p(\cdot+Y)-p(\cdot)\|_1 dY.
\end{align*}
Notice that we have 
\begin{align*}
    p(y+Y)-p(y)=\int_{0}^1 p^{\prime}(y+tY)Ydt 
\end{align*}
and $\|p(\cdot+Y)-p(\cdot)\|_1 \leq \|p^{\prime}\|_1 Y.$ Since $f_k=\phi_{\sigma}\ast G_k$, we have $f_k^{\prime}=\phi_{\sigma}^{\prime}\ast G_k$ and $\|f_k^{\prime}\|_1\leq \|\phi_{\sigma}^{\prime}\|_1\leq C_{\sigma}$. Therefore $\|p^{\prime}\|_1 \leq \sum_{k=1}^K\lambda_k \|f_k^{\prime}\|_1 \leq C_{\sigma}$ and we get 
\begin{align}
    \|\mathbb{E}\widehat{p}_n-p\|_1\leq\frac{1}{h_n}\int_{\mathbb{R}}H\left(\frac{Y}{h_n}\right) C_{\sigma} Y dY \leq C_{\sigma}h_n. \label{eq:kde bias}
\end{align}

\vspace{0.2cm}
\noindent{\textbf{Variance}}:
Applying Cauchy-Schwarz to $\mathbb{E}|\widehat{p}_n(y)-\mathbb{E}\widehat{p}_n(y)|$, we have 
\begin{align*}
    \mathbb{E} \left[\int_{\mathbb{R}} |\widehat{p}_n(y)-\mathbb{E}\widehat{p}_n(y)| dy \right]=\int_{\mathbb{R}} \mathbb{E} |\widehat{p}_n(y)-\mathbb{E}\widehat{p}_n(y)| dy \leq \int_{\mathbb{R}} \sqrt{\mathbb{E}|\widehat{p}_n(y)-\mathbb{E}\widehat{p}_n(y)|^2 }dy,
\end{align*}
where we compute 
\begin{align*}
    &\mathbb{E}|\widehat{p}_n(y)-\mathbb{E}\widehat{p}_n(y)|^2\\
    &=\mathbb{E}[\widehat{p}_n(y)]^2-[\mathbb{E}\widehat{p}_n(y)]^2\\
    &=\mathbb{E}\left[\frac{1}{n^2h_n^2}\sum_{i,j}H\left(\frac{y-Y_i}{h_n}\right)H\left(\frac{y-Y_j}{h_n}\right)\right]
    - \frac{1}{h_n^2}\left[\int \int H\left(\frac{y-Y}{h_n}\right)p(Y) dY\right]^2\\
    &=\frac{1}{n^2h_n^2}\sum_{i=1}^n  \int_{\mathbb{R}} H^2\left(\frac{y-Y}{h_n}\right)p(Y)dY 
    + \frac{n(n-1)}{n^2h_n^2}\left[ \int_{\mathbb{R}} H\left(\frac{y-Y}{h_n}\right)p(Y) dY\right]^2\\
    &\quad - \frac{1}{h_n^2}\left[\int_{\mathbb{R}} H\left(\frac{y-Y}{h_n}\right)p(Y) dY\right]^2\\
    &\leq \frac{1}{nh_n}  \int_{\mathbb{R}} \frac{1}{h_n} H^2\left(\frac{y-Y}{h_n}\right)p(Y)dY.
\end{align*}
Denoting $\mathcal{H}_{h_n}$ as the product kernel $h_n^{-1}H^2(h_n^{-1}\cdot)$, we have shown that 
\begin{align*}
    \mathbb{E} \left[\int_{\mathbb{R}} |\widehat{p}_n(y)-\mathbb{E}\widehat{p}_n(y)| dy \right]
    &\leq \frac{1}{\sqrt{nh_n}}  \int_{\mathbb{R}} \sqrt{(\mathcal{H}_{h_n}\ast p)(y)} dy.
\end{align*}
Recall that $p=\phi_{\sigma}\ast G$ with $\operatorname{supp}(G)\subset [-M,M]$. Then we have
\begin{align*}
    p(y)\leq \frac{1}{\sqrt{2\pi\sigma^2}}\exp\left(-\frac{y^2}{8\sigma^2}\right), \quad \quad |y|>2M. 
\end{align*}
For $y>h_n+2M$, we have 
\begin{align*}
    (\mathcal{H}_{h_n}\ast p)(y)&= \int_{-1}^1 H^2(Y)p(y+h_nY) dY\\
    &\leq \int_{-1}^1 H^2(Y)C_{\sigma}\exp\left(-c_{\sigma}(y+h_nY)^2\right) dY 
    \leq C_{\sigma}  \exp\left(-c_{\sigma}(y-h_n)^2\right).
\end{align*}
Therefore 
\begin{align*}
    \int^{\infty}_{h_n+2M} \sqrt{(\mathcal{H}_{h_n}\ast p)(y)} dy \leq C_{\sigma}
\end{align*}
for a possibly different $C_{\sigma}$. 
Similarly,
\begin{align*}
    \int_{-\infty}^{-h_n-2M}\sqrt{(\mathcal{H}_{h_n}\ast p)(y)} dy \leq C_{\sigma}.
\end{align*}
Meanwhile since $(\mathcal{H}_{h_n}\ast p)(y)\leq C_{\sigma}$, then
\begin{align*}
    \int_{-h_n-2M}^{h_n+2M}\sqrt{(\mathcal{H}_{h_n}\ast p)(y)} dy\leq C_{\sigma}( h_n+2M).
\end{align*}
Hence we have
\begin{align}
    \mathbb{E} \left[\int_{\mathbb{R}} |\widehat{p}_n(y)-\mathbb{E}\widehat{p}_n(y)| dy \right] 
    \leq \frac{C_{\sigma}}{\sqrt{nh_n}}\left[1\vee( h_n+2M)\right]. \label{eq:kde variance}
\end{align}
The result then follows from \eqref{eq:kde bias} and \eqref{eq:kde variance}. 
\end{proof}

\end{document}